\documentclass[11pt]{article}
\usepackage{amsmath}
\usepackage{amsfonts}
\usepackage{color}
\usepackage{latexsym}
\usepackage{tablefootnote}
\usepackage{epsfig}
\usepackage{multirow}
\usepackage{float}
\usepackage{array}
\usepackage{amssymb,amsthm}

\newcommand*{\QEDA}{\hfill\hbox{\vrule width1.0ex height1.0ex}}

\usepackage{makecell,booktabs}
\usepackage[version=4]{mhchem}
\usepackage[strict]{changepage}

\usepackage{amsmath}
\usepackage{amsfonts}
\usepackage{latexsym}
\usepackage{amssymb}

\newtheorem{thm}{Theorem}[section]
\newtheorem{theorem}[thm]{Theorem}

\newtheorem{lemma}[thm]{Lemma}

\newtheorem{proposition}[thm]{Proposition}

\newtheorem{corollary}[thm]{Corollary}
\newtheorem{definition}[thm]{Definition}

\newtheorem{remark}[thm]{Remark}

\newcommand{\beq}{\begin{equation}}
\newcommand{\eeq}{\end{equation}}
\newcommand{\beqa}{\begin{eqnarray}}
\newcommand{\eeqa}{\end{eqnarray}}
\newcommand{\beqas}{\begin{eqnarray*}}
\newcommand{\eeqas}{\end{eqnarray*}}
\newcommand{\bi}{\begin{itemize}}
\newcommand{\ei}{\end{itemize}}

\newcommand{\vgap}{\vspace{.1in}}
\newcommand{\nn}{\nonumber}

\newcommand{\R}{\mathbb{R}}

\newcommand{\lam}{{\lambda}}

\newcommand{\inner}[2]{\langle #1,#2\rangle}

\newcommand{\argmin}{\mathrm{argmin}\,}

\newcommand{\dom}{\mathrm{dom}\,}

\newcommand{\Argmin}{\mathrm{Argmin}\,}
\newcommand{\Argmax}{\mathrm{Argmax}\,}

\newcommand{\bConv}[1]{\overline{\mbox{\rm Conv}}\,(\R^{#1})}

\newcommand{\tx}{\tilde x}

\newcommand{\tz}{\tilde z}

\begin{document}
	\title{A proximal bundle variant
		with optimal
		iteration-complexity \\ for a large
		range of prox stepsizes}
	\date{March 26, 2020 \\
	1st revision: July 27, 2020 \\
	2nd revision: January 15, 2021\\
	3rd revision: June 7, 2021}
	% 	\author{
	% 		Jiaming Liang \thanks{School of Industrial and Systems
	% 			Engineering, Georgia Institute of
	% 			Technology, Atlanta, GA, 30332-0205.
	% 			(email: {\tt jiaming.liang@gatech.edu} and {\tt renato.monteiro@isye.gatech.edu}). This work
	% 			was partially supported by ONR Grant N00014-18-1-2077.}\qquad 
	% 		Renato D.C. Monteiro \footnotemark[1]\qquad }
	\maketitle
	
	\begin{center}
		\textsc{Jiaming Liang \footnote{School of Industrial and Systems
				Engineering, Georgia Institute of
				Technology, Atlanta, GA, 30332-0205.
				(email: {\tt jiaming.liang@gatech.edu}).
				This author
				was partially supported by
				ONR Grant
				N00014-18-1-2077.
				This author was also partially supported by
				NSF grant CCF-1740776 through ARC-TRIAD Fellowship and IDEaS-TRIAD Research Scholarship from Georgia Tech.
			}
			and 
			Renato D.C. Monteiro} \footnote{School of Industrial and Systems
			Engineering, Georgia Institute of
			Technology, Atlanta, GA, 30332-0205.
			(email: {\tt monteiro@isye.gatech.edu}). This author
			was partially supported by ONR Grant
			N00014-18-1-2077.}
	\end{center}

	\begin{abstract}
		%     This paper presents a relaxed proximal bundle method, referred to as the RPB method, for solving nonsmooth convex composite optimization problems. 
		% 	In contrast to other proximal bundle variants that decide whether to perform a serious or null iteration based on a descent condition, the RPB method uses a relaxed descent condition, which does not necessarily yield a decrease in the function value. This paper establishes the iteration-complexity of the RPB method to find an $ \varepsilon $-optimal solution in both convex and strongly convex settings and demonstrates both complexity bounds are optimal (possibly up to a logarithmic term) for a large range of prox stepsizes. 
		% 	To the best of our knowledge, this is the first time that the optimal complexity bounds for convex and strongly convex nonsmooth composite optimization problems have been established by a proximal bundle variant. 
		% 	Finally, this paper shows that the composite subgradient method with a constant prox stepsize can be included into the RPB method as a special instance.
		
		This paper presents a proximal bundle variant, namely, the relaxed proximal bundle (RPB) method, for solving convex nonsmooth composite optimization problems.
		Like other proximal bundle variants, RPB solves a sequence of prox bundle subproblems whose objective functions are regularized composite cutting-plane models.
		Moreover, RPB uses a novel condition to decide
		whether to perform a serious or null iteration
		which does not necessarily yield a function value decrease.
		Optimal 
		%		(possibly up to a logarithmic term) 
		iteration-complexity bounds for RPB are established for a large range of prox stepsizes, both in the convex and strongly convex settings.
		%    Iteration-complexity bounds for the RPB method in the
		%    convex and strongly convex settings are derived which are optimal
		%    (possibly up to a logarithmic term) for a large range of prox stepsizes.
		To the best of our knowledge, this is the first time that a proximal
		bundle variant is shown to be optimal for a large range of prox stepsizes.
		Finally, iteration-complexity results for RPB to obtain iterates satisfying practical termination criteria, rather than near optimal solutions,
		are also derived.
		\\
		
		% 	{\bf Keywords:} Nonsmooth optimization $\cdot$ Iteration-complexity $\cdot$ Proximal bundle method $\cdot$ Optimal complexity bound
		\\
		
		% 	{\bf Mathematics Subject Classification (2010)} 
		% 	49M37 $\cdot$ 65K05 $\cdot$ 68Q25 $\cdot$ 90C25 $\cdot$ 90C30 $\cdot$ 90C60

	\end{abstract}
	
	\section{Introduction}\label{sec:intro}
	
	The main goal of this paper is to present a proximal bundle variant,
	namely, the relaxed proximal bundle (RPB) method, whose 
	iteration-complexity is optimal (possibly up to a logarithmic term), for a large range of prox
	stepsizes, in the context of
	convex nonsmooth
	composite optimization (CNCO) %and strongly CNCO
	problems.
	
	RPB is presented in the context of the CNCO problem
	\begin{equation}\label{eq:ProbIntro}
		\phi^{*}:=\min \left\{\phi(x):=f(x)+h(x): x \in \R^n\right\}
	\end{equation}
	where:
	i) $ f, h: \R^{n} \rightarrow \R\cup \{ +\infty \} $ are proper closed convex functions
	such that
	$ \dom h \subseteq \dom f $;
	ii) $h$ is $M_h$-Lipschitz continuous and $ \mu $-convex on $\dom h$ for some $ M_h\in [0,\infty] $ and $ \mu\ge 0 $;
	% 	\red{revise, make iii) here}It is assumed that
	and iii) a zeroth-order (resp.,
	first-order) oracle, which
	for each $x \in \dom h$ returns
	$(f(x),h(x))$ (resp.,
	$f'(x)\in \partial f(x)$ such that
	$\|f'(x)\| \le M_f$),
	is available.
	Like other proximal bundle variants,
	the $j$-th iteration of RPB
	considers the cutting-plane model
	% 	\begin{equation}\label{def:fj}
	%  f_{j}(\cdot)=\max\{f(x_i)+\inner{g_i}{\cdot-x_i}: \, x_i \in C_j\}
	% 	\end{equation}
	\begin{equation}\label{def:fj}
		f_j(\cdot) = \max \left\lbrace  f(x)+\inner{f'(x)}{\cdot-x} : \, x \in C_j\right\rbrace 
	\end{equation}
	where
	$C_j$ is a suitable subset of the iterates $\{x_0,x_1,\ldots,x_{j-1}\}$ generated so far. RPB then
	solves the prox bundle subproblem
	\begin{equation}
		x_j :=\underset{u\in  \R^n}\argmin
		\left\lbrace \underline{\phi}_j^\lam(u) := f_j(u) + h(u) +\frac{1}{2\lam}\|u- x^c_{j-1} \|^2 \right\rbrace  \label{def:xj}
	\end{equation}
	% \[
	% x_{j}=\underset{u\in\dom h}\argmin\left\lbrace f_{j}(u)+h(u)+\frac{1}{2\lam}\|u-x_{j-1}^c\|^2 \right\rbrace 
	% \]
	for $x_j$ where $ \lam $ is the prox stepsize (which for simplicity is assumed constant
	throughout the execution of RPB) and
	% $m_j$ is the optimal value,
	$x^c_{j-1}$ is the prox-center.
	% 	and
	% 	finally invokes the aforementioned
	% 	zeroth-order oracle to
	% 	compute the subgradient $g_j=g(x_j)$ of $f$ at $x_j$.
	It is also assumed that a solver oracle that can exactly solve \eqref{def:xj} is available.
	Complexity bounds described in this paper are relative to the number of RPB iterations performed, each of which consisting of two zeroth-order oracle calls ($ f $ and $ h $), a subgradient call for $ f $, and the resolution of the prox bundle subproblem \eqref{def:xj}.

	Like many other proximal bundle methods, RPB performs two types of iterations, namely:
	i) serious ones during which the prox-centers are changed; and
	ii) null ones where the prox-centers are left unchanged.
	Moreover, RPB uses a novel condition
	% 	, which does not necessarily lead to a function value decrease,
	to decide
	whether to perform a serious or null iteration which does not necessarily yield a function value decrease.
	A nice feature of our complexity analysis of
	RPB is that it considers a flexible
	bundle management policy (i.e., the way $C_j$ is updated) 
	%that does not aggregate cuts.
	which allows for some of cuts
	to be removed but not aggregated (i.e., combined as convex combination).

	{\bf Contributions.}
	%	This paper establishes an iteration-complexity bound for RPB with an arbitrary prox stepsize $\lam>0$ 
	%	to obtain a $\bar \varepsilon$-solution (i.e., a point $\bar x\in \dom h$ satisfying $\phi(\bar x)-\phi^*\le \bar \varepsilon$) of
	%	any instance of \eqref{eq:ProbIntro}.
	%	Using this general result
	%	and considering the class ${\cal I}_\mu(M_f,R_0)$ consisting of all instances satisfying the three assumptions below \eqref{eq:ProbIntro}
	%	and the condition that
	%	the distance of the initial point $x_0$ to the set of optimal solutions of \eqref{eq:ProbIntro} is
	%	bounded by $R_0$, it is
	%	shown that
	%	RPB has optimal iteration-complexity (possibly up to a logarithmic term) given by
	%	\eqref{eq:bound}:
	%	1) with respect to
	%	${\cal I}_\mu(M_f,R_0)$
	%	for any $\mu \ge 0$
	%	as long as
	%	$ \Omega(R_0/M_f)= \lam = {\cal O}(\min\{R_0^2/\bar \varepsilon, 1/\mu \}) $;
	%	and 
	%	2) with respect to
	%	the subclass of ${\cal I}_0(M_f,R_0)$
	%	consisting of those instances for which %$M_h={\cal O}(M_f)$
	%	$h$ is the indicator function of a closed convex set (or more generally,
	%	$M_h={\cal O}(M_f)$)
	%	as long as
	%	$ \Omega(\bar \varepsilon/M_f^2)= \lam = {\cal O}(R_0^2/\bar \varepsilon) $.
	%	The latter case shows that,
	%	when $\mu=0$,
	%	RPB with a wide range of $\lam$ is optimal for solving instances of
	%	${\cal I}_0(M_f,R_0)$ for which
	%	$h$ is a closed indicator function. This contrasts
	%	with the proximal bundle variant of \cite{kiwiel2000efficiency}
	%	which the authors strongly believe to be
	%	non-optimal regardless of $\lam$, based on rigorous arguments given
	%	in Subsection \ref{subsec:compare}.
	This paper establishes an iteration-complexity bound for RPB with an arbitrary prox stepsize $\lam>0$ 
		to obtain a $\bar \varepsilon$-solution of \eqref{eq:ProbIntro} (i.e., a point $\bar x\in \dom h$ satisfying $\phi(\bar x)-\phi^*\le \bar \varepsilon$).
% 		As a consequence, letting
% 		 $d_0 $ be the distance of the initial point $x_0$ to the set of optimal solutions of \eqref{eq:ProbIntro},
% 		 it is shown that,
% 		for any given universal constants $C, C'>0$,
% 		there hold: i) if
% 		$ \mu \in [0,C'M_f/d_0] $, then the iteration-complexity bound of RPB with any stepsize
% 		$ \lam \in [d_0/M_f, C d_0^2/\bar \varepsilon] $
% 		is 
% 		\[
% 		{\cal O}\left(\min \left\lbrace \frac{M_f^2 d_0^2}{\bar \varepsilon^2}, \frac{M_f^2}{\mu \bar \varepsilon} \log\left( \frac{\mu d_0^2}{ \bar \varepsilon} + 1\right) \right\rbrace + 1 \right);
% 		\]
% 		ii) if $ M_h\le C M_f $ and $ \mu=0 $, then the iteration-complexity bound of RPB with any stepsize $ \lam \in [\bar \varepsilon/(C M_f^2), C d_0^2/\bar \varepsilon] $ is $ {\cal O}(M_f^2 d_0^2/\bar \varepsilon + 1) $.
		As a consequence, letting
		 $d_0 $ denote the distance of the initial point $x_0$ to the set of optimal solutions of \eqref{eq:ProbIntro},
		 it is shown that the iteration-complexity of RPB
		 is similar to that of the constant stepsize composite subgradient (CS-CS) method
		 under either one of the following two cases:
		 \begin{itemize}
		     \item [1)] $ \lam \in [d_0/M_f, C d_0^2/\bar \varepsilon] $ and
		 $ \mu \in [0,C'M_f/d_0] $;
		     \item[2)] $ \lam \in [\bar \varepsilon/(C M_f^2), C d_0^2/\bar \varepsilon] $, $ M_h\le C' M_f $ and $ \mu=0 $,
		 \end{itemize}
		where $C,C'$ are positive universal constants.
		It is worth noting that:
		a) case 1 allows $\mu$ to be zero and $M_h$ to be arbitrary, but its $\lam$-range is smaller than
		the one in case 2; and
		b) case 2 covers all instances of
		\eqref{eq:ProbIntro} for which $h$  is the indicator function of a closed convex set.
		Using these results, it is then argued that
 		RPB has optimal iteration-complexity with respect to
		some important instance classes of~\eqref{eq:ProbIntro}.

	Iteration-complexity results are also established for RPB to obtain iterates satisfying practical termination criteria rather than a $\bar \varepsilon$-solution.
	Another interesting conclusion of our analysis 
	is that the CS-CS method can be viewed as a special instance of RPB 
	as long as its prox stepsize $\lam$ is sufficiently small.
	%	Finally, a secondary contribution of the paper is the derivation of lower complexity bounds
	%	for some classes of
	%	composite instances including the ones
	%	mentioned in the previous paragraph.

	{\bf Related works.} 
	Some preliminary ideas towards the development of
	the proximal bundle method were first presented in \cite{lemarechal1975extension,wolfe1975method}
	and formal presentations of the method were given in
	\cite{lemarechal1978nonsmooth,mifflin1982modification}.
	% It has been further studied in \cite{du2017rate,kiwiel1983aggregate,kiwiel2000efficiency}
	% and a comprehensive treatment of the proximal bundle method is presented in \cite{ruszczynski2011nonlinear,urruty1996convex}. 
	Convergence analysis of the proximal bundle method for
	CNCO problems  has been broadly discussed
	in the literature and can be found for example
	in the textbooks \cite{ruszczynski2011nonlinear,urruty1996convex}.
	Different bundle management policies in the context of proximal bundle methods are discussed for example in
	\cite{du2017rate,frangioni2002generalized,kiwiel2000efficiency,de2014convex,ruszczynski2011nonlinear,van2017probabilistic}.
	
	%	{\color{red}
	%		Iteration-complexity lower bounds for obtaining a $ \bar \varepsilon $-solution of \eqref{eq:ProbIntro} were given in \cite{nesterov2018lectures}. 
	%		Under the assumption that $ h=0 $, and hence $ M_h=0 $, the lower bound for obtaining a $ \bar \varepsilon $-solution of the CNCO problem is $ \Omega(M_f^2d_0^2/\bar \varepsilon^2) $ (see for example Theorem 3.2.1 of \cite{nesterov2018lectures}).
	%		If, in addition, $ f $ is $ \mu $-strongly convex everywhere, then the lower bound is $\Omega(M_f^2/\mu\bar \varepsilon)$ (see for example Theorem 3.2.5 of \cite{nesterov2018lectures}).
	%		
	%		Also check Nemirovski's book \cite{nemirovsky1983problem}.
	%	}
	
	% Iteration-complexity results exist as well, but are less studied.
	% and no optimal complexity bounds have been established in general.

	Previous iteration-complexity analysis of some
	proximal bundle variants can be found in \cite{astorino2013nonmonotone,du2017rate,kiwiel2000efficiency}.
% 	discuss iteration-complexity bounds that have been previously derived
% 	in \cite{astorino2013nonmonotone,du2017rate,kiwiel2000efficiency}
	More specifically, papers \cite{astorino2013nonmonotone,kiwiel2000efficiency} 
	consider proximal bundle variants for the
	special case of the CNCO problem where
	$h$ is the indicator function
	of a nonempty closed convex set (and hence $\mu=0$).
	Paper \cite{du2017rate} analyzes the
	complexity of the proximal bundle method considered in \cite{kiwiel2000efficiency} under the condition that $h=0$ and $f$ is strongly convex. A detailed discussion
	of how the complexity
	bounds obtained in these papers compare to the ones
	obtained in this work is given in Subsection \ref{subsec:compare}
	and the conclusion is that
	the bounds in \cite{astorino2013nonmonotone,du2017rate,kiwiel2000efficiency} are generally much worse than the ones
	obtained in this work for most (in some cases, all)
	values of the prox stepsize $\lam$.

	Another method related, and developed subsequently, to
	the proximal bundle method is the  bundle-level method,
	which was first proposed in \cite{lemarechal1995new}
	and extended in many ways in \cite{ben2005non,kiwiel1995proximal,lan2015bundle}.
	These methods have been shown to have optimal
	iteration-complexity in the setting
	of the CNCO problem with $h$ being the indicator function
	of a compact convex set. Since their generated subproblems
	do not have a prox term, and hence do not use a prox stepsize,
	they are different
	from the ones studied in this paper.
	
	{\bf Organization of the paper.}
	Subsection~\ref{subsec:DefNot}  presents basic definitions and complexity theory notation used throughout the paper.
	Section~\ref{sec:main} formally
	describes the assumptions
	on the CNCO problem \eqref{eq:ProbIntro},
	reviews
	the CS-CS method
	and discusses its iteration-complexity.
	Subsections~\ref{subsec:method}-\ref{subsec:upper} present the RPB method and state the main results of the paper, namely, the general iteration-complexity for RPB and its implications in convex and strongly convex settings.
	%	and the description of ranges on $\lam$ for which RPB is optimal.
	Subsections~\ref{subsec:compare} 
	discusses, in the unconstrained CNCO context, results established for
	other proximal bundle variants
	in light of the ones obtained for RPB
	in this paper.
	Section~\ref{sec:null} establishes a bound on the number of null iterations between two consecutive serious iterations and 
	discusses the relationship between CS-CS and RPB.
	Section~\ref{sec:serious} provides the proof of the general iteration-complexity for RPB stated in
	Section~\ref{sec:RPB}.
	Section~\ref{sec:other} describes two alternative notions of approximate solutions for \eqref{eq:ProbIntro} and presents iteration-complexity results
	with respect to them.
	Section~\ref{sec:optimal} reviews basic concepts from complexity theory, presents the lower complexity bound,
	and shows both CS-CS and RPB are optimal with respect to some instance classes of \eqref{eq:ProbIntro} introduced in this section.
	Section~\ref{sec:conclusion} presents some concluding remarks and possible extensions.
	Finally, Appendix~\ref{sec:CS-CS pf} provides the proof of 
	the iteration-complexity for the CS-CS
	method, Appendix~\ref{sec:lb} gives the proof the lower complexity bound, and Appendix~\ref{sec:pf-opt} provides the proof of optimal complexity of the RPB method.

	\subsection{Basic definitions and notation} \label{subsec:DefNot}
	
	%This subsection provides some basic definitions and notation used throughout this paper.
	
	The set of real numbers is denoted by $\mathbb{R}$. 
	The set of non-negative real numbers and the set of positive real numbers are denoted by $ \R_+ $ and $ \R_{++} $, respectively.
	%	Let $ \bar \R $ denote the set $ \R\cup \{\pm \infty \} $.
	Let $\R^n$ denote the standard $n$-dimensional Euclidean 
	space equipped with  inner product and norm denoted by $\left\langle \cdot,\cdot\right\rangle $
	and $\|\cdot\|$, respectively. 
	%	Let  $ \bar B(0;M_f) =\{s\in \R^n: \|s\|\le r \} $ denote the ball centered at $ 0 $ with a radius $ r $.
	% 	The indicator function $I_X$ of a set $ X\subset \R^n $ is defined as $ I_X (z) =0 $ for every $ z\in X, $ and $ I_X (z) =\infty $, otherwise. 
	% 	Let $ \cball{x}{R} $ denote a closed ball $ \subset \R^n $ centered at $ x $ with radius $ R>0 $.
	%	Let $ d(X,Y):=\inf\{\|x-y\|: \forall x\in X, \forall y\in Y \} $ denote the distance between two sets $ X, Y \subset \R^n $.
	Given a set $ S\subset \R^n $, its linear (resp., convex) hull is
	denoted by $\text{Lin} \, S $ (resp., ${\rm conv} \, S$).
	% 	is defined as the set of all convex combinations of elements of $ S $.
	% 	The subspace  generated by vectors
	% 	 $ x_1,\ldots,x_k \in \R^n$ is denoted by $\text{Lin}\{x_1,\ldots,x_k \} $.
	% 	Let $ {\cal O}_1(\cdot) $ denote $ {\cal O}(\cdot + 1) $ where $ {\cal O} $ is the big O notation.
	Let $\log(\cdot)$ denote the natural logarithm and define 
% 	$ \log_1(t):=\log(t+1)$ and 
	$ \log_1^+(\cdot):= \max \{\log(\cdot),1\} $.

	Let $\psi: \R^n\rightarrow (-\infty,+\infty]$ be given. The effective domain of $\psi$ is denoted by
	$\dom \psi:=\{x \in \R^n: \psi (x) <\infty\}$ and $\psi$ is proper if $\dom \psi \ne \emptyset$.
	Moreover, a proper function $\psi: \R^n\rightarrow (-\infty,+\infty]$ is $\mu$-convex for some $\mu \ge 0$ if
	$$
	\psi(\alpha z+(1-\alpha) u)\leq \alpha \psi(z)+(1-\alpha)\psi(u) - \frac{\alpha(1-\alpha) \mu}{2}\|z-u\|^2
	%\quad \forall z, u \in \dom \psi, \, \forall \alpha \in [0,1].
	$$
	for every $z, u \in \dom \psi$ and $\alpha \in [0,1]$.
	The set of all proper lower semicontinuous convex functions $\psi:\R^n\rightarrow (-\infty,+\infty]$  is denoted by $\bConv{n}$.
	%If $\psi$ is differentiable at $\bar x \in \R^n$, then its affine approximation $\ell_\psi(\cdot;\bar x)$ at $\bar x$ is defined as
	%%\begin{equation}\label{eq:defell}
	%%\ell_\psi(z;\bar z) :=  \psi(\bar z) + \inner{\nabla \psi(\bar z)}{z-\bar z} \quad \forall  z \in \R^n.
	%%\end{equation}
	%\[
	%\ell_\psi(z;\bar z) :=  \psi(\bar z) + \inner{\nabla \psi(\bar z)}{z-\bar z} \quad \forall  z \in \R^n.
	%\]
	%The effective domain of $\psi$ is denoted by $\dom \psi:=\{x \in \Re^n: \psi (x) <\infty\}$.
	For $\varepsilon \ge 0$, the \emph{$\varepsilon$-subdifferential} of $ \psi $ at $z \in \dom \psi$ is denoted by
	% 	\begin{equation}\label{eq:epsubdiff}
	$\partial_\varepsilon \psi (z):=\left\{ s \in\R^n: \psi(u)\geq \psi(z)+\left\langle s,u-z\right\rangle -\varepsilon, \forall u\in\R^n\right\}$.
	% 	\end{equation}
	The subdifferential of $\psi$ at $z \in \dom \psi$, denoted by $\partial \psi (z)$, is by definition the set  $\partial_0 \psi(z)$.
	
% 	\subsection{Complexity theory notation} \label{subsec:cmplx-notation}
	Let constant $\bar c \in (1,\infty)$ and functions $ p, q: {\cal Y} \rightarrow \R_{+} $ defined in an arbitrary set ${\cal Y}$
	be given.
	We write $p(\cdot) = {\cal O}(q(\cdot))$ (with underlying constant  $\bar c$)
	if $ p(y)\le \bar c q(y) $ for every $y \in {\cal Y}$.
% 	We write $p(\cdot)=\Omega(q(\cdot))$
% 	if $q(\cdot) = {\cal O}( p(\cdot))$,
% 	and write $p(\cdot)=\Theta(q(\cdot))$
% 	if $p(\cdot) = {\cal O}( q(\cdot))$ and $p(\cdot) = \Omega( q(\cdot))$.
	Finally, we write $ p(\cdot)={\cal O}_1(q(\cdot)) $ if $ p(\cdot)={\cal O}(q(\cdot) + 1) $.
	It is worth emphasizing that the above ${\cal O}(\cdot)$ concept depends
	on the pre-specified constant
	$\bar c$.
	Clearly, it follows from the above definition that
	if $p_i(y) = {\cal O}(q_i(y))$
	with underlying constant $\bar c_i \in (1,\infty)$ for $i=1,2$, then
	$p_1(y)p_2(y) = {\cal O}(q_1(y)q_2(y))$ with
	underlying constant
	$\bar c_1 \bar c_2$.

	\section{Assumptions and the CS-CS method}\label{sec:main}
	This section contains two subsections. The first one formally describes
	the assumptions made on the CNCO problem \eqref{eq:ProbIntro}.
	The second one presents the CS-CS method and the iteration-complexity of it for solving
	\eqref{eq:ProbIntro}.

	\subsection{Assumptions}\label{subsec:assumption}
	% 	This subsection describes the main problem and the assumptions made on it in detail. It also discusses the notion of $ \bar \varepsilon $-solution that will be used as a stopping criterion of RPB for solving the main problem.
	
	% 	We start by formally describing the assumptions
	% 	made on \eqref{eq:ProbIntro}, which is the problem of
	% 	interest in this paper.
	For some triple
	$(M_f,M_h,\mu) \in \R_+ \times [0,\infty] \times \R_+ $,
	the following conditions on \eqref{eq:ProbIntro} are assumed to hold:
	% 	consider
	% 	a triple $(f,f';h)$ satisfies
	% 	the following conditions:
	\begin{itemize}
		\item[(A1)]
		functions $f, h \in \bConv{n}$ are such that
		$\dom h \subset \dom f$ and function $f':\dom h \to \R^n$
		is such that $f'(x) \in \partial f(x)$ for all $x \in \dom h$;
		\item[(A2)]
		the set of optimal solutions $X^*$ of
		problem \eqref{eq:ProbIntro} is nonempty;
		% 		$g_f:\dom h \to \R^n$ is such that
		% %		for some scalar $M_f \ge 0$, there holds
		% 		\begin{equation}\label{assumption:subgradient-bound}
		% 		g_f(u) \in \partial f(u) \cap \cball{0}{M_f}, %\quad \|g(u)\| \le M_f
		% 		\quad \forall u\in \dom h;
		% 		\end{equation}
		\item[(A3)]
		$h$ is $\mu$-convex and $\|f'(x)\| \le M_f$ for all $x \in \dom h$;
		\item[(A4)]
		$h$ is $M_h$-Lipschitz continuous on $\dom h$, i.e.,
		\[
		|h(u)-h(v)|\le M_h\|u-v\| \quad \forall u,v \in \dom h.
		\]
	\end{itemize}

	As already mentioned in Section~\ref{sec:intro},
	in addition to the above assumptions, it is assumed that
	a zeroth-order  oracle, which
	for each $x \in \dom h$ returns
	$(f(x),h(x))$,
	and a solver oracle that can exactly solve \eqref{def:xj}, are also available.
	Complexity bounds developed throughout this paper are in terms of RPB iterations. Since each iteration involves two
	zeroth-order oracle calls, one first-order oracle call, and one solver oracle call, they are
	also complexity bounds for the
	number of
	oracle calls.

	%	We make some remarks about the above assumptions. 
	%	First, (A1) does not imply that $\partial h(u)$ is bounded for every $u \in \dom h$.
	%	For example, an indicator function of a closed convex set
	%	satisfies (A1) but its subdifferential is unbounded.
	%	Second, function $ g(\cdot) $ can be thought as an oracle which, for given
	%	$u \in \dom h$, returns a subgradient
	%	of $f$ at $u$ whose magnitude is bounded by $M_f$.
	%	Third, it follows as a consequence of (A3) that
	%	\begin{equation}\label{ineq:func}
	%	|f(u)-f(v)|\le M_f\|u-v\| \quad \forall u,v \in \dom h.
	%	\end{equation}
	
	We now make some remarks about assumptions (A1)-(A4). 
	First, function $ f'(\cdot) $ should be viewed as an oracle which, for given
	$u \in \dom h$, returns a subgradient
	of $f$ at $u$ whose magnitude is bounded by $M_f$.
	Second, it follows as a consequence of (A3) that
	\begin{equation}\label{ineq:func}
		|f(u)-f(v)|\le M_f\|u-v\| \quad \forall u,v \in \dom h.
	\end{equation}
	Third, if $ \mu>0 $ and $ \dom h $ is unbounded, then $M_h$ can not be finite.
	Fourth, if $u \in \dom h$,
	(A4) does not imply that $\partial h(u)$ is bounded, 
	even when $ M_h$ is finite.
	For example, an indicator function of a closed convex set
	satisfies (A4) but its subdifferential at a point
	in its relative boundary is unbounded.
	
	For a given tolerance $\bar \varepsilon>0$, a point $ x $ is called a
	$ \bar \varepsilon $-solution of \eqref{eq:ProbIntro} if
	\begin{equation}\label{def:suboptimal}
		\phi(x) - \phi^*\le \bar \varepsilon
	\end{equation}
	where $\phi^*$ is as in \eqref{eq:ProbIntro}.
	Note that, while \eqref{def:suboptimal} is theoretically appealing from a complexity point of view, it can rarely be used as
	a stopping criterion since $\phi^*$ is generally
	not known. Other more practical
	stopping criteria are discussed
	in Section \ref{sec:other}.

	\subsection{Review of the CS-CS method}\label{subsec:cs}
	
	%	This subsection describes the main properties about the constant prox stepsize CS (CS-CS) method
	%	and shows that, under the assumption that
	%	$\lam$ is sufficiently small, RPB reduces to the
	%	CS-CS method. 
	
	% 	This subsection reviews the CS-CS method and describes its optimality property within a certain algorithm class for solving \eqref{eq:ProbIntro}.
	% 	With this goal in mind, it
	% 	also reviews some basic concepts
	% 	from complexity theory.
	
	The CS-CS method with initial point $x _0 \in \dom h$ and constant prox stepsize $\lambda$, denoted by CS-CS$(x_0,\lam)$,
	recursively computes its iterates
	according to
	\begin{equation}\label{eq:sub}
		x_j =\underset{u\in  \R^n}\argmin
		\left\lbrace  f(x_{j-1})+\inner{f'(x_{j-1})}{u-x_{j-1}} + h(u) +\frac{1}{2\lam} \|u- x_{j-1} \|^2 \right\rbrace.
	\end{equation}

	Let $(x_0,\lam) \in \dom h \times \R_{++}$ and $(M_f, \mu, d_0) \in \R_+ \times \R_+ \times \R_{++} $ be given.
	Consider an instance $(x_0,(f,f';h))$ satisfying conditions (A1)-(A3) and
	\begin{equation}\label{def:d0}
		d_0 =\inf \{\|x_0-x^*\|: x^*\in X^*\}=\|x_0-x_0^*\|
	\end{equation}
	where $ x_0^* $ is the closest point $x^* \in X^*$ with respect to $x_0$.
	%    It is easy to see that
	%	CS-CS$(x_0,\lam)$
	%	satisfies property b) in the paragraph containing \eqref{incl:xk}.
	We let the $\bar \varepsilon$-iteration complexity
	bound for CS-CS denote
	the bound on the total number of iterations performed by CS-CS until a $\bar \varepsilon$-solution is obtained.
	For any given universal constant $C >1$, it follows from
	Proposition \ref{prop:sub-new} that %and the above definition of ${\cal I}_\mu(M_f,R_0)$ that
	CS-CS$(x_0,\lam)$ with any stepsize $\lam$ such that
	$\bar \varepsilon/( C M_f^2) \le 4\lam \le \bar \varepsilon/ M_f^2$
	%	is in ${\cal A}({\cal I}_\mu(M_f,R_0),\bar \varepsilon)$
	has $\bar \varepsilon$-iteration complexity bound given by
	\begin{equation}\label{eq:bound}
		{\cal O}_1\left(\min \left\lbrace \frac{M_f^2 d_0^2}{\bar \varepsilon^2}, \left( \frac{M_f^2}{\mu \bar \varepsilon} +1 \right) \log\left( \frac{\mu d_0^2}{ \bar \varepsilon} + 1\right) \right\rbrace  \right)
	\end{equation}
	%	\begin{equation}\label{eq:bound}
	%		{\cal O}_1\left(\min \left\lbrace \frac{M_f^2 R_0^2}{\bar \varepsilon^2}, \red{\left( 1 + \frac{M_f^2}{\mu \bar \varepsilon}\right)} \log_1\left( \frac{\mu R_0^2}{ \bar \varepsilon} \right) \right\rbrace  \right)
	%	\end{equation}
	with the convention that
	the second term is equal to the first one when $\mu=0$.
	(It is worth noting that the second term 
% 	is a strictly decreasing function of $\mu>0$ and
	converges to the first one as
	$\mu \downarrow 0$.)

	We now make some remarks about bound \eqref{eq:bound}.
	First,
	for the case in which $\mu=0$ and under the extra assumption that $x_0\in \Argmin\{h(x): x\in \R^n \}$, 
	it is well-known that
	the $ \bar \varepsilon $-iteration complexity bound for
	CS-CS$(x_0,\lam)$ with $\bar \varepsilon/(C M_f^2) \le \lam \le \bar \varepsilon/M_f^2$ for a universal constant $C >1$
	is as in \eqref{eq:bound} with $ \mu=0 $
	(e.g., see Theorem 9.26 of \cite{beck2017first}).
	Hence, the result described in the previous paragraph
	generalizes the one in the previous sentence in that it removes the extra assumption above
	but requires changing the range on $\lam$ to
	$\bar \varepsilon/(C M_f^2) \le 4\lam \le  \bar \varepsilon/M_f^2$ for a universal constant $C >1$.
	Second, it follows as a special case of the analysis in
	Chapter 3.2.3 of \cite{nesterov2018lectures} that
	a certain variable stepsize projected subgradient method
	for the case in which $\mu=0$ has
	$\bar \varepsilon$-iteration complexity bound
	for instances $(x_0,(f,f';h)) $ satisfying (A1)-(A3) and such that $h$ is the indicator function of a
	closed convex set.
	In this regard, the result in the previous paragraph
	with $\mu=0$ extends the result just mentioned
	to all instances satisfying (A1)-(A3), but replaces
	the variable stepsize projected subgradient method
	with CS-CS$(x_0,\lam)$ with $\bar \varepsilon/(C M_f^2)\le 4\lam \le \bar \varepsilon/M_f^2$ for a universal constant $C >1$.

	\section{The RPB method and main results}\label{sec:RPB}
	
	This section contains three subsections. 
	The first one describes the RPB method
	and discusses
	serious/null decision policies,
	storage requirement of RPB,
	and bundle management policies.
	The second one states a general $\bar \varepsilon$-iteration complexity bound for RPB, and two consequences of the general bound in the convex and strongly convex settings.
	The third one derives  $\bar \varepsilon$-iteration complexity bounds for another proximal bundle variant
	with respect to unconstrained CNCO instances
	and compares them
	with the ones obtained for RPB in
	Subsection \ref{subsec:upper}.
	%	and compares them with results in Subsection~\ref{subsec:optimal}.

	\subsection{The RPB method}\label{subsec:method}

	We start by formally stating
	the RPB method.
	Its description below uses
	the cutting-plane model $f_j$ defined in \eqref{def:fj} and the availability
	of the subgradient
	oracle function $f'(\cdot)$ as in (A3).
	Note that the model $f_j$ is used in the construction of subproblem \eqref{def:xj} and
	is defined in terms of a finite  set
	$C_j \subset \{x_0,x_1,\ldots,x_{j-1}\}$
	which is updated
	according to step 2
	below. Moreover,
	RPB is stated without a
	specific termination criterion with the intent of
	making it as flexible as possible. 
	% {\color{red}???Remove}The results following it establish iteration-complexity results for obtaining  the three
	% notions of approximate solutions considered in Subsection \ref{subsec:assumption},
	% namely,
	% a $ \bar \varepsilon $-solution (see Theorem \ref{thm:suboptimal}), a $ (\hat \rho, \hat \varepsilon) $-solution triple (see Theorem \ref{thm:approximate}),
	%  and a $ \bar \varepsilon $-solution pair when
	%  $\dom h$ is bounded
	% (see Corollary \ref{cor:pair}).
	Subsection~\ref{subsec:upper} (resp., Section~\ref{sec:other}) then describes iteration-complexity bounds for it to obtain a $\bar \varepsilon$-solution (resp., other types of approximate solutions).
	
	%We are now ready to state RPB.
	
	% 	 RPB can be viewed as an inexact proximal point
	% 	 method for solving \eqref{eq:ProbIntro}.
	% 	 Indeed, it consists of
	% 	 two types of iterations, namely: serious and null iterations.
	% 	 Its serious ones can be interpreted as iterations
	% 	 of an inexact proximal point method. More specifically,
	% 	 its $k$-th serious iteration
	% 	 computes an approximate solution $\tz_k$ of the prox subproblem
	% 	 \begin{equation}\label{eq:prox}
	% 	     \min \left \lbrace \phi(u) +
	% 		\frac{1}{2\lam} \|u-z_{k-1}\|^2 :u\in \R^n \right\rbrace,
	% 	 \end{equation}
	% 	 where $\lam>0$ is the prox stepsize (assumed constant throughout this paper) and $z_{k-1}$ is the prox-center, and then computes
	% 	 a new prox-center $z_k$ such that
	% 	 $(z-$
	% 	More specifically, it consists of two related processes: 1) sequences of consecutive null iterations, where approximate solutions of \eqref{eq:prox} with a fixed prox-center are iteratively computed in each sequence; and 2) serious iterations separated by null ones, where approximate solutions of \eqref{eq:prox} with desired accuracy are obtained and prox-centers are updated.
	
	\noindent\rule[0.5ex]{1\columnwidth}{1pt}
	
	RPB
	
	\noindent\rule[0.5ex]{1\columnwidth}{1pt}
	\begin{itemize}
		\item [0.] Let $ x_0\in \dom h $, $\lam>0$ and $ \delta>0 $
		be given, invoke the oracle $f'(\cdot)$ to obtain $ f'(x_0)\in \partial f(x_0) $, and set $ x^c_0 = x_0 $, $\tx_0=x_0$, $\hat z_0=x_0$, $ C_1=\{x_0\} $, $ j=1$ and $k=1$;
		\item [1.]
		%Consider the $j$-th cutting-plane model as defined in \eqref{def:fj},
		% 	\begin{equation}\label{def:fj}
		%  	f_j(\cdot):=\max 
		% 	\{f(x_i) + \inner{g_i}{\cdot -x_i} : x_i \in C_j, \, g_i \in \partial f(x_i)  \}, 
		% 	\end{equation}
		Compute $x_j$ according to \eqref{def:xj}, the function values
		$f(x_j)$, $h(x_j)$ and
		$f_j(x_j)$, and the optimal value
		$m_j:=\underline{\phi}_j^\lam(x_j)$ of subproblem
		\eqref{def:xj}, and 
		invoke the oracle $f'(\cdot)$ to obtain $ f'(x_j)\in \partial f(x_j) $.
		% 	\begin{align}
		% 	(x_j,m_j) &: =\underset{u\in\dom h}\argmin
		% 	 \left\lbrace \underline{\phpreni}_j^\lam(u) := f_j(u) + h(u) +\frac{1}{2\lam}\|u- x^c_{j-1} \|^2 \right\rbrace,  \label{def:xj}
		% 	\end{align}
		Moreover, consider the function
		$\phi_j^\lam$ defined as
		\begin{equation}\label{def:philam}
			\phi_j^\lam:=\phi +
			\frac{1}{2\lam} \|\cdot-x^c_{j-1}\|^2
		\end{equation}
		and
		let $\tx_j$ be such that
		% 		be a point
		% 		which minimizes $\phi_j^\lam$
		% 	    over the finite set $X_{j-1} \cup \{x_j\}$;
		\begin{equation}\label{def:txj}
			\tx_j\in \Argmin\left\lbrace \phi_j^\lam(u): u\in \{x_j, \tx_{j-1}\}\right\rbrace;
		\end{equation}
		% 		\begin{equation}
		% 		\left\{ 
		% 		\begin{array}{cc}  
		% 		\tx_j = x_j, & \mbox{if $ j=1 $ or \eqref{ineq:hpe1} holds with $j=j-1$}, \\ [.1in]
		% 		\tx_j\in \Argmin\left\lbrace \phi_j^\lam(u): u\in \{x_j, \tx_{j-1}\}\right\rbrace, & \mbox{otherwise};
		% 		\end{array} \right. \label{def:txj1}
		% 		\end{equation}
		\item[2.]  {\bf If }
		\begin{equation}\label{ineq:hpe1}
			t_j := \phi^\lam_j(\tx_j) - m_j \le \delta,
		\end{equation}
		\begin{itemize}
			\item [2.a)] {\bf then} perform a serious iteration, i.e.,
			choose an arbitrary finite set $ C_{j+1} $
			such that $ \{x_j\} \subset C_{j+1} $,
			and set $ x^c_j = x_j $ and
			\begin{equation}\label{def:hat zk}
				\hat z_k \in \Argmin \left\lbrace  \phi(u) : u \in \{ \hat z_{k-1},\tx_j\} \right\rbrace ;\\
			\end{equation}
			if $\hat z_k$ satisfies the termination criterion, then {\bf stop} and return $\hat z_k$; else, set $ k \ensuremath{\leftarrow} k+1 $, and go to step 3;
			%	and go to step 3;
			\item[2.b)] 
			{\bf else} perform a null iteration, i.e., set $ x^c_j= x^c_{j-1} $, and choose $ C_{j+1} $ such that
			\begin{equation}\label{def:Cj}
				A_{j} \cup \{x_j\} \subset C_{j+1} \subset C_{j}\cup \{x_j\}
			\end{equation}
			where
			\begin{equation}\label{def:A_j}
				A_{j}:=\left\lbrace x\in C_{j}: f(x)+\inner{f'(x)}{x_j-x} =f_j(x_j) \right\rbrace
			\end{equation}
			and $f_j$ is defined in \eqref{def:fj};
			go to step 3;
		\end{itemize}
		\item[3.] %Select $ g_j \in \partial f(x_j) $,
		Set $ j $ \ensuremath{\leftarrow} $ j+1 $ and go to step 1.
	\end{itemize}
	\rule[0.5ex]{1\columnwidth}{1pt}
	
	% {\color{red}??? boundedness of subgradient when $h$ is strongly convex}
	We sometimes refer to RPB as RPB$(x_0,\lam,\delta)$
	whenever it is necessary to make its input
	$(x_0,\lam,\delta)$ explicit.
	An iteration index $j$ for which \eqref{ineq:hpe1} is satisfied is called
	a serious one in which
	case $x_j$ (resp., $\tx_j$) is called a serious iterate
	(resp., auxiliary serious iterate);
	otherwise, $j$ is called a null iteration index.
	Moreover, we assume throughout our presentation that $j=0$ is also a serious iteration index.
	
	We now make some basic observations about RPB.
	First, the index $j$ denotes
	the total iteration count
	and $k=k(j)$ equals
	the number of serious
	iteration indices (including $0$) less than $j$.
	Second, for any $j \ge 1$,
	if $\ell_0$ denotes the largest serious iteration index less than or equal to $ j$, then
	$\tx_j$ is the best point (in terms of $\phi_j^\lam$)
	among the set $\{\tx_{\ell_0}, x_{\ell_0+1}, \ldots, x_j\}$.
	% 	Second, if $\tx_j$ is an auxiliary serious iterate,
	% 	then $x_j$ is a serious iterate and $\tx_j$
	% 	is the best point (in terms of $\phi_j^\lam$)
	% 	among $x_j$, its preceding auxiliary serious iterate
	% 	%$x_{j-1}^c$
	% 	and all the null
	% 	iterates generated between them.
	Third, the iterate
	$\hat z_k$ can be easily seen to be the best auxiliary serious iterate $\tx_j$ (in terms of $\phi$) found
	up to and including the
	$k$-th serious iteration.
	Fourth,
	the complexity results established in Theorems \ref{thm:suboptimal} and \ref{thm:approximate}, and Corollary \ref{cor:pair} below, are with respect to
	$\hat z_k$.
	% 	the best (in terms of $ \phi $) auxiliary serious iterate $ \tx_j $ found so far. 
	This is in contrast to the iteration-complexity analysis of \cite{du2017rate,kiwiel2000efficiency}, which establish complexity bounds with respect to the best (in terms of $ \phi $) serious iterate $ x_j $ (instead of
	$\tx_j$ as above) found so far.
	% 	Fifth, the prox-center $x_{j-1}^c$
	% 	remains the same when $j$ is a
	% 	null iteration index and it is
	% 	updated to the most recent $x_j$ only when $j$ is a
	% 	serious iteration index.
	% 	Fifth, the subgradient
	% 	$f'(x_j)$ in step 1 of RPB
	% 	is used to construct the cutting-plane model $f_{j+1}$ in \eqref{def:fj}
	% 	for the next iteration.
	%	Moreover, the existence of $ g(x_j) \in \partial f(x_j)$ in step 1
	%	follows from (A3)
	%	% the assumption that $ \dom h \subseteq \dom \partial f$
	%	and the fact that \eqref{def:xj}
	%	implies that $x_j \in \dom h$.
	Fifth, the bundle set $C_j$ consists of the set of points that are
	used to construct the cutting-plane model $f_j$ which
	minorizes $f$. 
	Sixth,
	$A_j$ consists of the subset of points from $C_j$
	which are active at the most recent point $x_j$,
	i.e., the set of points  which attains the maximum
	in \eqref{def:fj}.

	We now provide some insights
	on
	how RPB can be viewed as an inexact proximal point method
	(see for example \cite{fuentes2012descentwise,guler1992new,lin2017catalyst,Rock:ppa,MR1783979}) for solving \eqref{eq:ProbIntro}.
	First, recall that each RPB iteration performs either a serious iteration (step 2.a)
	or a null
	one (step 2.b).
	Letting $(z_{k-1},\tz_{k-1})$ denote the $(k-1)$-th serious pair generated after $x_0$, it follows that the sequence
	of consecutive null pairs $\{(x_j,\tx_j): j \in J_k\}$ obtained immediately after $z_{k-1}$ together with
	the next serious pair $(z_k,\tz_k)$ can be
	viewed as an iterative procedure
	to compute a $\delta$-solution of the proximal subproblem  $\min\{\phi(u)+\|u-z_{k-1}\|^2/(2\lam):u\in \R^n\}$.
	% 	which is
	% 	based on the sequence $\{x_j: j \in J_k\}$ obtained
	% 	by exactly solving the
	% 	prox bundle subproblems \eqref{def:xj} for every $j \in J_k$.
	Indeed, first note that
	this subproblem is equivalent to the problem
	$\min \{ \phi_j^\lam(u):u\in\R^n \}$,
	where $\phi_j^\lam$ is as in \eqref{def:philam},  due to the fact that $x_{j-1}^c=z_{k-1}$ for every index $j \in J_k$.
	Second, using the definition of $t_j$ in \eqref{ineq:hpe1} and the fact that $m_j\le m_j^* \le \phi_j^\lam(\tx_j)$ 
	% 	where $\phi_j^\lam$ is defined in \eqref{def:philam}
	where $m_j^*:= \min \{\phi_j^\lam(u): u\in \R^n\}$ (see Lemma \ref{lem:101}),
	% 	and
	% 	the fact that
	% 	the termination
	% 	criterion to stop the above sequence of consecutive null iterations is \eqref{ineq:hpe1},
	we conclude that
	$\phi_j^\lam(\tx_j)-m_j^* \le t_j$. This observation together with the role played by \eqref{ineq:hpe1} in step 2 implies that
	$\tz_k$
	is a $\delta$-solution of the above proximal subproblem.
	Third, once such an approximate
	solution pair $(z_k,\tz_k)$
	is obtained,
	the prox-center $z_{k-1}$ of the above proximal subproblem
	is updated to $z_k$ (see step 2.a) and this
	essentially corresponds to performing an inexact proximal step to  problem \eqref{eq:ProbIntro}.
	Subsection~\ref{sec:serious} and Section~\ref{subsec:other} develop complexity bounds
	on the total number of
	proximal steps that can be performed as above,
	and hence on the number of 
	serious iterations performed
	by RPB, until a pre-specified
	termination criterion is satisfied.

	We now discuss
	some serious/null decision policies that were
	used in
	other proximal bundle methods.
	First,
	% 	while the RPB method uses the
	% 	criterion \eqref{ineq:hpe1} to decide whether to perform
	% 	a serious or null iteration,
	the ones in references
	\cite{bertsekas2015convex,du2017rate,kiwiel2000efficiency,de2014convex,ruszczynski2011nonlinear,urruty1996convex,van2014constrained} all rely on the unified condition 
	%$
	%\phi(x_j)\le (1-\gamma)\phi(x_{j-1}^c)+
	%\gamma (f_j+h)(x_j)$, or equivalently,
	%\begin{equation}\label{ineq:descent}
	%?????\phi(x_{j-1}^c) - \phi(x_j) \ge \frac{\gamma}{1-\gamma}
	%\left[ f(x_j)-f_j(x_j)\right],
	%\end{equation}
	%\begin{equation}\label{ineq:descent}
	%f(x_{j-1}^{c})-f(x_{j}) \ge \gamma\left( f(x_{j-1}^{c}) -f_j(x_{j}) - \frac{\alpha_j}{\lam}\|x_{j}-x_{j-1}^{c}\|^2\right), \quad \alpha_j\in[0,1],
	%\end{equation}
	\begin{equation}\label{ineq:descent}
		\phi(x_{j-1}^c) - \phi(x_j) \ge \frac{\gamma}{1-\gamma}
		\left[ f(x_j)-f_j(x_j) - \frac{\alpha_j}{2\lam}\|x_{j}-x_{j-1}^{c}\|^2\right]
	\end{equation}
	where $\alpha_j\in[0,2]$ and $\gamma \in (0,1)$.
	Under the assumption that $\alpha_j=0$, the above condition together with the
	fact that $f \ge f_j$ (see Lemma \ref{lem:101}) implies that  $\phi(x_{j-1}^c) \ge \phi(x_j)$,
	and hence that $\phi(x_{j-1}^c) \ge \phi(x_j^c)$ in view of
	the way $x_j^c$ is defined in step 2 of RPB.
	In view of the latter inequality, condition \eqref{ineq:descent} with $\alpha_j =0$ is referred to as the descent condition,
	% 	since it implies that
	% 	$\phi(x_{j-1}^c) \ge \phi(x_j^c)$, 
	and proximal bundle variants based on
	it have been studied in \cite{du2017rate,kiwiel2000efficiency,ruszczynski2011nonlinear}.
	Moreover, the one with $\alpha_j \in (0,2]$ can viewed as a relaxation
	of the descent condition which does not necessarily imply
	monotonicity  of $\{\phi(x_j^c)\}$ but guarantees the pointwise convergence of $\{x_j^c\}$ and $\{x_j\}$.
	Proximal bundle variants based on this relaxed condition
	have been studied in \cite{bertsekas2015convex,urruty1996convex}
	for $\alpha_j=1$ and
	in \cite{de2014convex,van2014constrained} for 
	the more general case where $\alpha_j \in [0,2]$.
	Second, paper \cite{van2018incremental} proposes a proximal bundle variant where the serious/null decision policies are not necessarily mutually exclusive.
	% 	Instead of 
	% 	relying on a serious/null decision policy as RPB and the algorithms in the references above, 
	Third, as opposed to RPB and the algorithms in the above references
	which rely on serious/null decision policies,
	paper \cite{astorino2013nonmonotone} proposes a proximal bundle variant
	which allows the next prox-center $x_j^c$ to be a specific point
	in the line segment $[x_{j-1}^c,x_j]$.
	% 	determined by $ x_{j-1}^c $ and $ x_j $.
	% 	making the dichotomic serious/null decision as discussed in references above and RPB, paper \cite{astorino2013nonmonotone} proposed a proximal bundle variant that can take a continuous step for updating the prox-center $ x_{j-1}^c $ between the potential null and serious updates,
	% 	i.e., it always updates the prox-center as a point between $ x_{j-1}^c $ (the null update) and $ x_j $ (the serious update).
	% 	Finally, since all the selection rules for $x_j^c$ mentioned in this
	% 	paragraph do not rely on the iterate $\tx_j$
	% 	(which generally differs from $x_j$ and does not necessarily lie in $[x_{j-1}^c,x_j]$),
	% 	we believe that \eqref{ineq:hpe1}, which relies
	% 	on $\tx_j$, differs from all of them.
	Finally, the selection rule \eqref{ineq:hpe1} involving $\tx_j$
	differs from the other aforementioned selection rules
	for $x_j^c$ since they
	do not rely on $\tx_j$
	(which generally differs from $x_j$ and does not necessarily lie in $[x_{j-1}^c,x_j]$).
	
	We now add a few remarks about the RPB storage requirement.
	First, at the beginning of each iteration of
	RPB, it is assumed that the following information
	are available:
	1) the data $ \{ (x, f(x),f'(x)): x\in C_j \} $ of the model \eqref{def:fj} in order to solve \eqref{def:xj} for $x_j$
	in step 1; and
	2) the triple
	$(x_{j-1}^c, \tx_{j-1}, \hat z_{k-1})$
	where $k=k(j)$ (see the first remark in the second paragraph following RPB for the definition of $k(j)$).
	Second, $ \tx_j $ is updated in every iteration of RPB according to \eqref{def:txj}.
	Third, $ x_j^c $ and $ \hat z_k $ only change during a serious iteration, and are updated as described in step 2.a.
	Hence, the size of the storage requirement of RPB is directly proportional to the cardinality of the bundle set $C_j$.
	
	We end this subsection by briefly discussing bundle management policies
	in the context
	of proximal bundle methods
	have been investigated
	in many works (see for example
	in \cite{du2017rate,frangioni2002generalized,kiwiel2000efficiency,de2014convex,ruszczynski2011nonlinear,van2017probabilistic}).
	In the context of RPB,
	this means 
	the way the bundle set $ C_j $ is updated in both steps 2.a and 2.b.
	The following two paragraphs specifically comment on these updates.
	
	Consider first the
	(flexible)
	rule imposed on the next bundle set $C_{j+1}$ relative to the
	current bundle set $C_j$ in the execution of step 2.b
	when a null iteration happens.
	First, this rule
	has already been considered in \cite{du2017rate,frangioni2002generalized,kiwiel2000efficiency,ruszczynski2011nonlinear}. Second,
	it can be seen from \eqref{def:Cj} that $ x_j \in C_{j+1}$ holds for every $ j\ge 0 $, and it is shown in Lemma \ref{lem:simple}(d) below that $ x_j\notin C_j $ for every null iteration index $ j $. This
	remark together with \eqref{def:Cj} then imply that, in
	every null iteration,
	one new point $x_j\notin C_j $ is added to $ C_{j+1}$, while some of the
	points in $C_j\setminus A_j$ are possibly removed from it.
	
	Consider next the
	rule imposed on the next bundle set $C_{j+1}$  in the execution of
	step 2.a
	when a serious iteration happens.
	First, this rule has already been considered in \cite{de2014convex}.
	Second, this rule, which
	requires $C_{j+1}$ to satisfy $C_{j+1} \supset \{x_j\}$, allows for the possibility of completely refreshing the bundle set by setting it to
	$C_{j+1} = \{x_j\}$.
	% 	In this case, the cutting-plane model $ f_{j+1} $ becomes an affine function, and \eqref{def:xj} can be efficiently solved by a resolvent evaluation of $ h $.
	Third, if $ C_{j+1} $ is chosen as
	$ \{x_j\} $ at every serious iteration, then it follows from Theorem \ref{thm:suboptimal}(b) below that the size of any bundle set $C_j$ is always bounded by \eqref{cmplx:null}.
	Finally, since the prox bundle subproblem \eqref{def:xj} generally becomes harder to solve as the size of the bundle set $C_j$ grows,
	it might be convenient to choose $C_{j+1}$ as lean as possible, i.e., 
	$C_{j+1}=\{x_j\}$ if $j$ is a serious iteration index and $ C_{j+1}=A_j\cup \{x_j\} $ if $j$ is a null iteration index.

	\subsection{A general $\bar \varepsilon$-iteration complexity bound for RPB}\label{subsec:upper}
	
	% 	This subsection presents one of the main
	% 	complexity results for RPB
	% 	whose proof will be given in Subsection \ref{subsec:proof}.
	
	% More specifically, the result
	The following result, whose proof will be given in Subsection \ref{sec:serious},
	presents among other facts, a $\bar \varepsilon$-iteration complexity
	bound for RPB, which is
	a bound on the total number of
	(both serious and null) iterations performed by RPB
	until a $\bar \varepsilon$-solution is
	obtained.
	% 	It describes among other facts an iteration-complexity bound
	% 	for RPB to obtain
	% 	a $ \bar \varepsilon $-solution of \eqref{eq:ProbIntro}.
	% 	The following iteration-complexity result
	% 	describes the complexity of RPB to obtain a
	% 	$\bar \varepsilon$-solution of \eqref{eq:ProbIntro}.
	The iterate used to obtain such a solution
	is $\hat z_k$ which, as already mentioned
	in the second paragrpah following RPB,
	is
	the best (in terms of $\phi$)
	auxiliary serious iterate $\tx_j$ generated up to and including the
	$k$-th serious iteration.
	The use of this iterate as a candidate to obtain a $\bar \varepsilon$-solution plays a fundamental role in the complexity analysis of RPB and
	clearly differs from
	the iteration-complexity analysis of 
	\cite{du2017rate,kiwiel2000efficiency}
	which are based on the best (in terms of $\phi$) serious iterate $x_j$
	(instead of $\tilde x_j$ as above)
	generated up to and including the
	$k$-th serious iteration.

	\begin{theorem}\label{thm:suboptimal}
		Assume that $(f,f';h)$ satisfies (A1)-(A4) for some
		$(M_f,M_h,\mu) \in \R_+ \times [0,\infty] \times \R_+$. 
		% Let $(x_0,\lam,\bar \varepsilon) \in \dom h \times \R_{++} \times \R_{++}$ be given and set $\delta=\bar \varepsilon/2$ and
		% 		$M=M_f+M_h$.
		Then, for any given
		$(x_0,\lam,\bar \varepsilon) \in \dom h \times \R_{++} \times \R_{++}$,
		the following statements about RPB$(x_0,\lam,\delta)$ with
		$\delta=\bar \varepsilon/2$ hold:
		\begin{itemize}
			\item[a)] the number of serious iterations performed
			%		generated by the RPB method with $\delta=\bar \varepsilon/4$ until 
			until it obtains a best auxiliary serious iterate $\hat z_k$ such that $ \phi(\hat z_k)-\phi^*\le \bar \varepsilon $ is bounded by
			% 			\begin{equation}\label{cmplx:serious-strong}
			%                \[
			%				\mathcal{O}_1\left( \min \left\lbrace \frac{d_0^2}{\lam \bar \varepsilon}, \frac{1+\lam \mu}{\lam\mu} \log_1\left(  \frac{\mu d_0^2}{\bar \varepsilon} \right) \right\rbrace \right);
			%				\]
				\[
				\min \left\lbrace \frac{d_0^2}{\lam \bar \varepsilon}, \frac{1}{\tilde\lam \mu} \log\left(  \frac{\mu d_0^2}{\bar \varepsilon} + 1\right) \right\rbrace + 1
				\]
			where 
			\begin{equation}\label{def:tlam}
				\tilde \lam:=\frac{\lam}{1+\lam \mu};
			\end{equation}
			% 			\end{equation}
			% 			where $ \log_1(\cdot) $ is as in Subsection~\ref{subsec:DefNot};
			%			\item[b)]
			%			the number of serious iterations performed until it obtains a best auxiliary serious  iterate $\hat z_k$ such that $ \phi(\hat z_k)-\phi^*\le \bar \varepsilon $
			%			is bounded by
			%			$ 1 + d_0^2/(\lam \bar \varepsilon) $;
			\item[b)]
			if $\ell_0$ denotes an
			arbitrary serious  iteration index and
			all
			the auxiliary serious iterates $\hat z_k$ generated up to and including
			the $\ell_0$-th iteration satisfy $\phi(\hat z_k)-\phi^* > \bar \varepsilon$,
			then the next serious iteration index $\ell_1 > \ell_0$ occurs  and satisfies
			\begin{equation}\label{cmplx:null}
				\ell_1 - \ell_0\le \frac{\min  \left\lbrace 2 (16)^{4/3} \lam M M_f, 2 (16)^{4/3} \tilde \lam M_f^2 + 40\sqrt{2} M_f d_0\right\rbrace}{\bar \varepsilon} +1;
			\end{equation}
			\item[c)]
			the total number of iterations performed until it
			obtains
			an auxiliary serious iterate $\hat z_k$
			such that $ \phi(\hat z_k)-\phi^*\le \bar \varepsilon $
			is bounded by
			%			\begin{equation}\label{cmplx:total-strong}
			%				\mathcal{O}\left( \left( \frac{M_f \min\{\lam M, \lam M_f+ d_0\}}{\bar \varepsilon} + 1\right) \left[ \min \left\lbrace \frac{d_0^2}{\lam \bar \varepsilon}, \frac{1+\lam \mu}{\lam\mu} \log_1\left(  \frac{\mu d_0^2}{\bar \varepsilon} \right) \right\rbrace  + 1\right] \right)
			%			\end{equation}
				\begin{equation}\label{cmplx:total-strong}
					\left( \frac{ 2 (16)^{4/3} M_f \min\{\lam M, \tilde \lam M_f+ d_0\}}{\bar \varepsilon} + 1\right) \left[ \min \left\lbrace \frac{d_0^2}{\lam \bar \varepsilon}, \frac{1}{\tilde\lam\mu} \log\left(  \frac{\mu d_0^2}{\bar \varepsilon} +1\right) \right\rbrace  + 1\right] 
				\end{equation}
		\end{itemize}
		where $\phi^*$ and $d_0$ are as in \eqref{eq:ProbIntro} and \eqref{def:d0}, respectively, and $M=M_f+M_h$.
	\end{theorem}
	
% 	\[
% 	\frac{\lam M_f}{1+\lam \mu}=\tilde \lam M_f \ge  d_0 
% 	\]
	
% 	\[
% 	\lam(M_f-\mu d_0)\ge d_0
% 	\]
	
	We now make some comments about Theorem \ref{thm:suboptimal}.
	First, the behavior of RPB clearly depends on the choice of the
	prox stepsize $\lam$ in its step 0.
	More specifically, as $\lambda$ decreases, Theorem \ref{thm:suboptimal}(a) implies that
	the total number of serious iteration indices increases,
	while Theorem \ref{thm:suboptimal}(b)
	implies that bound \eqref{cmplx:null} on the number of null iterations
	between any two consecutive serious iterations decreases.
	% 	Second, the complexity bound \eqref{cmplx:total-strong} for RPB is established with respect to the best auxiliary serious iterate $ \hat z_k $ (see \eqref{def:hat zk} and the third remark in the second paragraph following the statement of RPB).
	% 	This is in contrast to
	% 	the iteration-complexity analysis of 
	% 	\cite{du2017rate,kiwiel2000efficiency}, which establish complexity bounds with respect
	% 	to the best serious iterate $x_j$.
	Second, in the unusual case where
	$C M_f d_0/\bar \varepsilon \le 1$ for a given universal constant $C>0$, it can be easily seen
	that \eqref{cmplx:total-strong} reduces to
	${\cal O}([\kappa + C^{-1} + 1][C^{-2}\kappa^{-1}+1])$
	where $\kappa := \lam M_f^2/\bar \varepsilon$. Hence, the
	$\bar \varepsilon$-iteration complexity
	bound of
	RPB reduces to ${\cal O}((1+C^{-1})^2)$ when
	$\lam$ is chosen as $\lam = \bar \varepsilon/(C M_f^2)$.

	We now make a few remarks about some of the input
	required by the CS-CS and RPB
	methods as well as the assumptions made to obtain iteration-complexity bounds
	for them. 
	First, none of the two methods requires the
	availability of a Lipschitz constant $M_h$ as in (A4).
	Second, while the CS-CS method
	uses $ M_f $ as input, RPB has the advantage of not needing it.
	Third, iteration-complexity bounds for both of them have been established for any
	choice of initial point $x_0 \in \dom h$
	and regardless
	of whether $M_h$ is finite or not
	(see Theorem \ref{thm:suboptimal}(c) and Proposition \ref{prop:sub-new}).
	Fourth, complexity bound \eqref{eq:bound} for
	the CS-CS method,
	and the one for RPB implied by
	\eqref{cmplx:total-strong}
	where $\min\{\lam M, \tilde \lam M_f+ d_0\}$ is replaced by $\tilde \lam M_f+d_0$,
	do not depend on $ M_h $.
	
	Under some reasonable conditions on
	 the triple $(\mu,M_f,M_h)$, 
	the next two results describe ranges on the prox stepsize
	$\lam$ which guarantee that the $\bar \varepsilon$-iteration complexity \eqref{cmplx:total-strong} of RPB reduces to
	that of the CS-CS method, namely \eqref{eq:bound}.
	The first result covers the strongly convex case
	where $\mu$ is not too large and allows $M_h$ to be arbitrary.
	
	\begin{corollary}\label{thm:bound-strg}
		Let universal constants $C, C' >0$ be given and consider an instance $(x_0,(f,f';h))$ of \eqref{eq:ProbIntro}
		which satisfies (A1)-(A4) with parameter triple $(M_f,M_h,\mu)$ such that
		\begin{equation}\label{ineq:assumption1}
		    %0< \mu \le \frac{C' M_f}{d_0}, \qquad 
		\frac{C M_f d_0}{\bar \varepsilon}\ge 1, \qquad  M_h \in [0,+\infty], \qquad 0\le \mu \le \frac{C' M_f}{d_0}.
		\end{equation}
		Then, RPB$(x_0, \lam, \bar \varepsilon/2)$ with any $\lam$ lying in the (nonempty) interval
		\begin{equation}\label{ineq:lam2}
			\frac{d_0}{M_f}
			\le \lam \le   \frac{C d_0^2}{\bar \varepsilon}
% 			\min\left \lbrace\frac{ d_0^2}{\bar \varepsilon}, \frac{1}{\mu} \right\rbrace
		\end{equation}
		has $\bar \varepsilon$-iteration complexity bound given by \eqref{eq:bound}.
	\end{corollary}
	
% 	\[
% 		{\cal O}_1\left(\min \left\lbrace \frac{M_f^2 d_0^2}{\bar \varepsilon^2}, \left( \frac{M_f^2}{\mu \bar \varepsilon} +\frac{\lam M_f^2}{\bar \varepsilon} \right) \log\left( \frac{\mu d_0^2}{ \bar \varepsilon} + 1 \right) \right\rbrace  \right)
% 	\]
	
% 		\red{Note that the two assumptions imply that
% 	\[
% 	\frac{M_f^2}{\mu} \ge M_f d_0 \ge \bar \varepsilon
% 	\]	
% 	so 
% 	\[
% 	\frac{\bar \varepsilon}{4M_f^2} \le \frac{1}{\mu}
% 	\]
% 	}
	
	\begin{proof}
% 		For shortness, 
% 		RPB$ (x_0, \lam, \bar \varepsilon/2) $ is referred below to as RPB.
		% 		Without loss of generality, we assume that $C=1$ and $C'=1$.
		First, the conclusion that the set of $\lambda$ satisfying \eqref{ineq:lam2} is nonempty follows directly from the first inequality in \eqref{ineq:assumption1}.
%		Next, we show the remaining part of the lemma.
% 		---------------------------
% 		\[
% 		(\tilde C+1) \lam M_f = C' \tilde \lam M_f \ge d_0
% 		\]
% 		\[
% 		C' \frac{\lam}{1+\lam \mu} M_f \ge d_0
% 		\]
% 		\[
% 		C'\lam M_f \ge d_0 + \lam\mu d_0
% 		\]
% 		\[
% 		\frac{C'M_f}{d_0} \ge  \frac1\lam + \mu
% 		\]
% 		Write $C'=C'_1+C'_2$.
% 		Then, if
% 		\[
% 		\mu \le C'_1 \frac{M_f}{d_0} \quad \frac1\lam \le C'_2 \frac{M_f}{d_0}
% 		\]
% 		A sufficient condition is
% 		\[
% 		\frac{C'M_f}{d_0} \ge \mu + \frac{M_f}{d_0}
% 		\]
% 		or
% 		\[
% 		\tilde C \frac{M_f}{d_0} =
% 		(C'-1) \frac{M_f}{d_0} \ge \mu 
% 		\]
% 		-----------------------------
	 Now, the assumption that $ C' M_f/d_0 \ge \mu $, the first inequality in \eqref{ineq:lam2}, and the definition of $\tilde \lam$ in \eqref{def:tlam},
	 imply that
		\[
		(C'+1) \frac{M_f}{d_0} \ge \mu + \frac1\lam = \frac{1}{\tilde \lam},
		\]
		and hence that
		\begin{equation}\label{ineq:tlam}
	        (C'+1)\tilde \lam M_f \ge d_0.   
	    \end{equation}
% 		\[
% 	    \lam M_f \ge (1-C') \lam M_f + \lam \mu d_0\ge (1-C') d_0 + \lam \mu d_0 > (1-C')(1+\lam \mu)d_0,
% 	    \]
% 	    which together with the definition of $\tilde \lam$ in \eqref{def:tlam} implies that
% 	    \begin{equation}\label{ineq:tlam}
% 	        \tilde \lam M_f = \frac{\lam M_f}{1+\lam \mu} \ge (1-C') d_0.   
% 	    \end{equation}
	    Defining
		\[
		a= \frac{\tilde \lam M_f^2}{\bar \varepsilon}, \quad 
		b= \min \left\lbrace \frac{d_0^2}{\lam \bar \varepsilon}, 
		\frac{1 }{\tilde \lam \mu} \log\left(  \frac{\mu d_0^2}{\bar \varepsilon} +1 \right) \right\rbrace,
		\]
	    and using Theorem \ref{thm:suboptimal}(c) and \eqref{ineq:tlam}, we conclude that
		$ {\cal O}((a+1)(b+1)) $ 
		is a $ \bar \varepsilon $-iteration complexity bound for RPB$ (x_0, \lam, \bar \varepsilon/2) $.
		Now, \eqref{ineq:tlam} and the first inequality in \eqref{ineq:assumption1} can be easily seen to imply that
		$ a \ge 1/[C(C'+1)]$.
		Moreover, it follows from the definition of $b$ and the second inequality in \eqref{ineq:lam2} that
		\begin{equation}\label{ineq:b}
		b\ge \min \left\lbrace \frac1C \, ,  \,
		\frac{1 }{\tilde \lam \mu} \log\left(  \frac{\lam \mu}{C} +1 \right) \right\rbrace.    
		\end{equation}
		Using the fact
		that $\log(1+t) \ge t/(1+t)$ for every $t>0$,
		we easily see that
		$\log(1+t) \ge t/2$ if
		$t \le 1$ and
		$\log(1+t) \ge \log 2>0$ if $t \ge 1$. 
		This observation with $t=\lam \mu/C$ and the definition of $\tilde \lam$ in \eqref{def:tlam} then imply that
		\[
		\frac{1}{\tilde \lam \mu} \log\left(  \frac{\lam \mu}{C} +1 \right) 
		\ge \min \left \lbrace \frac{\lam}{2\tilde \lam C} \, , \, \left( 1+ \frac{1}{\lam \mu} \right) \log2 \right \rbrace 
		\ge \min \left \lbrace \frac{1}{2 C} \, , \, \log2 \right \rbrace,
		\]
		and hence that
		$b \ge \min\{ 1/(2C), \log 2 \}$.
% 		=\Omega(d_0^2/(\lam \bar \varepsilon))
% 		It is clear to see that
% 		$ b=\Omega(1/\red{C}) $
% 		in view of 
% 		the second inequality in \eqref{ineq:lam2}.
        This inequality and the fact that $ a \ge 1/[C(C'+1)]$ imply that
		$ {\cal O}((a+1)(b+1)) $
		is equal to
		${\cal O}(ab+1)$.
		Using this observation, the definitions
		of $a$ and $b$, and the fact that $\tilde \lam \le \lam$,
		we then conclude that
		the bound 
		$ {\cal O}((a+1)(b+1)) $
		reduces to \eqref{eq:bound},
		and hence that the lemma holds.
	\end{proof}
	
	We now make a few remarks about an instance $(x_0,(f,f';h))$ which satisfies the
	assumptions of Corollary \ref{thm:bound-strg}.
	First, if $M_h=+\infty$ then it is possible to have $\dom h=\R^n$.
	Actually, if $\dom h$ is unbounded and $\mu>0$, then $M_h$ must be equal to $+\infty$.

	The second result covers the convex case (i.e., $\mu=0$)
	under the condition that the ratio $M_h/M_f$ is bounded
	and shows that \eqref{cmplx:total-strong} reduces to \eqref{eq:bound} for a larger range of $\lam$'s.

	\begin{corollary}\label{thm:bound-cvx}
		Let universal constants $C,C' >0$ be given and consider an instance $(x_0,(f,f';h))$ of \eqref{eq:ProbIntro}
		which satisfies (A1)-(A4) with parameter triple $(\mu,M_f,M_h)$ such that
		\begin{equation}\label{ineq:assumption2}
		    \frac{C M_f d_0}{\bar \varepsilon}\ge 1, \qquad M_h \le C' M_f, \qquad \mu=0.
		\end{equation}
		Then, RPB$(x_0, \lam, \bar \varepsilon/2)$ with any $\lam$ lying in the (nonempty) interval
		\begin{equation}\label{ineq:lam1}
			\frac{\bar \varepsilon}{C M_f^2} \le \lam \le \frac{C d_0^2}{\bar \varepsilon}
		\end{equation}
		% 			satisfying \eqref{ineq:lam1}
		has $\bar \varepsilon$-iteration complexity bound $ {\cal O}_1( M_f^2 d_0^2/ \bar \varepsilon^2)$, and hence agrees with \eqref{eq:bound}.
		% 			$q:= M_h/M_f$.
	\end{corollary}
	
	\begin{proof}
% 		For shortness, 
% 		RPB$ (x_0, \lam, \bar \varepsilon/2) $ is referred below to as RPB.
		First, the conclusion that the set of $\lambda$ satisfying \eqref{ineq:lam1} is nonempty follows immediately from the first inequality in \eqref{ineq:assumption2}.
		Moreover, it follows from the second inequality in \eqref{ineq:assumption2} and
		Theorem \ref{thm:suboptimal}(c) with $\mu=0$ that
		the $ \bar \varepsilon $-iteration complexity bound for RPB$ (x_0, \lam, \bar \varepsilon/2) $ is
		$ {\cal O}([1+\lam M_f^2/\bar \varepsilon ][1+d_0^2/(\lam \bar \varepsilon)]) $.
		Since \eqref{ineq:lam1} implies
		that $\max\{ \lam M_f^2/\bar \varepsilon \,,\, d_0^2/(\lam \bar \varepsilon ) \}  = {\cal O}( M_f^2d_0^2/\bar \varepsilon^2)$, we then conclude that 
		the previous bound reduces to $ {\cal O}_1(M_f^2 d_0^2/ \bar \varepsilon^2) $.
	\end{proof}
	
	We now make two remarks about an instance $(x_0,(f,f';h))$ which satisfies the
	assumptions of Corollary \ref{thm:bound-cvx}.
	First, if $h$ is an indicator function then $M_h=0$ and hence the
	second inequality in \eqref{ineq:assumption2} is trivially satisfied.
	Second, if $h$ is $\mu$-convex with $\mu>0$, then $\mu$ can not be large.
	Indeed, it can be easily  seen that
	$\mu \le 4 M_h/D_h$ where $D_h$ denotes the diameter of $\dom h$,
	and hence that $D_h$ is finite.
	
% 	Second, the instance class $\hat {\cal I}_0(M_f,R_0)$ consisting of
% 	$ (x_0,(f,f';h)) \in {\cal I}_0(M_f,R_0;C)$
% 	such that $h$ is the indicator function of a closed convex set
% 	clearly satisfies \eqref{incl:I1}.
% 	Third, an instance class ${\cal I}$ satisfying \eqref{incl:I1}
% 	% 	the class $ {\cal I}_0^u(M_f,R_0)$ (and hence $ \tilde {\cal I}_0(M_f,R_0)$)
% 	may contain instances for which
% 	$h$ is strongly convex but its strong convexity
% 	parameter must be small. This is due to the fact that, if
% 	$h$ is a $\mu$-strongly convex function,
% 	% 	in $ {\cal I}_0^u(M_f,R_0)$, 
% 	then it can be easily
% 	seen that
% 	$\mu \le 4 M_h/D_h$,
% 	and hence that $D_h$ is finite,  where $D_h$ denotes the diameter of $\dom h$.

	We now make some remarks about Corollary \ref{thm:bound-cvx} in light of
	Corollary \ref{thm:bound-strg}.
	First, range \eqref{ineq:lam1} is larger than range \eqref{ineq:lam2} since they both have the same right endpoints and the
	left endpoint of the first one is the geometric mean of
	the endpoints of the latter one.
	Second, Corollary \ref{thm:bound-strg} also holds
	when $\mu=0$ but, because it does not assume any condition
	on $M_h$, its conclusion is only guaranteed for
	a smaller range on $\lam$.
	
    We end this subsection by arguing that the first inequality
    in \eqref{ineq:assumption1} or \eqref{ineq:assumption2} is a mild assumption.
    Indeed, for those instances which violate 
	this inequality, i.e., it satisfies
	$ C M_f d_0/\bar \varepsilon \le 1$,
	the second remark in the paragraph following Theorem \ref{thm:suboptimal}
	implies that RPB with $ \lam=\bar \varepsilon/(C M_f^2) $ finds a $ \bar \varepsilon $-solution of \eqref{eq:ProbIntro} in $ {\cal O}((1+C^{-1})^2) $ iterations.
	Hence, instances that do not satisfy this inequality are trivial.

	\subsection{Complexity bounds for other proximal bundle variants} \label{subsec:compare}

	Papers \cite{du2017rate,kiwiel2000efficiency} study a
	proximal bundle variant
	% 	which also starts from some with initial point $x_0$ and prox stepsize $\lam$, which we refer
	% 	below to as PBV$(x_0,\lam)$, 
	for solving
	the set constrained problem
	\begin{equation}\label{eq:set}
	    \min \{ \tilde f(x): x \in X \}
	\end{equation}
	where $X$ is a nonempty closed convex set\footnotemark
	\footnotetext{Actually, \cite{du2017rate} only considers the case where $X=\R^n$.}
	and $\tilde f$ is a $\mu$-convex
	($\mu \ge 0$) finite everywhere function such that
	a first-order oracle
	$\tilde f':\R^n \to \R$ satisfying
	$\tilde f'(x)\in \partial \tilde f(x)$ for every $x\in \R^n$ is available.
	The method of \cite{du2017rate,kiwiel2000efficiency} starts from some
	$x_0 \in X$ and also uses a 
	constant prox stepsize $\lam$,
	and hence is referred to as PBV$(x_0,\lam)$
	below.
	If $\{x_j\}$ denotes
	the sequence of iterates
	generated by PBV$(x_0,\lam)$ and
	\begin{equation}\label{def:tD}
		\tilde D = \tilde D[\tilde f]:=\sup\{ d(x_j,X^*): j\ge 0 \}, \quad \tilde M = \tilde M [\tilde f]:= \sup\{\|\tilde f'(x_j)\|:j\ge 0\},
	\end{equation}
	then,
	under the assumption that
	the set $X^*$ of optimal solutions of
	the above problem is nonempty,
	\cite{kiwiel2000efficiency} shows
	that PBV$(x_0,\lam)$ has
	$\bar \varepsilon$-iteration complexity bound
	\begin{equation}\label{cmplx:kiwiel}
		{\cal O}_1 \left( \frac{\tilde M^2 \tilde D^4}{ \lam \bar \varepsilon^3} \right),
	\end{equation}
	for the case in which $\mu=0$,
	and
	\cite{du2017rate} shows that
	PBV$(x_0,\lam)$ has
	$\bar \varepsilon$-iteration complexity bound\footnotemark
	\footnotetext{Actually, bound \eqref{cmplx:rus} has been formally derived in \cite{grimmer2019general}, which corrects a small error in the one derived in \cite{du2017rate}.} 
	given by 
	\begin{equation}\label{cmplx:rus}
		{\cal O}_1 \left( \left[\frac{\tilde M^{2} \lam}{ \alpha^{2} \bar \varepsilon} \log_1^+ \left(\frac{1}{ \alpha^{2} }\right)+\frac{1}{ \alpha}\right]
		\log_1^+ \left(\frac{\tilde f(x_{0})- \tilde f^*}{ \alpha \bar\varepsilon}\right)
		+\frac{\tilde M^2\lam}{ \alpha \bar \varepsilon} \log_1^+  \left(\frac{\tilde M^2\lam}{ \alpha \bar \varepsilon}\right)
		\right) 
	\end{equation}
	for the case where $\mu>0$,
	where $\alpha:=\min \{\lam \mu,1\}$ and $ \log_1^+(\cdot) $ is defined in Subsection \ref{subsec:DefNot}.

	For the purpose of comparing the
	implication of the above bounds with the $\bar \varepsilon$-iteration complexity bounds established in Corollaries \ref{thm:bound-strg} and \ref{thm:bound-cvx},
	we restrict our attention to the unconstrained CNCO problem \eqref{eq:ProbIntro} where $ f $ satisfies (A1)-(A3), $ h \equiv \mu \|\cdot - x_0\|^2/2 $ and $ x_0 $ is the initial point.

	Clearly, such an unconstrained CNCO problem can be solved
	by applying PBV$(x_0,\lam)$ to \eqref{eq:set} with $\tilde f=f+h$ and
	$X=\R^n$ and with
	first-order oracle
	$\tilde f':=f'+\mu(\cdot-x_0)$.
	As a consequence, the
	$\bar \varepsilon$-iteration complexity bound of PBV$(x_0,\lam)$ for solving 
	the aforementioned unconstrained CNCO problem in the above manner
	is given by \eqref{cmplx:kiwiel} with
	$\tilde f=f+h$ if $\mu=0$ and \eqref{cmplx:rus} if $\mu>0$.

	We will now derive
	$\bar \varepsilon$-iteration complexity
	bounds for PBV$(x_0,\lam)$ in terms of $ M_f $, $ d_0 $, $ \lam $ and $ \bar \varepsilon $ for any $\mu \ge 0$.
	We first claim that,
	for some constant $C'' > \sqrt{2}$ determined by
	the input of PBV$(x_0,\lam)$, we have:
	\begin{itemize}
		\item[a)] $\tilde D \le
		\sup_{j \ge 0} \{ \|x_j-x_0^*\|\} \le C'' (d_0+ \lam \tilde M)$ where $x_0^*$ is as in the line below \eqref{def:d0}; %(see Lemma 4.1 of \cite{grimmer2019general}
		\item[b)] if $2C'' \lam \mu\le 1$, then $\tilde M \le 2[M_f+\mu (1+C'') d_0]$.
	\end{itemize}
	Indeed, a) is proved in Lemma 4.1 of \cite{grimmer2019general}. To prove b),
	first note that
	the definition of $ \tilde f'$, \eqref{def:d0}, the assumption that $ f $ satisfies (A3), and the triangle inequality, imply 
	that for every $ j\ge 0 $,
	\begin{align*}
		\|\tilde f'(x_j)\| \le \| f'(x_j)\|+ \mu \|x_j-x_0\| 
		\le M_f + \mu \|x_j-x_0\| \le M_f + \mu ( d_0+\|x_j-x_0^*\|),
	\end{align*}
	and hence that $ \tilde M \le M_f + \mu(d_0 + \sup_{j \ge 0} \|x_j-x_0^*\| ) $, due to the definition of $\tilde M$ in \eqref{def:tD}.
	This inequality together with the second inequality in a) implies that
	$\tilde M \le M_f+\mu d_0 + \mu C'' (d_0 + \lam\tilde M)$, and hence that
	b) holds.
	
	Now, using \eqref{cmplx:kiwiel}, \eqref{cmplx:rus}, 
	% $1 < \tau/\sqrt{2} = {\cal O}(1)$, 
	and statements a) and b) above, we conclude that PBV$(x_0,\lam)$ has
	$\bar \varepsilon$-iteration complexity bound
	\begin{equation}\label{cmplx:kiwiel-1}
		{\cal O}_1 \left( \frac{ M_f^2 (d_0 + \lam M_f)^4}{ \lam \bar \varepsilon^3} \right)
	\end{equation}
	if $ \mu=0 $,
	and
	%????? (up to logarithmic terms)
	\begin{equation}\label{cmplx:rus-1}
		{\cal O}_1 \left( \left[\frac{M_f^2}{\lam \mu^2\bar \varepsilon} + \frac{d_0^2}{\lam \bar \varepsilon}\right] 
		 \log_1^+ \left(\frac{1}{\lam \mu}\right) \log_1^+ \left(\frac{\tilde f(x_{0})- \tilde f^*}{ \lam \mu \bar\varepsilon}\right)
		\right)
	\end{equation}
	if $\mu>0$ and $2 C'' \lam \mu \le 1$.
	In summary, we have argued
	that bound \eqref{cmplx:kiwiel} (resp.,  \eqref{cmplx:rus})
	obtained in \cite{kiwiel2000efficiency} (resp., \cite{du2017rate}) yields the
	$\bar \varepsilon$-iteration complexity 
	bound \eqref{cmplx:kiwiel-1}
	(resp., \eqref{cmplx:rus-1})
	if $ \mu=0 $ (resp., if $\mu>0$).
	
	In the remaining part of this subsection,
	we compare the $\bar \varepsilon$-iteration complexity bounds \eqref{cmplx:kiwiel-1} and \eqref{cmplx:rus-1} established for PBV$(x_0,\lam)$ and those for RPB$(x_0, \lam, \bar \varepsilon/2)$ presented in Corollaries \ref{thm:bound-cvx} and \ref{thm:bound-strg}.
	We first discuss the case of bound \eqref{cmplx:kiwiel-1} under
	the same assumption
	made in Corollary \ref{thm:bound-cvx}, i.e., the inequality $C M_f d_0/\bar \varepsilon \ge 1$ holds.
	Note that the arithmetic-geometric mean
	inequality implies that
	\[
	d_0 + \lam M_f = \frac{d_0}{3} + \frac{d_0}{3} + \frac{d_0}{3} + \lam M_f \ge 4\left( \frac1{27} d_0^3 \lam  M_f\right) ^{1/4},
	\]
	and hence that \eqref{cmplx:kiwiel-1} is minorized by
	${\cal O}_1 (M_f^3 d_0^3/\bar \varepsilon^3)$, which in turn is
	minorized
	by ${\cal O}_1 (M_f^2 d_0^2/\bar \varepsilon^2)$ in view of above assumption. Moreover, if
	$M_f d_0/\bar \varepsilon$ is significantly larger than $1$,
	then it also
	follows from the above reasoning that, for any $\lam>0$,
	bound \eqref{cmplx:kiwiel-1}
	is much worse
	than the $\bar \varepsilon$-iteration complexity bound
	${\cal O}_1(M_f^2 d_0^2/\bar \varepsilon^2)$ established in Corollary \ref{thm:bound-cvx}.
	%	This conclusion strongly indicates
	%	(although does not formally prove)
	%	that PBV$(x_0,\lam)$ is not
	%	$\bar \varepsilon$-optimal for ${\cal I}^u_0(M_f,R_0)$
	%	for any $\lam >0$ and a
	%	large collection of
	%	triples $(M_f,R_0,\bar \varepsilon)$.

	We now discuss the case of bound \eqref{cmplx:rus-1} under the assumptions of Corollary \ref{thm:bound-strg}, i.e.,
	the two inequalities
	$C M_f d_0/\bar \varepsilon\ge 1$ and $C' M_f/d_0 \ge \mu$ hold.
	Since \eqref{cmplx:rus-1} was proved
	under the condition that
	$2C'' \lam \mu \le 1$,
	we also assume that this condition holds
	in this paragraph.
	The assumption that
	$ C' M_f/d_0 \ge \mu$ implies that \eqref{cmplx:rus-1}
	is equivalent to
	${\cal O}_1(M_f^2/(\lam \mu^2 \bar \varepsilon))$.
	In view of
	\eqref{eq:bound} and the fact
	that $2C'' \lam \mu \le 1$, the latter bound is and can only be as good
	as the bound in Corollary \ref{thm:bound-strg} (i.e., \eqref{eq:bound})
	when $\lam \mu$ is bounded away from zero and $\mu$ is not too small.
	On the other hand, it follows from Corollary \ref{thm:bound-strg} that the $\bar \varepsilon$-iteration complexity of
	RPB$(x_0,\lam,\bar \varepsilon/2)$ is given by \eqref{eq:bound}
	regardless of the sizes of
	 the quantities $(\lam \mu)^{-1}$ and $\mu^{-1}$ and this happens for
	 a reasonably large $\lam$-range which is independent of $\mu$.
	 Moreover, Corollary \ref{thm:bound-strg} does not assume the restrictive condition that $2C'' \lam \mu \le 1$.
	 
	 Finally, \cite{astorino2013nonmonotone} establishes an ${\cal O}_1( M_f^3 D^3/  \bar\varepsilon^3)$ $\bar \varepsilon$-iteration complexity bound for an alternative proximal bundle variant, where $D$ is the diameter of $X$.
	 Clearly, this bound is much worse than the bound established in Corollary \ref{thm:bound-cvx}.

	\section{Analysis of null iterations}\label{sec:null} 
	
	This section contains two subsections. The first one establishes a preliminary
	upper bound on the number
	of null iterations between two consecutive
	serious iterations. 
	The second one discusses the relationship between CS-CS and RPB,  and presents a result showing that the former one 
	can be viewed as a special instance of the latter
	one.
	% The purpose of this section is to establish Propositions \ref{prop:null} and \ref{prop:reduction}.
	
	\subsection{An upper bound on the number of consecutive null iterations}\label{subsec:pf}
	We assume throughout this subsection that
	$\ell_0$ 
	denotes an arbitrary
	serious iteration index (and hence it can be equal to zero) and
	$B(\ell_0) $ 
	denotes the set consisting of the next serious  iteration index $\ell_1$ (if any)
	and all null iteration
	indices between $\ell_0$ and $\ell_1$, i.e., $B\left(\ell_{0}\right)=\left\{\ell_{0}+1, \ldots, \ell_{1}\right\} $.
	%	In view of Proposition \ref{prop:null}, $\ell_1$ will indeed occur.
	%	In terms of this notation,
	%	we easily see that Proposition \ref{prop:null} is equivalent to
	%	show that $\ell_1$
	%	is such that  $\ell_1-\ell_0$ is bounded by \eqref{cmplx:null}, and hence
	%	that  $\ell_1$ will indeed occur.
	
	%	??????? Even though our interest in this subsection
	%	is to give the proof of
	%	Proposition \ref{prop:null} under the assumptions (A1)-(A4) given in Section~\ref{sec:main},
	%	we will establish a more general result
	%	(i.e., Proposition \ref{prop:null-strong} below), under a different set of assumptions
	%%	(A1)-(A3), (B1) and (B2),
	%	which will immediately imply Proposition \ref{prop:null}.
	%	Condition (B1) allows for the possibility that
	%	$f$ be either convex or strongly convex, while
	%	condition (B2) weakens condition (A4)
	%	in that it only requires $f$ to be Lipschitz continuous
	%	in a ball centered at $x_{\ell_0}$ and with
	%	radius depending on the prox stepsize $\lam$.
	
	%	In addition, we assume that the function $f_j$ and the active set $ A_j $ are now given as
	%	\begin{align}
	%	f_j(\cdot) &= \max \left\lbrace  f(x)+\inner{g(x)}{\cdot-x}+\frac{\mu}{2}\|\cdot-x\|^2: \, x \in C_j\right\rbrace, \label{new fj} \\
	%	A_j & =\left\lbrace x\in C_{j}: f(x)+\inner{g(x)}{x_j-x}+\frac{\mu}{2}\|x_j-x\|^2=f_j(x_j) \right\rbrace \label{new Aj}.
	%	\end{align}
	
	We start by making some simple observations that immediately
	follow from the description of RPB.
	For any $j \in B(\ell_0)$,
	it follows from the definition of $ x_j^c $ in step 2 of RPB that
	%\begin{align} \label{eq:obs1}
	% 	\begin{equation}\label{eq:obs}
	$x^c_{j-1}=x_{\ell_0}$,
	% 	\end{equation}
	%\quad \forall j \in B(\ell_0),
	%\end{align}
	and hence that
	\begin{align}
		\label{basic obs}
		\phi^\lam_j &= \phi + \frac1{2\lam} \|\cdot-x_{\ell_0}\|^2,\\
		\underline{\phi}_j^\lam(u) &= 
		f_j(u) + h(u) +\frac{1}{2\lam}\|u- x_{\ell_0} \|^2, \label{basic obs ii}
	\end{align}
	in view of the definitions of $\underline{\phi}_j^\lam$ and $\phi^\lam_j$ in \eqref{def:xj} and \eqref{def:philam}, respectively.
	Hence, it follows from the last identity and \eqref{def:xj} that
	% which,
	% together with the definitions of $ \phi_j^\lam $ and $ x_j $ in \eqref{def:philam} and \eqref{def:xj}, respectively,
	% and the definition of $ m_j $ in step 1 of RPB, imply
	\begin{equation}\label{xj}
		x_j=\underset{u \in \R^n} \argmin \left\lbrace f_j(u) + h(u) +\frac{1}{2\lam}\|u- x_{\ell_0} \|^2 \right\rbrace
		\quad \forall j \in B(\ell_0).
	\end{equation}
	
	%with the understanding that $x_j$ (resp., $m_j$) is the optimal solution (resp., value) of \eqref{def:xj'}. 
	
	We now make a few immediate observations that will be used in the analysis of this subsection.
	First, it follows from the above equation that
	\begin{equation}\label{obs1}
		% m_j = \underline{\phi}_j^\lam(x_j), \quad 
		\frac1\lam (x_{\ell_0}-x_j) \in
		\partial \left(f_j +h \right)(x_j).
	\end{equation}
	Second, since \eqref{basic obs} implies that the function
	$\phi^\lam_j$ remains the same  whenever $j \in B(\ell_0)$ and
	$\ell_0$ remains fixed  throughout
	the analysis of this section,
	we will
	simply denote the function $\phi_j^\lam$ for $j \in B(\ell_0)$ by $\phi^{\lam}$, i.e., 
	\begin{equation}\label{eq:simple}
		\phi^{\lam} = \phi^\lam_j \quad \forall j \in B(\ell_0).
	\end{equation}
	Third, in view of the definition of $ \tx_j $ in \eqref{def:txj} (see also the second remark in the second paragraph following RPB)
	and the above relation, it then follows that
	\begin{equation}
		\tx_j \in \Argmin \left\lbrace \phi^\lam(x) :
		x \in \{\tx_{\ell_0},x_{\ell_0+1},\ldots,x_j\}
		%\}\cup \{ x_i : i\in B(\ell_0), \, i \le j \} 
		\right\rbrace. \label{def:tx}
	\end{equation}
	Fourth, it directly follows from \eqref{xj} and \eqref{def:tx} that $\{x_j, \tx_j\}  \subset \dom h$.
	Fifth, $\ell_1$ is characterized as the first index $j>\ell_0$ satisfying
	condition \eqref{ineq:hpe1}.
	% \begin{equation}\label{ineq:inner}
	% t_j \le  \delta.
	% \end{equation}
	% 	{\color{red}Fifth, recall that we have defined $ t_j $ in \eqref{ineq:hpe1}.}
	Sixth, it will be shown below that
	the sequence $ \{t_j: j\in B(\ell_0) \} $,
	where $t_j$ is	defined in \eqref{ineq:hpe1}, is non-increasing (see Lemma \ref{lem:mj}(b)) and
	converges to zero with an ${\cal O}(1/j)$ convergence rate (see Proposition \ref{prop:rate}).
	
	%For the sake of generality, we state all the results below 
	%in terms of $\ell_0$. However,
	%we assume in their proofs that $\ell_0=0$ in order to keep
	%their notation and formulae simple.
	
	The following result describes some basic facts
	about the prox subproblem \eqref{def:mj*}
	and the prox bundle subproblem \eqref{def:xj}.
	
	\begin{lemma}\label{lem:101}
		For every $ j\in B(\ell_0) $, define
		\begin{equation}\label{def:mj*}
			m_j^*:=\min \left\lbrace \phi^\lam(u):u\in \R^n\right\rbrace 
		\end{equation}
		where $ \phi^\lam $ is as in \eqref{eq:simple}.
		Then, 
		%		The following statements hold
		for every $j \in B(\ell_0)$ and $ u \in \dom h$, we have
		\begin{equation}\label{ineq:prox}
			f(u)\ge f_j(u), \quad \phi^\lam(u) \ge \underline{\phi}_j^\lam(u), \quad \phi^\lam(u)\ge m_j^*\ge m_j.
		\end{equation}
		As a consequence, $t_j \ge \phi^\lam(\tx_j)- m_j^* \ge 0$ where $t_j$ is as in \eqref{ineq:hpe1}.
	\end{lemma}
	\begin{proof}
		It follows from the definition of $f_j$ in \eqref{def:fj} and (A1) that the first inequality in \eqref{ineq:prox} holds.
		This inequality, and relations \eqref{basic obs}, \eqref{basic obs ii} and \eqref{eq:simple}, 
		%		and the definitions of $ \phi_j^\lam $ and $ \underline{\phi}_j^\lam $ in \eqref{def:philam} and \eqref{def:xj}, respectively, 
		imply that the second inequality in \eqref{ineq:prox} holds.
		It follows from the definition of $m_j^*$ in \eqref{def:mj*} that $ \phi^\lam(u)\ge m_j^* $ for every $ u\in \dom h $.
		Using the second inequality in \eqref{ineq:prox}, and the definitions of $m_j$ and $m_j^*$ in step 1 in RPB and \eqref{def:mj*}, respectively, we have $ m_j^*\ge m_j $. 
		Moreover, it follows from the fact that $ \{\tx_j\}\subset \dom h $ (see the fourth remark below \eqref{def:tx}),
		the last  two inequalities in \eqref{ineq:prox} with $u=\tx_j$,  and the definition of $t_j$ in \eqref{ineq:hpe1} that $t_j\ge \phi^\lam(\tx_j)-m_j^*\ge 0$.
		%		a)
		%		It follows from the definition of $f_j$ in \eqref{def:fj} and the convexity of $f$ that the first inequality in (a) holds.
		%		This inequality, relation \eqref{eq:simple}, and the definitions of $ \phi_j^\lam $ and $ \underline{\phi}_j^\lam $ in \eqref{def:philam} and \eqref{def:xj}, respectively, imply that the second inequality in (a) holds.
		%		
		%		b) It follows from the second inequality in (a), and the definitions of $m_j$ and $m_j^*$ in step 1 in RPB and \eqref{def:mj*}, respectively, that the first inequality in (b) holds. The second inequality in (b) simply follows from the definition of $m_j^*$ in \eqref{def:mj*}.
		%		Moreover, it follows from the first two inequalities in (b) and the definition of $t_j$ in \eqref{ineq:hpe1} that $t_j\ge 0$.
		% 		c) This statement follows from the definitions of $m_j^*$ and $\phi_j^\lam$ in \eqref{def:mj*} and \eqref{def:philam}, respectively, and the fact that $j$ is a null iteration index if and only if $x_j^c=x_{j-1}^c$.
	\end{proof}
	
	% 	{\color{red}
	% 	Recall that it has been claimed in ... that $\tx_j$ is a $\delta$
	% 		 	In view of Lemma \ref{lem:101}(b) and the definition of $ t_j $ in \eqref{ineq:hpe1},
	% 		 	$ t_j $ is an upper bound on the optimality gap of \eqref{def:mj*}, i.e., $\phi_j^\lam(\tx_j) - m_j^* \le t_j$. 
	% 		 	If condition \eqref{ineq:hpe1} is satisfied, then the optimality gap $\phi_j^\lam(\tx_j) - m_j^*$ is bounded by $ \delta $ from above, and hence $\tx_j$ is a $\delta$-solution of \eqref{def:mj*}.
	
	The following technical result provides basic properties of RPB that are used in our
	analysis.
	
	\begin{lemma}\label{lem:simple}
		The following statements about the RPB method hold for every $j \in B(\ell_0)$:
		\begin{itemize}
			%			\item[a)]  $\{x_j, \tx_j\}  \subset \dom h$;
			\item[a)]
			for every $x \in C_j$, we have
			$f(x)= f_j(x)$;
			\item[b)] for every $ i \in B(\ell_0)$ such that $i<j $, we have
			$ \phi^\lam(\tx_j) \le \phi^\lam(\tx_i) $;
			\item[c)] $t_j \le f(x_j) - f_j(x_j)\le 2M_f\|x_j-x_{j-1}\| $;
			%			\le 2M_f\|x_j-x_{j-1}\|$; \red{move the second inequality}
			\item[d)]
			if $x_j \in C_j$ then
			$t_j=0$ and $j$ coincides with $\ell_1$ (i.e., the only serious iteration index
			in $B(\ell_0)$);
			\item[e)] $f_j$ is $M_f$-Lipschitz continuous on $\dom h$.
		\end{itemize}
	\end{lemma}
	\begin{proof}
		%		a) This statement directly follows from \eqref{xj} and \eqref{def:tx}.
		a) Let $x \in C_j$ be given.
		Using the first inequality in \eqref{ineq:prox},
		the assumption that $x \in C_j$, and the definition of $f_j$ in \eqref{def:fj},
		we conclude that
		$
		f \ge f_j \ge f(x) + \inner{f'(x)}{\cdot-x}
		$,
		and hence that
		$f(x) \ge f_j(x) \ge f(x)  + \inner{f'(x)}{x-x} = f(x)$.
		Thus, a) follows.
		
		b) This statement follows immediately from \eqref{def:tx}.
		
		c) Using the definitions of $ t_j $ and $ m_j $ in \eqref{ineq:hpe1} and step 1 of RPB, respectively, relations \eqref{basic obs ii}, \eqref{eq:simple} and \eqref{def:tx}, and the fact that $ \phi=f+h $, we have
		\[
		t_{j}
		=\phi^\lam(\tx_{j})-m_{j}
		\le \phi^\lam(x_{j})- \underline{\phi}_j^\lam(x_j) 
		=  f(x_{j})-f_{j}(x_{j}),
		\]
		and hence the first inequality in the statement holds.
		Next we show the second inequality in the statement.
		It follows from \eqref{def:Cj}
		with $j=j-1$ that $ x_{j-1}\in C_{j} $.
		This inclusion and the definition of $ f_j $ in \eqref{def:fj}  imply that
		\[
		f_j(\cdot) \ge f(x_{j-1}) + \inner{f'(x_{j-1})}{\cdot-x_{j-1}},
		\]
		and hence that
		\begin{align}
			f(x_j) - f_j(x_j)&\le f(x_j) - [ f(x_{j-1}) + \inner{f'(x_{j-1})}{x_{j} - x_{j-1}}] \nn \\
			&\le |f(x_j)-f(x_{j-1})|+\|f'(x_{j-1})\| \|x_j-x_{j-1}\| \label{ineq:fj}
		\end{align}
		where the second inequality is due to the triangle and the Cauchy-Schwarz inequalities.
		The second inequality in the statement now follows from
		(A3), \eqref{ineq:func}, the fact that $ \{x_j\}\subset \dom h $ (see the fourth remark below \eqref{def:tx}), and inequality \eqref{ineq:fj}.

		d) Assume that $ x_j\in C_j $. It then follows from statement a) with $x=x_j$ and the first inequality in statement c) that $ t_j\le 0 $.
		In view of Lemma \ref{lem:101}, we then conclude that $ t_j =0 $.
		In view of step 2 of RPB, this implies
		that $ j $ is a serious iteration index.
		Thus,  since
		$\ell_1$ is the only serious iteration index in $B(\ell_{0})$, we must
		have $ j=\ell_1 $.
		%		The fact that $ x_j\notin C_j $ can also be proved by induction: $ C_1=\{x_0\} $ and $ x_{j+1} \notin C_{j+1} $ by \eqref{def:Cj}.
		
		e)
		It follows from (A3),
		the definition of $f_j$ in \eqref{def:fj}, and a
		well-known formula for the subdifferential of the pointwise
		maximum of finitely many affine functions
		(e.g., see Example 3.4 of \cite{urruty1996convex1}) that 
		$f_j$ is $M_f$-Lipschitz continuous on $\dom h$.
	\end{proof}
	
	\vgap
	
	The following result gives a few useful properties about the relationship
	between the  active sets $\{A_j: j \in B(\ell_0)\}$ and the iterates
	$\{x_j: j \in B(\ell_0)\}$.
	
	\begin{lemma}\label{lem:Aj}
		Define
		\begin{equation}\label{def:fAj}
			f_{A_j}(\cdot) := \max \left\lbrace f(x)+\inner{ f'(x)}{\cdot-x}: x\in A_{j} \right\rbrace  \quad\forall j \in B(\ell_0)
		\end{equation}
		where $A_j$ is as in \eqref{def:A_j}.
		Then, the following statements hold for every $j \in B(\ell_0)$:
		\begin{itemize}
			\item[a)] $(f_{A_j}+h)(x_j) = (f_j+h)(x_j)$ and
			$\partial(f_{A_j}+h)(x_j)= \partial (f_j+h)(x_j)$;
			\item[b)] $f_{A_j} \le \min \{ f_j , f_{j+1}\}$;
			\item[c)] we have
			\begin{align}\label{def:new}
				x_j &= \underset{u\in \R^n}\argmin \left\{ (f_{A_j}+h)(u) + \frac1{2\lam} \|u-x_{\ell_0}\|^2 \right\}, \\ 
				m_j&=\underset{u\in \R^n}\min \left\{ (f_{A_j}+h)(u) + \frac1{2\lam} \|u-x_{\ell_0}\|^2 \right\} \nn
			\end{align}
			where $m_j$ is as in step 1 of RPB;
			\item[d)] for every $u \in \R^n$, we have
			\[
			(f_{A_j}+h)(u) + \frac1{2\lam} \|u-x_{\ell_0}\|^2 \ge m_j + \frac1{2 \tilde \lam} \|u-x_{j}\|^2
			\]
			where $\tilde \lam$ is as in \eqref{def:tlam}.
		\end{itemize}
	\end{lemma}
	\begin{proof}
		%	Recall that even though the lemma is stated for a general serious iteration index $\ell_0$, we assume in its proof that $\ell_0=0$ in order to keep the notation simple.
		%	
		a) The first conclusion immediately follows from the definitions of $ A_j $ and $ f_{A_j} $ in \eqref{def:A_j} and \eqref{def:fAj}, respectively. Using
		the definition of $A_j$ in \eqref{def:A_j},
		the definition of $f_j$ in \eqref{def:fj},
		and a well-known formula for the subdifferential of the pointwise maximum of finitely many convex functions (e.g., see Corollary 4.3.2 of \cite{urruty1996convex1}), we conclude that $ \partial f_j(x_j) $ is the convex hull of $\cup \{ f'(x): x \in A_j\} $.
		Using the same reasoning but with \eqref{def:fj} replaced by
		\eqref{def:fAj}, we conclude that the latter set is also the
		subdifferential of $f_{A_j}$ at $x_j$. Hence,
		statement a) follows. 
		% 		\red{revised} The first conclusion immediately follows from the definitions of $ A_j $ and $ f_{A_j} $ in \eqref{new Aj} and \eqref{def:fAj}, respectively. 
		% 		It follows from a) and Lemma \ref{lem:prime}(a) that
		% 		\begin{equation}\label{eq:equiv}
		% 		\partial f_{A_j}(\cdot)=\partial f'_{A_j}(\cdot)+\mu(\cdot-x_{\ell_0}), \quad \partial f_j(\cdot)=\partial f'_j(\cdot)+\mu(\cdot-x_{\ell_0}).
		% 		\end{equation}
		% 		Using Lemma \ref{lem:prime}(b),
		% 		the definition of $f'_j$ in \eqref{def:prime},
		% 		and a well-known formula for the subdifferential of the pointwise maximum of finitely many affine functions (e.g., see Example 3.4 of \cite{urruty1996convex1}), we conclude that $ \partial f'_j(x_j) $ is the convex hull of $\cup \{ g'(x) : x \in A_j\} $.
		% 		Using the same reasoning but with $ f'_j $ in \eqref{def:prime} replaced by $ f'_{A_j} $ in a), we conclude that the latter set is also the
		% 		subdifferential of $f'_{A_j}$ at $x_j$. Hence, the second conclusion of the statement follows in view of \eqref{eq:equiv}.
		
		b) 
		%    In view of \eqref{def:Cj} and {\color{red}\eqref{def:A_j}}, it is clear that $ A_j \subset C_{j+1} $ and $ A_j \subset C_j $, and hence that this statement follows from the definitions of $ f_j $ and $ f_{A_j} $ in \eqref{def:fj} and \eqref{def:fAj}, respectively.
		Note that $ A_j \subset C_j $
		due to \eqref{def:A_j}. Also, it follows from
		rule \eqref{def:Cj} regarding the choice of $C_{j+1}$ that
		$ A_j \subset C_{j+1} $. Hence,
		the definitions of $ f_j $ and $ f_{A_j} $ in \eqref{def:fj} and \eqref{def:fAj}, respectively,
		imply that $f_{j+1} \ge f_{A_j}$ and
		$f_{j} \ge f_{A_j}$.
		Thus, (b) holds.
		
		c) It follows from \eqref{obs1} and the second identity in a) that
		\[
		\frac{1}{\lam}(x_{\ell_0}-x_j) \in \partial (f_j+h)(x_j)
		=\partial (f_{A_j}+h)(x_j).
		\]
		Using the definition of $m_j$ in step 1 of RPB, \eqref{basic obs ii}, the first identity in a), and the fact that
		the above inclusion implies that $x_j$ satisfies the
		optimality condition of \eqref{def:new}, we conclude that c) holds.
		% The optimality condition of \eqref{def:xj'} and the second conclusion of a) imply that 
		% \[
		% (x_0-x_j)/\lam \in \partial f_{A_j}(x_j) + \partial h(x_j)
		% =\partial f_{j}(x_j) + \partial h(x_j).
		% \]
		% where the last equality is due to a). Using a),
		% the definition of $m_j$ in
		% \eqref{def:xj'}, and the fact that
		% the above inclusion implies that $x_j$ satisfies the
		% optimality condition of \eqref{def:new}, we conclude that c) holds.
		
		d) This statement follows immediately from c),  the definition of $\tilde \lam$ in \eqref{def:tlam}, the fact
		that the objective function of \eqref{def:new} is $(\mu+1/\lam)$-strongly convex and Theorem 5.25(b) of \cite{beck2017first} with $ f=f_{A_j}+h+\|\cdot-x_{\ell_0}\|^2/(2\lam) $, $ x^*=x_j $ and $ \sigma=\mu + 1/\lam $. %Lemma \ref{lem:strong}.
	\end{proof}

	The following lemma provides a bound on
	$\|x_j-x_{\ell_0}\|$ for $j \in B(\ell_0)$.

	\begin{lemma}\label{lem:ell0}
		Let $M=M_f+M_h$ and define $d_{\ell_0}:=\|x_{\ell_0}-x_0^*\| $ where $x_0^*$ is as in the line below \eqref{def:d0}. Then, the following statements hold:
		\begin{itemize}
			\item[a)]
			$\|x_j-x_{\ell_0}\| \le \lam M$ for every $j \in B(\ell_0)$; %where $ \red{????? }M=M_f+M_h $;
			\item[b)]
			if $ j \in B(\ell_0) $ is such that
			the bundle set $C_j$ contains
			$x_{\ell_0}$, then
			$\|x_j-x_{\ell_0}\| \le 2 d_{\ell_0} + 2\tilde \lam M_f $;
			% 	    \[
			% 	\|x_j-x_{\ell_0}\| \le \min\{\lam M, 2 d_{\ell_0} + 2\lam M_f  \}
			% 	\]
			% 	where $d_{\ell_0}:=\|x_{\ell_0}-x_0^*\| $;
			\item[c)]
			$\|x_{\ell_0+1}-x_{\ell_0}\| \le \min\{\lam M, 2 d_{\ell_0} + 2 \tilde \lam M_f  \}$.
		\end{itemize}
		% 	Let $ j \in B(\ell_0) $ be given and
		% 	assume that the bundle set $C_j$ contains $x_{\ell_0}$. Then,
		% 	\[
		% 	\|x_j-x_{\ell_0}\| \le \min\{\lam M, 2 d_{\ell_0} + 2\lam M_f  \}
		% 	\]
		% 	where $d_{\ell_0}:=\|x_{\ell_0}-x_0^*\| $.
		% 	As a consequence, $\|x_{\ell_0+1}-x_{\ell_0}\| \le \min\{\lam M, 2 d_{\ell_0} + 2\lam M_f  \}$.
	\end{lemma}
	
	\begin{proof}
		% 		a)
		% 		It follows from (A3),
		% 		the definition of $f_j$ in \eqref{def:fj}, and a
		% 		well-known formula for the subdifferential of the pointwise
		% 		maximum of finitely many affine functions
		% 		(e.g., see Example 3.4 of \cite{urruty1996convex1}) that 
		% 		$f_j$ is $M_f$-Lipschitz continuous on $\dom h$.
		a)  
		%		Let $j \in B(\ell_0)$ be given.
		%		Using the fact that the objective function of \eqref{xj} is $ (1/\lam) $-strongly convex and Theorem 5.25(b) of \cite{beck2017first}, 
		Using Lemma \ref{lem:Aj}(b) and (d), and the definitions of $ m_j $ and in step 1 of RPB and \eqref{def:tlam}, respectively
		we conclude that for every $ u \in \dom h $,
		\begin{equation}\label{ineq:basic}
			\frac1{2 \lam} \|u-x_j\|^2 + (f_j+h)(x_j) + \frac{1}{2\tilde \lam}\|x_j-x_{\ell_0}\|^2
			\le (f_j+h)(u)+ \frac{1}{2\lam}\|u-x_{\ell_0}\|^2,
		\end{equation}
		which upon setting $ u=x_{\ell_0} $ yields 
		\[
		\frac{1}{\lam}\|x_j-x_{\ell_0}\|^2\le (f_j+h)(x_{\ell_0}) - (f_j+h)(x_j)
		\le M\|x_{\ell_0}-x_j\|
		\]
		where the last inequality is due to
		Lemma \ref{lem:simple}(e) and (A4).
		Hence, (a) follows.
		% 		\begin{equation}\label{ineq:xj}
		% 		$\|x_{\ell_0}-x_j\| \le \lam M$.
		% 		\end{equation}
		
		b) It follows from \eqref{ineq:basic} with $u=x_0^*$,
		the fact that $(f_j+h) (x_0^*) \le \phi(x_0^*) = \phi^*$, and the definition of $d_{\ell_0}$, that
		\[
		\frac1{2 \lam}  \|x_0^*-x_j\|^2 + \phi(x_j) - \phi^* 
		\le \frac{d_{\ell_0}^2}{2\lam}  +f(x_j) - f_j(x_j) - \frac{1}{2\tilde \lam} \|x_j-x_{\ell_0}\|^2.
		\]
		Using the above inequality, 
		Theorem 5.25(b) of \cite{beck2017first} with $(f, x^*, \sigma)=(\phi, x_0^*, \mu)$, and the definition of $\tilde \lam$ in \eqref{def:tlam}, we have
		\begin{equation}\label{ineq:new}
		\frac1{2 \tilde \lam}  \|x_0^*-x_j\|^2 
		= \frac1{2 \lam}  \|x_0^*-x_j\|^2 + \frac{\mu}{2}  \|x_0^*-x_j\|^2
		\le \frac{d_{\ell_0}^2}{2\lam}  +f(x_j) - f_j(x_j) - \frac{1}{2\tilde \lam} \|x_j-x_{\ell_0}\|^2.    
		\end{equation}
		Now, using the assumption that $ x_{\ell_0} \in C_j$ and an argument similar to one in
		the proof of Lemma \ref{lem:simple}(c), we can see that
		$ f(x_j)-f_j(x_j) \le 2M_f\|x_j-x_{\ell_0}\| $. This conclusions and \eqref{ineq:new} impliy that
		% 		Using this inequality and the first inequality in \eqref{ineq:prox}, and 
		% 		rearranging the terms in \eqref{ineq:basic}, we obtain
		% 		\begin{align*}
		% 		\frac1{2 \lam}  \|u-x_j\|^2 + \phi(x_j) - \phi(u) 
		% 		&\le \frac{1}{2\lam} \|u-x_{\ell_0}\|^2  +f(x_j) - f_j(x_j) - \frac{1}{2\lam} \|x_j-x_{\ell_0}\|^2\\
		% 		&\le \frac{1}{2\lam} \|u-x_{\ell_0}\|^2  +2M_f\|x_j-x_{\ell_0}\| - \frac{1}{2\lam} \|x_j-x_{\ell_0}\|^2.
		% 		\end{align*}
		% 		The above  inequality with $ u=x_0^* $ and the fact that $ \phi(x_j)\ge \phi(x_0^*) $ imply that
		\begin{equation} \label{ineq:xj-x0}
			\frac1{2 \tilde \lam}  \|x_0^*-x_j\|^2
			\le \frac{d_{\ell_0}^2}{2\lam}  + 2M_f\|x_j-x_{\ell_0}\| - \frac{1}{2\tilde \lam} \|x_j-x_{\ell_0}\|^2
			\le \frac{d_{\ell_0}^2}{2\lam}  + 2\tilde \lam M_f^2
		\end{equation}
		%		
		%		------------------------------------
		%		
		%		\[
		%			\frac1{2 \lam}  \|x_0^*-x_j\|^2
		%			\le \frac{d_{\ell_0}^2}{2\lam}  + 2M_f\|x_j-x_{\ell_0}\|
		%			\le \frac{d_{\ell_0}^2}{2\lam}  + 2M_f\|x_j-x_0^*\| + 2 M_fd_{\ell_0}
		%		\]
		%		------------------------------------
		%		
		where the second inequality follows from the fact that $ a^2+b^2\ge 2ab $ for every $ a,b\in \R $.
		Since the triangle inequality and
		the definition of $d_{\ell_0}$ imply that $\|x_j-x_{\ell_0}\| \le \|x_j-x_0^*\| + d_{\ell_0} $, the conclusion of b)
		immediately follows from \eqref{ineq:xj-x0}, the last inequality and the fact that $\tilde \lam \le \lam$.
		
% 		------------An alternative argument-----------
		
% 		\[
% 		\|x_0^*-x_j\| \le \frac{d_{\ell_0}}{\sqrt{1+\lam \mu}} + 2\tilde \lam M_f
% 		\]
		
% 		\begin{align*}
% 		    \|x_{\ell_0}-x_j\| &\le d_{\ell_0} + \|x_0^*-x_j\| \\
% 		    &\le d_{\ell_0} + \frac{d_{\ell_0}}{\sqrt{1+\lam \mu}} + 2\tilde \lam M_f\\
% 		    &\le \frac{d_{\ell_0^{-}}}{\sqrt{1+\lam \mu}} + 2\tilde \lam M_f+ \frac{d_{\ell_0}}{\sqrt{1+\lam \mu}} + 2\tilde \lam M_f\\
% 		    &= \frac{d_{\ell_0^{-}} +d_{\ell_0}}{\sqrt{1+\lam \mu}} + 4\tilde \lam M_f
% 		\end{align*}
		
% 		Note that $d_{\ell_0}\le \sqrt{2} d_0$ and $d_{\ell_0^{-}}\le \sqrt{2} d_0$. Hence we can have the shrinking factor $1/\sqrt{1+\lam\mu}$ for $d_0$. However this argument does not hold for $\ell_0=0$, so we need to single out the first cycle.
		
% 		----------------------------------------------
		
		c) It is easy to see that \eqref{def:Cj} with $ j=\ell_0 $ implies that $ x_{\ell_0}\in C_{\ell_0+1} $. The conclusion of c) now follows from a) and b) with $ j=\ell_0+1 $.
	\end{proof}
	
	We now make some remarks about Lemma \ref{lem:ell0}.
	% 	Observe that, for the case in which
	First, while the bound in a) is meaningless
	when $M_h=\infty$,
	the one in b) is finite but it requires the mild condition that $C_j$ contain $x_{\ell_0}$.
	Second, the results in a) and b) can be used
	in conjunction with \eqref{ineq:zk-d0} to show that the whole
	RPB sequence $\{x_j:j\ge 0\}$ is bounded.
	Third, the complexity analysis of RPB does
	not make use of the last observation
	but only of the fact stated in Lemma \ref{lem:ell0}(c).

	The following lemma presents a few technical results about the set of
	scalars $\{t_j : j \in B(\ell_0) \}$ and plays an important role in the estimation of the cardinality of the set $B(\ell_0)$.
	
	\begin{lemma}\label{lem:mj}
		Consider the sequence
		$\{t_j\}$ as in \eqref{ineq:hpe1},
		and the sequences $\{m_j\}$ and $\{x_j\}$ as in step 1 of RPB.
		Then, the following statements hold:
		\begin{itemize}
			\item[a)]
			for every $ i, j \in B(\ell_0) $
			such that $  i<j $, we have
			\begin{equation}\label{ineq:lower}
				t_i \ge m_j- m_i\geq \frac{1}{2 \tilde \lam} \sum_{l=i+1}^{j}\left\|x_l-x_{l-1}\right\|^{2};
			\end{equation}
%			where $\tilde \lam$ is as in \eqref{def:tlam};
			\item[b)]
			$\{t_j : j \in B(\ell_0)\}$ is non-increasing;
			\item[c)]
			%			$ t_j \le 2M_f\|x_j-x_{j-1}\| $ and
			$ t_j \le 2 M_f \min\{\lam M, 2 d_{\ell_0} + 2\tilde \lam M_f  \} $ for every $j \in B(\ell_0)$.
			% 			{\color{red}where $ t_j $ is defined in \eqref{ineq:hpe1};}
			
			%{\color{red} changed} for every $ j \in B(\ell_0) $, we have
			% 			\begin{equation}\label{ineq:tj-disp}
			% 			t_{j}\le 2 M_f \|x_{j}-x_{j-1}\|;
			% 			\end{equation}
			% 			\item[c)] {\color{red} changed }
			% 			$\{t_j : j \in B(\ell_0)\}$ is non-increasing and 
			% $ t_j \le 2\lam M{\color{red} M_f } $ for every $j \in B(\ell_0)$.
		\end{itemize}
	\end{lemma}
	\begin{proof}
		%Recall that even though the lemma is stated for a general serious iteration index $\ell_0$, we assume in its proof that $\ell_0=0$ in order to keep the notation simple.
		%Using \eqref{def:xj'}, \eqref{ineq:below}, \eqref{ineq:mono}, 
		%and the assumption that $ i, j \in B(\ell_0) $ and $ j>i $,
		%we conclude that
		%\[
		%m_j=\underset{u\in \R^n}\min \underline{\phi}_j^\lam(u)\le \underset{u\in \R^n} \min \phi^\lam(u)\le \phi^\lam(\tx_j) \le \phi^\lam(\tx_i),
		%\]
		a) It follows from the last two inequalities in \eqref{ineq:prox} with $ u=\tx_j $
		%		Lemma \ref{lem:101}(b) 
		and Lemma \ref{lem:simple}(b) that
		\[
		m_j\le \phi^\lam(\tx_j) \le \phi^\lam(\tx_i),
		\]
		and hence that the first inequality in \eqref{ineq:lower} holds in view of the definition of $ t_i $ in \eqref{ineq:hpe1}. 
		Using the definition of $ m_{j+1} $ in step 1 of RPB, \eqref{basic obs ii}, and
		statements b) and d) with $u=x_{j+1}$ of Lemma \ref{lem:Aj},  we conclude that
		\begin{align*}
			m_{j+1}&=(f_{j+1}+h)(x_{j+1})+\frac{1}{2\lam}\|x_{j+1}-x_{\ell_0}\|^2 \nn \\
			&\ge (f_{A_j}+h)(x_{j+1})+\frac{1}{2\lam}\|x_{j+1}-x_{\ell_0}\|^2
			\ge m_j+\frac{1}{2\tilde \lam}\|x_{j+1}-x_j\|^2. 
		\end{align*}
		The second inequality in \eqref{ineq:lower} now
		follows by adding the above inequality from $ j=i $ to $ j=j-1 $,
		and simplifying the resulting inequality.
		
		b) It immediately follows from \eqref{ineq:lower} that $ \{m_j\}$ is non-decreasing, which together with Lemma \ref{lem:simple}(b) and the definition of $ t_j $ in \eqref{ineq:hpe1}, implies that $ \{t_j\} $ is non-increasing.
		
		c) It follows from Lemma \ref{lem:simple}(c) with $ j=\ell_0+1 $ and  Lemma \ref{lem:ell0}(c) that 
		\[
		t_{\ell_0+1} \le 2 M_f \min\{\lam M, 2d_{\ell_0} + 2\tilde \lam M_f  \}.
		\]  
		The statement now follows from b).
		% 		In view of statement b), in order to show statement c), it then suffices to show that $ t_1 \le 2\lam MM_f $, which follows from Lemma \ref{lem:simple}(d) with $ j=\ell_0+1 $ and Lemma \ref{lem:sub}(b).
		%		c) It immediately follows from \eqref{ineq:lower} that $ m_j\ge m_i $, which together with \eqref{ineq:mono} and the definition of $ t_j $ in \eqref{ineq:hpe1} implies that $ \{t_j\} $ is non-increasing.
		%		%	In order to show the second conclusion in b), it then suffices to show that $ t_1 \le 2\lam M^2 $.
		%		%	Indeed, it follows from the definitions of $ \tx_j $, $ \phi^\lam $ and $ m_j $ in \eqref{def:tx}, \eqref{eq:simple} and \eqref{def:xj'}, respectively, and the inequality in \eqref{rel:fj} that
		%		%	\begin{align*}
		%		%		\phi^\lam(\tx_{1}) &\le  (f+h)(x_{1})+\frac{1}{2\lam}\|x_{1}-x_0\|^2, \\
		%		%		m_1= (f_{1}+h)(x_{1})+\frac{1}{2\lam}\|x_{1}-x_{0}\|^2 &\ge 
		%		%		\ell_f(x_1;x_0)+ h(x_{1})+ \frac{1}{2\lam}\|x_{1}-x_{0}\|^2.
		%		%	\end{align*}
		%		%	Using the definition of $ t_j $ in \eqref{ineq:hpe1}, the above two relations and \eqref{ineq:ellf}, we have
		%		%	\[
		%		%	t_1\le f(x_{1}) - \ell_f(x_1;x_0) \le 2 M_f \|x_{1}-x_{0}\|.
		%		%	\]
		%		%	where the last inequality is due to \eqref{assumption:subgradient-bound} and \eqref{ineq:func}.
		%		In order to show the second conclusion in c), it then suffices to show that $ t_1 \le 2\lam MM_f $, which follows from statement b) with $ j=1 $ and Lemma \ref{lem:sub}(b).
	\end{proof}
	
	% The following technical result connects $ t_j $ to the minimum of distance between two consecutive null iterates up to the $ j $-th iteration, i.e., $ \Delta_j $ defined below, and it will in turn give a recursive formula in terms of $ t_j $ and $ t_{j/2} $, which will be useful in the proof of the next proposition.
	
	The following technical result relates $t_j$ and
	the minimum distance $\Delta_j$ between two consecutive iterates among
	$\{x_{\ell_0}, \ldots,x_j\}$, a quantity that plays an
	important role in the complexity analysis of the null iterations.

	\begin{lemma}\label{lem:tj}
		Let
		\begin{equation}\label{def:Delta_j}
			\Delta_j:=\min \left\{\left\|x_i-x_{i-1}\right\|: i \in B(\ell_0), \, i \le j \right\}, \quad \forall j\in B(\ell_0).
		\end{equation}
		Then, the following statements hold:
		\begin{itemize}
			\item[a)]
			for every $ j \in B(\ell_0) $, we have $ t_j \leq 2M_f \Delta_j $;
			%			\begin{equation}\label{ineq:tj-upper}
			%			t_j \leq 2M_f \Delta_j;%+\frac{1}{2\lambda} \Delta_j^2;
			%			\end{equation}
			\item[b)]
			for every $ j\in B(\ell_0) $ such $j \ge \ell_0+4$, we have
			\[
			\Delta_j^2\le \frac{32 \tilde \lam M_f}{(j-\ell_0)^2}  \sqrt{2 \tilde \lam \lceil (j-\ell_0)/2 \rceil t_{\ell_0+\lfloor (j-\ell_0)/2 \rfloor- 1}}% + t_{\lfloor j/2 \rfloor}.
			\]
			where $\tilde \lam$ is as in \eqref{def:tlam}.
		\end{itemize}
	\end{lemma}
	\begin{proof}
		%	Recall that even though the lemma is stated for a general serious iteration index $\ell_0$, we assume in its proof that $\ell_0=0$ in order to keep the notation simple.
		%	
		a)  Let $j \in B(\ell_0)$ and
		an arbitrary $i \in B(\ell_0)$
		such that $ i\le j $ be given.
		Using Lemma \ref{lem:mj}(b) and Lemma \ref{lem:simple}(c) with $ j=i $, we conclude that
		\[
		t_j\le t_i\le 2 M_f \|x_{i}-x_{i-1}\|.
		\]
		The statement now follows from the definition of $ \Delta_j  $ in \eqref{def:Delta_j} and the fact that the
		above inequality holds for every
		$i \in B(\ell_0)$
		such that $ i\le j $.
		%	It directly follows from Lemma \ref{lem:sub}(c)
		%	with $u=x_{j-1}$ and  Lemma \ref{lem:simple}(c)
		%	\[
		%	\phi^{\lambda}\left(x_{j-1}\right)-m_j \leq 2 M\left\|x_{j}-x_{j-1}\right\|+
		%	\frac{1}{2\lambda}\left\|x_{j}-x_{j-1}\right\|^{2}.
		%	\]
		%	{\color{red}
		%		\[
		%		\phi^{\lambda}(x_{j-1})-m_j \le 2 M_f \|x_{j-1}-x_{j-2}\| - \frac{1}{2\lam}\|x_j - x_{j-1}\|^2
		%		\quad j-1,j \in B(\ell_0)
		%		\]
		%	}
		
		%	\begin{equation}\label{ineq:tj}
		%	t_j \leq t_i \leq \phi^{\lambda}\left(x_{i-1}\right)-m_i \leq 2M\left\|x_i-x_{i-1}\right\|+\frac{1}{2\lambda}\left\|x_i-x_{i-1}\right\|^{2}.
		%	\end{equation}
		
		%	{\color{red}
		%		\[
		%		t_j\le t_i\le 2 M_f \|x_{i}-x_{i-1}\|, \quad i,j \in B(\ell_0), j\ge i
		%		\]
		%		\[
		%		t_j\le 2M_f \Delta_j
		%		\]
		%	}

		b) Let $j \in B(\ell_0)$ such that $j\ge \ell_0+4$ be given.
		For any
		$i \in B(\ell_0)$
		such that $ i < j $, it follows from
		Lemma \ref{lem:mj}(a), Lemma \ref{lem:simple}(c) with $ j=i $,
		and the definition of $ \Delta_j  $ in \eqref{def:Delta_j}, that
		\[
		\frac{1}{2 \tilde \lam}(j-i) \Delta_j^{2} \le \frac{1}{2 \tilde \lam} \sum_{l=i+1}^{j}\left\|x_l-x_{l-1}\right\|^{2} \leq t_i \le 2 M_f\left\|x_i-x_{i-1}\right\|.%+\frac{1}{2\lambda}\left\|x_i-x_{i-1}\right\|^{2}.
		\]
		Since the set of indices ${\cal I} := \{\ell_0+\lfloor (j-\ell_0)/2 \rfloor, \ldots,j-1\}$ is clearly in
		$\{i \in B(\ell_0): i < j\}$ and $|{\cal I}| = \lceil (j-\ell_0)/2 \rceil $, we conclude
		by adding the above inequality as $i$ varies in ${\cal I}$ that
		% 		{\color{red}$i=\ell_0+\lfloor (j-\ell_0)/2 \rfloor$} to $i=j-1$, and noting that all these $i$'s are
		% 		in $B(\ell_0)$, we obtain
		\begin{equation}\label{ineq:sum}
			\frac{(j-\ell_0)^2}{16 \tilde \lam} \Delta_j^2
			\le \frac{\lceil (j-\ell_0)/2 \rceil(\lceil (j-\ell_0)/2 \rceil +1)}{4 \tilde \lam} \Delta_j^{2}
			\leq 2 M_f\sum_{i \in {\cal I} }
			% 		i={\color{red}\ell_0+\lfloor (j-\ell_0)/2 \rfloor} }^{j-1}
			\left\|x_i-x_{i-1}\right\|.
			%+\frac{1}{2\lambda} \sum_{i=\lfloor j/2 \rfloor}^{j-1}\left\|x_i-x_{i-1}\right\|^{2}.
		\end{equation}
		On the other hand, using the fact that
		$j\ge \ell_0+4$ implies that $\ell_0+\lfloor (j-\ell_0)/2 \rfloor - 1\ge \ell_0+1$,
		the Cauchy-Schwarz inequality, and Lemma \ref{lem:mj}(a) with $(i,j)=(\ell_0+\lfloor (j-\ell_0)/2 \rfloor - 1,j-1)$, we conclude that
		\[
		\sum_{i\in{\cal I}}
		\left\|x_i-x_{i-1}\right\|
		\le \left \lceil \frac{j-\ell_0}{2} \right \rceil^{1/2}
		\left(  \sum_{i\in{\cal I}}\left\|x_i-x_{i-1}\right\|^{2}\right) ^{1/2} 
		\le \sqrt{2 \tilde \lam \lceil (j-\ell_0)/2 \rceil 
			t_{\ell_0+\lfloor (j-\ell_0)/2 \rfloor - 1}}.
		\]
		Statement (b) now follows by plugging the above inequality into \eqref{ineq:sum} and rearranging the resulting inequality.
	\end{proof}
	
	The following proposition shows that the sequence $ \{t_j: j\in B(\ell_0) \}$ converges to zero with an ${\cal O}(1/j)$ convergence rate.
	
	\begin{proposition}\label{prop:rate}
		For every $j \in B(\ell_0)$,
		we have
		\begin{equation}\label{ineq:claim}
			t_j\le \frac{ \min  \left\lbrace(16)^{4/3} \lam M M_f, (16)^{4/3} \tilde \lam M_f^2 + 20 M_f d_{\ell_0}\right\rbrace }{j-\ell_0}.
		\end{equation}
	\end{proposition}
	\begin{proof}
		%	Recall that even though the lemma is stated for a general serious iteration index $\ell_0$, we assume in its proof that $\ell_0=0$ in order to keep the notation simple.
		The proof of the proposition is by induction on $ j\in B(\ell_0) $.
		%	$k \in B(\ell_0)-\{\ell_0\}$ .
		% 			Clearly, \eqref{ineq:claim} is equivalent to
		% 			\[
		% 			\tilde t_k := t_{k+\ell_0} \le \frac{a}{k} \quad \forall k \in \{1,\ldots,\ell_1-\ell_0\}
		% 			\]
		First note that \eqref{ineq:claim} holds for every
		$j \in B(\ell_0)$ such that $j \le \ell_0+5$ in view of Lemma \ref{lem:mj}(c).
		Now, let $j \in B(\ell_0)$ be such that $  j\ge \ell_0+6 $ and assume
		%	and assume for simplicity and without any loss of generality
		%	that $j$ is an even number.
		for the induction argument that \eqref{ineq:claim} holds for the indices $\ell_0+1,\ldots,j-1$.
		Also, define $ a:=\min\{ (16)^{4/3} \lam M M_f, (16)^{4/3} \tilde \lam M_f^2 + 20 M_f d_{\ell_0}\}$.
		Since $\ell_0 + 1 \le \ell_0+\lfloor (j-\ell_0)/2\rfloor-1 \le j-1$ when $j \ge \ell_0+6$, we then conclude that
		% 	Hence, we have
		% 	$ t_{\lfloor j/2\rfloor}\le a/\lfloor j/2\rfloor $
		% 	since $\lfloor j/2\rfloor \le j-1$ when $j \ge 4$.
		% 	Using the induction hypothesis for $ \lfloor j/2\rfloor $, i.e., $ t_{\lfloor j/2\rfloor}\le a/\lfloor j/2\rfloor $, we have for $j\ge 3$,
		\[
		\left \lceil (j-\ell_0)/2 \right \rceil t_{\ell_0+\lfloor (j-\ell_0)/2\rfloor-1} \le \frac{\lceil (j-\ell_0)/2 \rceil}{\lfloor (j-\ell_0)/2 \rfloor -1} a
		\le 2a
		\]
		%	\[
		%	\left \lceil k/2 \right \rceil t_{\ell_0+ \lfloor k/2\rfloor-1} \le \frac{\lceil k/2 \rceil}{\lfloor k/2 \rfloor -1} a
		%	\le 2a
		%	\]
		where the last inequality is due to the assumption that $j \ge \ell_0+ 6$ and the definition of $a$.
		The last conclusion together with
		Lemma \ref{lem:tj}(b) then implies that
		\[
		\Delta_j^2\le \frac{64\tilde \lam M_f\sqrt{\tilde \lam a}}{(j-\ell_0)^2}.
		\]
		Now, using Lemma \ref{lem:tj}(a) and
		the last inequality, we then conclude that
		\[
		t_j \le 2 M_f \Delta_j\le \frac{16M_f^{3/2}\tilde \lam^{3/4}a^{1/4}}{j-\ell_0}
		\le \frac{a}{j-\ell_0}
		\]
		where the last inequality follows from the
		definition of $a$ and the fact that $\lam \ge \tilde \lam$ (see \eqref{def:tlam}).
		% 	Plugging the above inequality into \eqref{ineq:tj-upper}, we have for $j\ge 6$,
		%	\begin{align}\label{ineq:comb}
		%	t_j\le \frac{8\sqrt{3} (2)^{1/4} \lam^{3/4} M_f^{3/2} a^{1/4}}{j}.
		%%	t_j
		%%	\le \frac{8\sqrt{3}M^{3/2}\lam^{3/4}(2a)^{1/4}}{j} + \frac{8M\sqrt{3\lam a}}{j^{3/2}} + \frac{24M\sqrt{2\lam a}}{j^2} + \frac{24a}{j^3}.
		%	\end{align}	
		% 	\[
		% 	t_j\le \frac{16 M_f^{3/2} \lam^{3/4}  a^{1/4}}{j}.
		% 	\]
		%	{\color{red}
		%		\[
		%		t_j\le \frac{3^{1/2} 2^{13/4} \lam^{3/4} M_f^{3/2} a^{1/4}}{j}
		%		\]
		%		To have $ t_j\le a/j $, we need
		%		\[
		%		a\ge (2)^{13/3}( 3)^{2/3}  \lam M_f^2
		%		\]
		%	}
		%	Using the definition of $a$ in the statement of the proposition
		%	and the fact that $j \ge 12$, we can easily see that
		%	\begin{equation}\label{ineq:indv}
		%	\frac{8\sqrt{3}M^{3/2}\lam^{3/4}(2a)^{1/4}}{j} \le \frac{a}{4 j}, \quad \frac{8M\sqrt{3\lam a}}{j^{3/2}} \le \frac{a}{4 j}, \quad \frac{24M\sqrt{2\lam a}}{j^2} \le \frac{a}{4 j}, \quad \frac{24a}{j^3} \le \frac{a}{4 j}.
		%	\end{equation}
		% 	Using the above inequality, the definition of $a$ in the statement of the proposition and the fact that $ M\ge M_f $, we conclude that $ t_j\le a/j $.
		We have thus shown that the conclusion of the proposition holds.
	\end{proof}
	
	We are now ready to state the main result of this subsection.
	\begin{proposition}\label{prop:null-strong}
		Let $(x_0,\lam,\delta) \in \dom h \times \R_{++} \times \R_{++}$ be given
		and assume that
		(A1)-(A4) hold and
		$j=\ell_0$ is a serious iteration index of RPB$(x_0,\lam,\delta)$.
		Then, the next serious iteration index $j=\ell_1>\ell_0$ exists and satisfies
		% 		\begin{equation}\label{cmplx:null-strong}
		\[
		\ell_1 - \ell_0\le \frac{ \min  \left\lbrace(16)^{4/3} \lam M M_f, (16)^{4/3} \tilde \lam M_f^2 + 20 M_f d_{\ell_0}\right\rbrace }{\delta} + 1
		\]
		% 		\end{equation}
		%		where $\tilde \lam$ is as in \eqref{def:tlam}.
		where $M_f$ is as in (A3), and
		$M$ and $d_{\ell_0}$ are as in Lemma \ref{lem:ell0}.
	\end{proposition}
	
	\begin{proof}
		If $\ell_1=\ell_0+1$, then \eqref{cmplx:null} is obviously true.
		Assume then $\ell_1>\ell_0+1$.
		This clearly implies that $\ell_1-1 \in B(\ell_0)$,
		and hence is a null iteration index.
		Using this observation and the fact
		that an iteration index $ j $ is null if and only if \eqref{ineq:hpe1} does not hold,
		we conclude that $ t_{\ell_1-1}>\delta $.
		This conclusion,
		the fact that $ \ell_1-1\in B(\ell_0) $, and
		Proposition \ref{prop:rate} with $j=\ell_1-1$,
		then imply that
		\[
		\delta<t_{\ell_1-1} \le \frac{ \min  \left\lbrace(16)^{4/3} \lam M M_f, (16)^{4/3} \tilde \lam M_f^2 + 20 M_f d_{\ell_0}\right\rbrace }{\ell_1-1-\ell_0},
		\]
		from which the conclusion of the proposition immediately
		follows.
	\end{proof}

	%	 	We are now ready to give the proof of Proposition \ref{prop:null}.
	%	 	\vgap
	%	
	%	 	\noindent
	%	 	{\bf Proof of Proposition \ref{prop:null}}
	%	 	%	Proposition \ref{prop:rate}, together with
	%	 	%	the fact that $j$ is a serious iteration index if and only if
	%	 	%	\eqref{ineq:hpe1} holds,
	%	 	%	immediately implies the  conclusion of the proposition.
	%	 	%	
	%	 	If $\ell_1=\ell_0+1$, then \eqref{cmplx:null} is obviously true.
	%	 	Assume then $\ell_1>\ell_0+1$.
	%	 	This clearly implies that $\ell_1-1 \in B(\ell_0)$,
	%	 	and hence is a null iteration index.
	%	 	Using this observation and the fact
	%	 	that an iteration index $ j $ is null if and only if \eqref{ineq:hpe1} does not hold,
	%	 	we conclude that $ t_{\ell_1-1}>\delta $.
	%	 	This conclusion,
	%	 	the fact that $ \ell_1-1\in B(\ell_0) $, and
	%	 	Proposition \ref{prop:rate} with $j=\ell_1-1$,
	%	 	then imply that
	%	 	\[
	%	 	\delta<t_{\ell_1-1} \le \frac{(16)^{4/3}\red{\tilde \lam} MM_f}{\ell_1-1-\ell_0},
	%	 	\]
	%	 	from which the conclusion of the proposition immediately
	%	 	follows.
	%	 	\QEDA
	
	\subsection{Relationship between RPB and CS-CS}\label{subsec:CS}

	We start by making a few trivial remarks
	about the  relationship between
	CS-CS and RPB.
	First, if they use the same stepsize $\lam$,
	then they both generate the same first iterate $x_1$.
	Second, if $ d_0\le \bar \varepsilon/(4M_f) $, then it follows from Proposition \ref{prop:sub-new} that the CS-CS method, and hence RPB,
	with $ \lam=\bar \varepsilon/(4M_f^2) $ finds a $ \bar \varepsilon $-solution of \eqref{eq:ProbIntro} in one iteration.
	% 	This remark justifies the second remark in the second paragraph following Proposition \ref{prop:opt2}.
	
% 	\red{
% 		under the two assumptions of Theorem \ref{thm:bound-strg}, we have
% 		\[
% 		\lam = \min \left \lbrace \frac{\bar \varepsilon}{4M_f^2}, \frac{1}{\mu} \right \rbrace = \frac{\bar \varepsilon}{4M_f^2}
% 		\]
% 	}
	
	The following result describes a less trivial
	relationship between
	RPB and the CS-CS method.
	More specifically, 
	it shows that the first remark in the previous
	paragraph can be extended to the other iterates
	as well as long as $\lam$ is sufficiently small.
	
	% 	The following result shows that
	% 	RPB with any $\lam$ sufficiently small
	% 	reduces to CS-CS. %(with prox stepsize equal to $\lam$).
	
	% 	In view of Proposition \ref{prop:null}, it is natural to conjecture whether
	% 	RPB with small prox stepsize reduces to CS
	% 	with constant prox stepsize. The following result shows that this is indeed the case.
	
	\begin{proposition}\label{prop:reduction}
		Let $(x_0,\lam,\delta)\in \dom h \times \R_{++} \times \R_{++}$ satisfying $ \lam \le \delta/(2MM_f)$ (and hence $M_h < \infty$) be given, where $M$ is as in Lemma \ref{lem:ell0}. Then, every iteration index of RPB$(x_0,\lam,\delta)$ is a serious one.
		% 		If the prox stepsize $\lam>0$ input to RPB satisfies $ \lam \le \delta/(2MM_f)$ (and hence $M_h < \infty$)
		% 		where $ \delta  $ is as in step 0 of RPB and $M$ is as in Lemma \ref{lem:ell0}, then every iteration index is a serious one.
		As a consequence, if the set $ C_{j+1}$, which necessarily contains $x_j$,
		is always set to be  $\{x_j\} $ in step 2.a
		of RPB, then RPB$(x_0,\lam,\delta)$ reduces to  CS-CS$(x_0,\lam)$.% with constant prox stepsize $ \lam $.
	\end{proposition}
	
	\begin{proof}
		%As in the preceding proofs of this section,
		%we assume without loss of generality that $ \ell_0=0 $.
		Using Lemma \ref{lem:mj}(c) and the assumption that $ \lam\le\delta/(2M M_f) $, we have $ t_j\le 2\lam M M_f \le \delta $ for every $ j\in B(\ell_0) $.
		%		It follows from the assumption that $ \lam\le\delta/(2M M_f) $ and Lemma \ref{lem:mj}(c) that $ t_j\le \delta $ for every $ j\in B(\ell_0) $. 
		Hence, we have $ t_{\ell_0+1}\le \delta $, and in view of \eqref{ineq:hpe1}, we conclude that every iteration index $j$ is serious.
		We now show that, under the assumption of the proposition, RPB$(x_0,\lam,\delta)$ reduces to CS-CS$(x_0,\lam)$. % with constant prox stepsize $ \lam $.
		Since every iteration index is a serious one, using step 2.a of RPB, the definition of $ f_j $ in \eqref{def:fj},
		and the assumption of this proposition that $ C_{j+1}=\{x_j\} $, we conclude that $ x_j^c = x_j $
		and 
		$ f_j(\cdot) = f(x_{j-1}) + \inner{f'(x_{j-1})}{\cdot-x_{j-1}}$ 
		for every $j \ge 1$.
		In view of  this observation and \eqref{eq:sub}, it is now easy to see that RPB$(x_0,\lam,\delta)$ reduces to CS-CS$(x_0,\lam)$.
		% 		with
		% 		% constant prox stepsize 
		% 		$\lam\le\delta/(2M M_f)$.
	\end{proof}

	% 	We end this subsection by making three remarks
	% 	about certain features of the CS-CS and RPB
	% 	methods.
	% 	% 	First, using a similar argument as in the proof of Lemma 9.25 of \cite{beck2017first}, it can be shown
	% 	% 	that the composite subgradient method with prox stepsize $ \lam\le \bar \varepsilon/(4M_f^2) $ converges and the number of iterations to obtain a $ \bar \varepsilon $-solution is $ {\cal O}( d_0^2/(\lam \bar \varepsilon)) $. 
	% 	% 	Second, for the unconstrained case in which $h \equiv 0$,
	% 	% 	the composite subgradient method with
	% 	% 	stepsize $\lam$ satisfying
	% 	% 	$ \lam =\Omega ( \bar \varepsilon/M_f^2) $ is known to be optimal (see for example Chapter 3.2.3 of \cite{nesterov2018lectures}).
	% 	First, CS-CS does not require $h$ to be Lipschitz continuous while RPB does. However,
	% 	RPB does not require any estimate on
	% 	$M_h$ as input.
	% 	Second, CS-CS uses $ M_f $ as input, while RPB has the advantage of not requiring $ M_f $ as input.
	% 	Third, while the complexity bound for
	% 	the CS-CS method in \eqref{eq:bound}
	% 	does not depend on
	% 	$ M_h $, the one for RPB in
	% 	\eqref{cmplx:total-strong} does. % depend on $ M_h $.
	
	% 	Second, if $ d_0\le \bar \varepsilon/(4M_f) $, then it follows from Lemma \ref{lem:sub-new} that CS-CS with $ \lam=\bar \varepsilon/(4M_f^2) $ finds a $ \bar \varepsilon $-solution of \eqref{eq:ProbIntro} in one iteration.
	% 	Third, since the first iteration analysis of RPB is the same as that of CS-CS, the second remark also holds for RPB with $ \lam=\bar \varepsilon/(4M_f^2) $.

	\section{Proof of Theorem \ref{thm:suboptimal}}\label{sec:serious}
	
	This section provides the proof of Theorem~\ref{thm:suboptimal}, which describes a general iteration-complexity for
	RPB to find a $ \bar \varepsilon $-solution of \eqref{eq:ProbIntro}.

	%	\subsection{Proof of Theorem \ref{thm:suboptimal}}\label{subsec:proof}
	
	We start by introducing
	some notation and definitions.
	Consider the sequences
	$\{f_j\}$, $\{x_j\}$ and $\{\tx_j\}$ as in \eqref{def:fj}, \eqref{def:xj} and \eqref{def:txj}, respectively, and let $\{j_k : k \ge 0\}$
	denote the sequence of serious iteration indices generated by RPB (and hence $j_0=0$). Moreover,  define $ z_0:=x_0 $, $ \tz_0:=x_0 $ and, for
	every $k \ge 1$,
	\begin{align}
		z_k &:= x_{j_k}, \quad \tz_k := \tx_{j_k}, \quad \tilde f_k:=f_{j_k}. \label{not}
		%	\delta_k&:=\phi(\tz_k) - (\tilde f_k+h)(z_k)-\frac{1}{2\lam}\|z_k-z_{k-1}\|^2. \label{def:delta-k}
	\end{align}
	Using the definitions of $ \hat z_k $ and $ \tz_k $ in \eqref{def:hat zk} and \eqref{not}, respectively, we have
	\begin{equation}\label{eq:hatz}
		\hat z_k \in \Argmin \left\lbrace \phi(z)  : z \in \{\tz_0,\tz_1,\ldots,\tz_k\} \right\rbrace \quad \forall k \ge 1.
	\end{equation}
	
	The following lemma provides some technical results that will be used in the proof of Theorem \ref{thm:suboptimal}.
	\begin{lemma}\label{lem:iterate}
		%	    Assume that $z_{k-1}$ for some $k \ge 1$ is a
		%	    serious iterate such that
		%		 $ \|z_{k-1}-x_0^*\| \le \sqrt{2} d_0 $ and 
		%		 that $\lam$ satisfies the second inequality in (C1).
		%		 Then, the next serious iterate $ z_k $ and the next auxiliary serious iterate $ \tz_k $ will occur.
		The following statements
		about RPB$(x_0,\lam,\delta)$
		hold for every $ k\ge 1 $:
		\begin{itemize}
			\item[a)] $ z_{k-1} =x^c_j $,  for every $ j = j_{k-1}, \ldots, j_k-1$;
			\item[b)] $ z_k =\argmin \left\lbrace (\tilde f_k+h)(u) +\|u- z_{k-1} \|^2/(2\lam): u\in\R^n \right\rbrace $;
			\item[c)] $ \delta_k + \|\tz_k-z_{k-1}\|^2/(2\lam) \le \delta $ where $ \delta_k:=\phi(\tz_k) - (\tilde f_k+h)(z_k)-\|z_k-z_{k-1}\|^2/(2\lam) $;
			\item[d)] $ \phi(\tz_k) - \phi(z) + (1+\lam \mu)\|z_{k} - z\|^2/(2\lam) \le \delta_k +\|z_{k-1}-z\|^2/(2\lam) $;
			\item[e)] $ \|\tz_k-z_k\|^2 \le 2\lam \delta $.
		\end{itemize}
	\end{lemma}
	
	\begin{proof}
		%		The assumptions of the lemma imply that (C1) holds with $ x_{\ell_0}=z_{k-1} $, and hence that the analysis in Subsection \ref{subsec:pf} follows. In particular, it follows from Proposition \ref{prop:null-strong} that $ x_{\ell_1} $ exists. In view of \eqref{not} with $ j_k=\ell_1 $, we conclude that $ z_k $ and $ \tz_k $ will occur.
		%		Now, we prove statements a)-d).
		a) This statement follows from the definition of $ z_k $ in \eqref{not} and the prox-center update policy in step 2 of RPB.
		
		b) This statement follows from \eqref{def:xj} with $ j=j_k $, \eqref{not} and a).
		
		c) Using the fact that $ m_j=\underline{\phi}_j^\lam(x_j) $ (see step 1 of RPB) with $ j=j_k $ and a), we have
		\[
		m_{j_k}=(\tilde f_k+h)(z_k) + \frac{1}{2\lam}\|z_k-z_{k-1}\|^2.
		\]
		Relation \eqref{def:philam} with $ j=j_k $, \eqref{not} and a) imply that
		\[
		\phi^\lam_{j_k}(\tx_{j_k}) = \phi(\tz_k) + \frac{1}{2\lam}\|\tz_k-z_{k-1}\|^2.
		\]
		Since $ j_k $ is a serious iteration index, \eqref{ineq:hpe1} holds with $ j=j_k $. Using this conclusion, the above two identities and the definition of $ \delta_k $ in the statement, we conclude that
		\[
		\delta_k + \frac{1}{2\lam} \|\tz_k-z_{k-1}\|^2 = \phi^\lam_{j_k}(\tx_{j_k}) - m_{j_k} \le \delta.
		\]
		
		d) Noting that the objective function in b) is $ (\mu + 1/\lam)$-strongly convex,
		and using b) and Theorem 5.25(b) of \cite{beck2017first}
		with $f=\tilde f_k+h+\|\cdot-z_{k-1}\|^2/(2\lam)$, $x^*=z_k$ and $ \sigma=\mu + 1/\lam $, %Lemma \ref{lem:strong}
		we have for every $k \ge 1$ and $ z\in \R^n $,
		\[
		(\tilde f_k+h)(z_k) + \frac1{2\lam}\|z_k-z_{k-1}\|^2 
		\le (\tilde f_k + h)(z) + \frac1{2\lam}\|z-z_{k-1}\|^2
		- \frac1{2}\left( \mu+\frac{1}{\lam}\right) \|z-z_{k}\|^2.
		\]
		The above inequality, the fact that $\phi \ge \tilde f_k+h $ and the definition of $ \delta_k $ in c) imply that
		\begin{align*}
			\phi(\tz_k) &- \phi(z) + \frac1{2}\left( \mu+\frac{1}{\lam}\right)\|z_{k} - z\|^2 
			\le \phi(\tz_k) - (\tilde f_k+h)(z) + \frac1{2}\left( \mu+\frac{1}{\lam}\right)\|z_{k} - z\|^2 \\
			& \le \phi(\tz_k) - (\tilde f_k+h)(z_k) -  \frac1{2\lam}\|z_k-z_{k-1}\|^2 +\frac1{2\lam}\|z_{k-1}-z\|^2  =\delta_k +\frac1{2\lam}\|z_{k-1}-z\|^2. %\label{ineq:main}
		\end{align*}
		
		e) This statement follows from d) with $ z=\tz_k $ and c).
	\end{proof}

	We are now ready to prove Theorem \ref{thm:suboptimal}.
	
	\noindent
	{\bf Proof of Theorem \ref{thm:suboptimal}:}
	Recall that Theorem~\ref{thm:suboptimal} deals with RPB$(x_0,\lam,\delta)$ with $\delta = \bar \varepsilon/2$.
	Using Lemma~\ref{lem:iterate}(d) with $ z=x_0^* $, and the fact that $ \delta_k \le \delta $ (see Lemma \ref{lem:iterate}(c)), we have
	\begin{equation}\label{ineq:observation}
		\phi(\tz_k) - \phi^* \le \frac1{2\lam}\|z_{k-1}-x_0^*\|^2 - \frac{1+\lam \mu}{2\lam}\|z_{k} - x_0^*\|^2 + \delta \quad \forall k \ge 1.
	\end{equation}
	Since \eqref{ineq:observation} satisfies \eqref{eq:easyrecur} with $ \eta_k=\phi(\tz_k) - \phi^*$, $\alpha_k=\|z_{k}-x_0^*\|^2/(2\lam)$, $ \theta=1+\lam\mu $, and $ \delta= \bar \varepsilon/2 $,
	it follows from Lemma \ref{lm:easyrecur},
	the fact that $ \alpha_0=d_0^2/(2\lam) $,
	and relation \eqref{eq:hatz}, that
	%	Using the fact that $ 1+x>e^{x/2} $ 
	%	%	\red{($ \log(1+\alpha)\ge \alpha/(1+\alpha) $ for every $ \alpha \ge 0 $)} 
	%	for every $ 0<x\le 1 $, we have
	%	\[
	%	\sum_{j=1}^{k} (1+\lam \mu)^{j-1} = \max \left\lbrace k, \frac{(1+\lam \mu)^k-1}{\lam \mu} \right\rbrace \ge \max\left\lbrace k, \frac{e^{\lam \mu k /2} - 1}{\lam \mu}\right\rbrace.
	%	\]
	%	This inequality, the fact that $ \delta=\bar \varepsilon/2$ and \eqref{ineq:unify1} imply that for $ k\ge 1 $,
	%	\[
	%	\phi(\hat z_k)-\phi^* \le \frac{d_0^2}{2\lam \max\left\lbrace k, \frac{e^{\lam \mu k /2} - 1}{\lam \mu}\right\rbrace} + \frac{\bar \varepsilon}{2}.
	%	\]
	%	the index $k$ of
	every $k\ge 1$ such that
	$ \phi(\hat z_k)-\phi^*> \bar \varepsilon $
	satisfies
	\begin{equation}\label{ineq:k}
		k< \min \left\lbrace  \frac{d_0^2}{\lam \bar \varepsilon},  \frac{1+\lam \mu}{\lam\mu} \log\left( \frac{\mu d_0^2}{\bar \varepsilon} + 1\right) \right\rbrace
	\end{equation}
	and
	\begin{equation}\label{ineq:zk-d0}
		\|z_k-x_0^*\| \le
		\sqrt{d_0^2 + 2\lam k \delta}
		=\sqrt{d_0^2 + \lam k \bar \varepsilon} \le \sqrt{2} d_0
	\end{equation}
	where the identity is due to the fact that $ \delta= \bar \varepsilon/2 $ and the last inequality is due to \eqref{ineq:k}.
	% 	must satisfy $ \phi(\hat z_k)-\phi^*\le \bar \varepsilon $. 
	%     Hence, statement a) follows.
	% 	Moreover, if $k$ is such that
	% 	$\phi(\hat z_k)-\phi^*>\bar \varepsilon $, and hence
	% 	$k \le d_0^2/(\lam \bar \varepsilon)$, then
	% 	it follows from Lemma \ref{lm:easyrecur}(b) and
	% 	the fact that $ \alpha_0=d_0^2/(2\lam) $ that 
	% 	\begin{equation}\label{ineq:zk-d0}
	% 		\|z_k-x_0^*\| \le \sqrt{d_0^2 + 2\lam k \delta}.
	% 	\end{equation}
	Clearly, the first conclusion above (i.e., \eqref{ineq:k})
	and the definition of $\tilde \lam$ in \eqref{def:tlam}
	imply a). Moreover, 
	the second one (i.e., \eqref{ineq:zk-d0}) together with
	the first identity in \eqref{not}
	and 
	Proposition \ref{prop:null-strong} with $ \delta=\bar \varepsilon/2 $
	imply that b) holds.
	% 	Moreover, it follows from the first identity in \eqref{not} and the definition of $\ell_0$ that  $x_{\ell_0}=z_k$ for some $k$ satisfying \eqref{ineq:k}.
	% 	Hence, this conclusion, \eqref{ineq:zk-d0} and Proposition \ref{prop:null-strong} with $ \delta=\bar \varepsilon/2 $ imply that b) holds.
	% 	We next show b).
	% 	In view of Proposition \ref{prop:null-strong} with $ \delta=\bar \varepsilon/2 $ and
	% 	the fact that
	% 	Since a serious iterate 
	% 	It suffices to show that
	% 	for every serious iterate
	% 	Let $ \ell_0 $ be an arbitrary serious iteration index preceding the last one.
	% 	In view of Proposition \ref{prop:null-strong} with $ \delta=\bar \varepsilon/2 $, it suffices to show that $ d_{\ell_0}\le \sqrt{2}d_0 $ where $ d_{\ell_0}=\|x_{\ell_0}-x_0^*\| $ (see Lemma \ref{lem:ell0}).
	% 	It follows from \eqref{ineq:k} that the number of serious iterates is bounded by $d_0^2/(\lam \bar \varepsilon) + 1$, and hence that $ x_{\ell_0}=z_k $ for some $ k\le d_0^2/(\lam \bar \varepsilon) $ in view of the first identity in \eqref{not}.
	% 	Using this conclusion, \eqref{ineq:zk-d0}, and the facts that $ d_{\ell_0}=\|x_{\ell_0}-x_0^*\| $ and $ \delta=\bar \varepsilon/2 $, we conclude that $ d_{\ell_0}\le \sqrt{2}d_0 $, and hence that b) holds.
	Finally, c) follows immediately from a) and b).
	\QEDA

	\section{Complexity results for other termination criteria}\label{sec:other}
	
	This section contains two subsections.
	The first one describes two
	alternative notions of approximate solutions for problem \eqref{eq:ProbIntro}.
	The second one states iteration-complexity results 
	with respect to these approximate solutions.
	% 	More specifically, it establishes iteration-complexity bounds for RPB to obtain a $ (\hat \rho, \hat \varepsilon) $-solution triple (see Theorem~\ref{thm:approximate}), and a $ \bar \varepsilon $-solution pair when $\dom h$ is bounded (see Corollary~\ref{cor:pair}).
	For simplicity, we assume in this section that $ \mu=0 $ and $ M_h $ is finite.
	
	\subsection{Other termination criteria}\label{subsec:notion}
	
	Usually, algorithms for solving \eqref{eq:ProbIntro} naturally
	generate pairs $(x,\eta)$ satisfying
	the inclusion $0 \in \partial_{\eta} \phi(x)$,
	or equivalently, the inequality $\phi(x) - \phi^* \le \eta$,
	in all of 
	%	(or, an infinite many number of) 
	their iterations
	(see the discussion in the second and third paragraphs following Definition \ref{def:approximate} below).
	For the purpose of our discussion in this section,
	we refer to such a pair $(x,\eta)$ as a
	$\phi$-compatible pair.
	Moreover, a $\phi$-compatible pair $(x,\eta)$
	is called a $\bar \varepsilon$-solution pair of \eqref{eq:ProbIntro} if its residual $\eta$
	satisfies $\eta \le \bar \varepsilon$.
	We now make a few remarks about a given $\phi$-compatible pair $(x,\eta)$.
	First, if 
	% 	$(x,\eta)$ is a
	% 	$\bar \varepsilon$-solution pair, i.e.,
	$\eta \le \bar \varepsilon$, then
	$x$ is a $\bar \varepsilon$-solution.
	Second, checking whether $\eta \le \bar \varepsilon$ is satisfied
	is much
	simpler than checking whether \eqref{def:suboptimal} holds.
	Third,
	it is possible for $(x,\eta)$ to
	satisfy the inequalities
	\eqref{def:suboptimal}
	and $\eta > \bar \varepsilon$,
	which means that
	$x$ is already a desired
	$\bar \varepsilon$-solution but the certificate (or residual)
	$\eta$ is not suitable
	to detect this fact. 
	% Mathematically, this means
	% 	that \eqref{def:suboptimal} holds but $\eta$
	% 	does not lie in the non-empty interval
	% 	$[\phi(x)-\phi^*, \bar \varepsilon]$.
	
	More generally, the following definition of
	an approximate solution triple of \eqref{eq:ProbIntro}
	will be useful.
	
	\begin{definition}\label{def:approximate}
		A triple $(x,v,\eta)$ is called
		$\phi$-compatible if it satisfies the
		inclusion $v \in \partial_\eta \phi(x)$.
		For a given tolerance pair $(\hat \rho, \hat \varepsilon)$,
		a $\phi$-compatible triple $(x,v,\eta)$
		is called a
		$(\hat \rho,\hat\varepsilon)$-solution triple of \eqref{eq:ProbIntro} if it satisfies
		$	\|v\|\le \hat \rho$ and
		$ \eta \le \hat \varepsilon$.
	\end{definition}
	
	At this point, it is interesting to illustrate
	the notion of a $\phi$-compatible triple in the specific setting of \eqref{eq:ProbIntro} where $h(\cdot)=I_K(\cdot)$
	and $K$ is a nonempty closed convex cone. In such setting,
	$(x,v,\eta)$ is $\phi$-compatible if and only if there exists
	$s \in \partial f (x) $
	such that
	$s-v \in K^*$ and $\inner{x}{s-v} \le \eta$ (see Lemma 3.3 in \cite{monteiro2010complexity}).
	Clearly, when $v=0$ and $\eta=0$,
	the latter condition
	implies that $x$ is an
	optimal solution of \eqref{eq:ProbIntro}.
	In general, $v$ is
	a perturbation made on $s$
	to obtain a dual feasible
	point $s-v \in K^*$ and
	$\eta$ is an upper bound on
	the
	complementarity gap
	of the primal-dual feasible
	pair $(x,s-v)$
	(see Proposition 3.4 in \cite{monteiro2010complexity}).
	This specific setting shows
	that the two residuals
	$v$ and $\eta$
	have their own natural meanings. This same phenomenon
	can also be observed in
	the context of other
	constrained convex optimization problems and
	monotone variational inequalities
	(see for example \cite{monteiro2010complexity,monteiro2011complexity}).
	
	We now make some comments about the use of
	the above definition as a natural algorithmic stopping
	criterion.
	Many algorithms,
	including the one considered in this paper,
	are able to naturally generate
	a sequence of $\phi$-compatible triples
	$\{(\hat z_k,\hat v_k,\hat \varepsilon_k)\}$ for which
	the residual pair
	$(\hat v_k,\hat \varepsilon_k)$
	can be made arbitrarily small (see for example Proposition \ref{prop:approximate} below).
	As a consequence,
	some $(\hat z_k,\hat v_k,\hat \varepsilon_k)$ will eventually become a $(\hat \rho,\hat\varepsilon)$-solution triple of \eqref{eq:ProbIntro} and verifying this
	simply amounts to checking whether
	the two inequalities $\|\hat v_k\| \le \hat \rho$
	and $\hat \varepsilon_k \le \hat \varepsilon$
	hold.
	
	It is natural to wonder whether
	these same algorithms can also produce
	a sequence as above but with
	$\hat v_k=0$ for every $k \ge 0$. It turns out that, when
	$\dom h$ is unbounded, such a sequence is generally
	difficult or impossible to obtain.
	However, when $\dom h$ is bounded, we can easily construct such a sequence using the one as in the
	previous paragraph.
	Indeed, let $S$ be a compact convex set containing $\dom h$ and,
	for every $k$, define
	\begin{equation}\label{def:etak}
		\hat \eta_k := \hat \varepsilon_k +
		\sup \{ \inner{\hat v_k}{\hat z_k - x}  : x \in S \}.
	\end{equation}
	Then, using the assumption that
	$(\hat z_k,\hat v_k,\hat \varepsilon_k)$ is
	a $\phi$-compatible triple,
	the definition of $ \varepsilon $-subdifferential in Subsection~\ref{subsec:DefNot},
	and the above definition
	of $\hat \eta_k$, we conclude that
	\[
	\phi(x) \ge \phi(\hat z_k) + \inner{\hat v_k}{x-\hat z_k} - \hat{\varepsilon}_k \ge \phi(\hat z_k) - \hat \eta_k \quad \forall x \in \dom h,
	\]
	or equivalently,
	$0 \in \partial \phi_{\hat \eta_k}(\hat z_k)$.
	Hence,  $\{(\hat z_k,0,\hat \eta_k)\}$ is
	a sequence of $\phi$-compatible triples
	with $\hat v_k=0$ for every $k$,
	or equivalently,
	$\{(\hat z_k,\hat \eta_k)\}$ 
	is a sequence of $\phi$-compatible pairs.
	Moreover, using \eqref{def:etak}, and the assumptions that
	$S$ is bounded and $(\hat v_k,\hat \varepsilon_k)$
	can be made arbitrarily small, we easily see that
	$\hat \eta_k$ can also be made arbitrarily small.
	Observe that this implies that,
	for any given
	tolerance $\bar \varepsilon>0$,
	an index $k$ will eventually be generated such that
	$(\hat z_k,\hat \eta_k)$ is
	a $\bar \varepsilon$-solution pair, and
	detecting the latter property simply amounts
	to checking whether the inequality
	$\hat \eta_k \le \bar \varepsilon$ holds.

	\subsection{Iteration-complexity results} \label{subsec:other}
	
	%	We assume (A1)-(A3) holds and also assume for simplicity that a stronger version of (A4) holds:
	%	\begin{itemize}
	%		\item[(A4')]
	%		a (first-order) oracle function $g:\dom h \to \R^n$ is available such that
	%		\[
	%		g(u) \in \partial f(u), \quad \|g(u)\| \le M_f \quad \forall u\in \dom h
	%		\]
	%	\end{itemize}	
	%	where $ M_f $ is as in (A4).
	%	Since (A4') is a stronger version of (A4), we can prove stronger versions of Lemma \ref{lem:iterate} and Proposition \ref{prop:dist}. We present the results in the following lemma without giving the proof, since it can be easily proved by following similar arguments as in the proofs of Lemma \ref{lem:iterate} and Proposition \ref{prop:dist}.
	%	\begin{lemma}\label{lem:strong}
	%		For every $ k \ge 1 $, we have
	%		\begin{itemize}
	%			\item[a)] $ \|z_k-x_0^*\| \le \sqrt{d_0^2 + 2\lam k \delta} $;
	%			\item[b)] $ \|\tz_k-z_k\|^2 \le 2\lam \delta $;
	%			\item[c)] $ \phi(\tz_k) - \phi(z) \le \delta_k + \left( \|z_{k-1}-z\|^2 - \|z_{k} - z\|^2\right)/(2\lam) $ for every $ z\in \dom h $;
	%			\item[d)] $ \delta_k \le \delta $. 
	%		\end{itemize}
	%	\end{lemma}

	%	The following lemma states some bounds on the magnitude of the sequences
	%	$ \{z_k\} $ and $ \{\hat z_k \} $ which are used in the proof of
	%	Proposition \ref{prop:approximate} to obtain convergence rate
	%	bounds on the sequence of residual pairs $ (\hat v_k, \hat \varepsilon_k) $.
	%	Recall that the latter bounds have been used in the proof  of Theorem \ref{thm:approximate} to establish
	%	the second iteration-complexity for RPB to obtain a $(\hat \rho,\hat \varepsilon)$-solution triple.
	
	The following lemma states some bounds on the magnitude of the sequences
	$ \{z_k\} $ and $ \{\hat z_k \} $ which are used in establishing the iteration-complexity for RPB to obtain a $(\hat \rho,\hat \varepsilon)$-solution triple.
	
	\begin{lemma}
		For every $ k\ge 1 $, we have
		\begin{align}\label{ineq:z}
			\|z_k-z_0\|&\le \sqrt{2k\lam \delta} + 2d_0, \\
			\|\hat z_k-z_0\|^2&\le 2\lam \delta + 5\sqrt{k}\lam \delta + 3k\lam \delta
			%	+\frac{3d_0^2}{\sqrt{k}} 
			+ \frac{15d_0^2}{2} \label{ineq:hatz}
		\end{align}
		where $\hat z_k$, $d_0$ and $z_k$ are as in \eqref{def:hat zk}, \eqref{def:d0} and \eqref{not}, respectively, and $\delta$ is as in step 0 of RPB.
	\end{lemma}
	\begin{proof}
		Using the first inequality in \eqref{ineq:zk-d0}, the triangle inequality, and the facts that $d_0=\|z_0-x_0^*\|$ and $ \sqrt{a+b}\le \sqrt{a} + \sqrt{b} $ for every $ a,b \in \R_+ $, we have 
		\[
		\|z_k-z_0\| \le \|z_k-x_0^*\| + \|z_0-x_0^*\|
		\le \sqrt{2k\lam\delta} + 2 d_0,
		\]
		and hence \eqref{ineq:z} holds.
		Using the fact that
		$(\sum_{i=1}^n a_i)^2 \le (\sum_{i=1}^n s_i)(\sum_{i=1}^n a_i^2/s_i)$
		for every $(a_1,\ldots,a_n)\in \R^n$ and $(s_1,\ldots,s_n)\in \R^n_{++}$, 
		the triangle inequality, the first inequality in \eqref{ineq:zk-d0}, and Lemma \ref{lem:iterate}(e),
		we conclude that for every $ k\ge 1 $,
		\begin{align*}
			\| \tz_k-z_0\|^2&\le 
			\left(\|\tz_k -z_k\| + \|z_k - x_0^*\| + \|x_0^*-z_0\|\right)^2 \\
			&\le \left( \frac{1}{\sqrt{k}}+1+\frac{1}{2}\right) \left( \sqrt{k}\|\tz_k-z_{k}\|^2+\|z_{k}-x_0^*\|^2+2\|x_0^*-z_0\|^2\right)  \\
			&\le \left( \frac{1}{\sqrt{k}}+\frac{3}{2}\right) \left( 2\sqrt{k}\lam \delta + 2k\lam \delta + 3d_0^2\right)  \\
			&=2\lam \delta + 5\sqrt{k}\lam \delta + 3k\lam \delta +\frac{3d_0^2}{\sqrt{k}} + \frac{9d_0^2}{2}.
		\end{align*}
		Since \eqref{eq:hatz}
		implies that
		there exists $i \in \{0,1,\ldots,k\}$ such that $\hat z_k =\tilde z_i$, 
		the above inequality with $ k=i $ then implies that
		\[
		\| \hat z_k-z_0\|^2 = \| \tz_i-z_0\|^2 \le 2\lam \delta + 5\sqrt{i}\lam \delta + 3i\lam \delta + \frac{15d_0^2}{2},
		\]
		from which \eqref{ineq:hatz} immediately follows due to
		the fact that $ i\le k $.
	\end{proof}
	
	We now make a remark about the above result.
	Bound \eqref{ineq:hatz} and its proof can be
	significantly simplified
	at the expense of obtaining a bound whose constant
	multiplying the term $k\lam\delta$ is not as tight as
	its current value, namely 3.
	The current value is the best
	we could obtain and, as we will see from the second inequality for $ \hat \varepsilon_k $ in \eqref{ineq:vk}, %Proposition \ref{prop:approximate},
	the smaller this constant is, the closer $\delta$ can
	be chosen to the tolerance
	$\hat \varepsilon$.
	
	%	The complexity bound \eqref{cmplx:total-strong} is in regard to the termination
	%	criterion \eqref{def:suboptimal}. 
	The following two results
	establish the iteration-complexity for
	RPB to find a $(\hat \rho,\hat\varepsilon)$-solution triple (see Definition~\ref{def:approximate}).
	The first one of these two results describes the convergence rate of
	a certain sequence of triples
	$\{(\hat z_k,\hat v_k,\hat \varepsilon_k)\}$ generated by RPB.
	
	\begin{proposition}\label{prop:approximate}
		Define
		\begin{align}
			\hat v_k &:=\frac{z_0-z_k}{\lam k}, \quad
			\hat{\varepsilon}_k := \frac1k \sum_{i=1}^k \delta_i + \frac{\|\hat z_k - z_0\|^2- \|\hat z_k - z_k\|^2}{2\lam k} \label{def:ek} \quad \forall k\ge 1
		\end{align}
		where $\lam$ is as in step 0 of RPB.
		Then, the following statements hold for every $k \ge 1$:
		\begin{itemize}
			\item[a)] $ \hat v_k\in \partial_{\hat \varepsilon_k} \phi(\hat z_k) $;
			\item[b)] the residual pair $(\hat v_k, \hat \varepsilon_k)$ is bounded by
			\begin{equation}\label{ineq:vk}
				\|\hat v_k\|\le \frac{2d_0}{\lam k} + \frac{\sqrt{2\delta}}{\sqrt{\lam k}}, \qquad 0 \le \hat \varepsilon_k\le \frac{5\delta}{2}\left(1 + \frac{1}{\sqrt{k}} + \frac{2}{5k} \right) + 
				%	\frac{3d_0^2}{2\lam k^{3/2}}\left(1+\frac{3\sqrt{k}}{2}\right)
				\frac{15d_0^2}{4\lam k}
			\end{equation}
			where $ d_0 $ is as in \eqref{def:d0} and $\delta$ is as in step 0 of RPB.
		\end{itemize}
		%		\begin{equation}\label{def:z-hat}
		%		    \hat z_k \in \Argmin \left\lbrace \phi(\tz_0),\phi(\tz_1),\ldots,\phi(\tz_k)\right\rbrace,
		%		\end{equation}
		%		\begin{equation}\label{conclusion}
		%		\hat v_k\in \partial_{\hat \varepsilon_k} \phi(\hat z_k),
		%		\end{equation}
		%		and the residual pair $(\hat v_k, \hat \varepsilon_k)$ is bounded by
		%		\begin{equation}\label{ineq:vk}
		%		\|\hat v_k\|\le \frac{2d_0}{\lam k} + \frac{\sqrt{2\delta}}{\sqrt{\lam k}}, \qquad 0 \le \hat \varepsilon_k\le \frac{5\delta}{2}\left(1 + \frac{1}{\sqrt{k}} + \frac{2}{5k} \right) + 
		%		%	\frac{3d_0^2}{2\lam k^{3/2}}\left(1+\frac{3\sqrt{k}}{2}\right)
		%		\frac{15d_0^2}{4\lam k}
		%		\end{equation}
		%		where $ d_0 $ is as in \eqref{def:d0} and
		%		$\delta$ is as in step 0 of RPB.
	\end{proposition}
	\begin{proof}
		a) 
		%		It follows from Proposition \ref{prop:dist} and Lemma \ref{lem:iterate}(d) that
		%		\[
		%		\phi(\tz_k) - \phi(z) \le \delta_k + \frac{1}{2\lam}\left( \|z_{k-1}-z\|^2 - \|z_{k} - z\|^2\right) \quad \forall k \le \frac{d_0^2}{2\lam \delta}+1.
		%		\]
		It follows from Lemma \ref{lem:iterate}(d) that
		\[
		\phi(\tz_k) - \phi(z) \le \delta_k + \frac{1}{2\lam} \left( \|z_{k-1}-z\|^2 - \|z_{k} - z\|^2\right).
		\]
		Summing the above inequality from $ k=1 $ to $ k=k $ and using \eqref{eq:hatz}, we have
		\[
		\phi(\hat z_k) - \phi(z) \le \frac1k \sum_{i=1}^k \delta_i + \frac{1}{2\lam k}(\|z_0 - z\|^2- \|z_k - z\|^2).
		\]
		This inequality, the obvious identity 
		\[
		\|z-z_0\|^2-\|z-z_k\|^2=\|\hat z_k-z_0\|^2 - \|\hat z_k - z_k\|^2 + 2\inner{z_0-z_k}{\hat z_k-z} \quad \forall z\in \R^n,
		\]
		and the definitions of $ \hat v_k $ and $ \hat \varepsilon_k $ in \eqref{def:ek} imply that
		for every $z\in \dom h $,
		\begin{align}
			\phi(\hat z_k) - \phi(z) &\le \frac{1}{k}\sum_{i=i}^{k}\delta_i + \frac1{2\lam k} \left( \|\hat z_k-z_0\|^2 - \|\hat z_k - z_k\|^2 + 2\inner{z_0-z_k}{\hat z_k-z} \right)\nn \\
			&= \hat \varepsilon_k + \inner{\hat v_k}{\hat z_k -z},\label{ineq:incl}
		\end{align}
		from which we conclude that statement a) holds due to the definition of $ \varepsilon $-subdifferential.
		
		b) The first inequality in \eqref{ineq:vk} follows by plugging \eqref{ineq:z} into the definition of $ \hat v_k $ in \eqref{def:ek}.
		The first inequality for $ \hat \varepsilon_k $, i.e. $ \hat \varepsilon_k\ge 0 $, follows from  \eqref{ineq:incl} with $ z=\hat z_k $. 
		Using the fact that $ \delta_k \le \delta $ (see Lemma \ref{lem:iterate}(c)), the definition of $ \hat \varepsilon_k $ in \eqref{def:ek} and relation \eqref{ineq:hatz}, we have
		\[
		\hat{\varepsilon}_k \le \frac1k \sum_{i=1}^k \delta_i + \frac{\|\hat z_k - z_0\|^2}{2\lam k} \le \delta + \frac{1}{2\lam k}\left( 2\lam \delta + 5\sqrt{k}\lam \delta + 3k\lam \delta + \frac{15d_0^2}{2}\right),
		\]
		from which the second inequality
		for $ \hat \varepsilon_k $ immediately follows.
	\end{proof}

	We now make some remarks about the above result.
	First, Proposition \ref{prop:approximate}(a) shows
	that RPB naturally generates
	a sequence $\{(\hat z_k,\hat v_k,\hat \varepsilon_k)\}$ of $\phi$-compatible triples.
	Second, Proposition \ref{prop:approximate}(b) implies that the sequence $\{\hat \varepsilon_k\}$ can be made arbitrarily small,
	say $\hat \varepsilon_k \le \hat \varepsilon$,
	for sufficiently large $k$,
	as long as $\delta$ is chosen in $(0,2\hat \varepsilon/5)$.
	Third, the two previous remarks
	ensure that RPB is able to generate a $ (\hat \rho, \hat \varepsilon) $-solution triple
	$(\hat z_k,\hat v_k,\hat \varepsilon_k)$.
	Fourth, the three previous remarks in turn show
	that RPB is able to generate
	a sequence $\{(\hat z_k,\hat v_k,\hat \varepsilon_k)\}$ satisfying the properties
	outlined in the second paragraph
	following Definition~\ref{def:approximate}.
	%	satisfying the two properties in the paragraph below Definition \ref{def:approximate}. As a consequence,
	% 	which can be used as a suitable criterion to terminate
	% 	RPB (see the discussion in the last paragraph of Subsection \ref{subsec:assumption}).
	
	We are now ready to describe the iteration-complexity for
	RPB  to find a $(\hat \rho,\hat \varepsilon)$-solution triple
	of \eqref{eq:ProbIntro}.

	\begin{theorem}\label{thm:approximate}
		For a given tolerance pair $(\hat \rho, \hat \varepsilon) \in \R^2_{++}$,
		the following statements about the RPB method hold with $\delta=\hat \varepsilon/3$:
		\begin{itemize}
			\item[a)] the number of serious iterations performed 
			%	generated by the RPB method with $\delta=\hat \varepsilon/3$ until 
			until it obtains
			a $ (\hat \rho,\hat \varepsilon) $-solution triple $ (\hat z_k,\hat v_k,\hat \varepsilon_k) $ is bounded by
			%	\begin{equation}\label{cmplx:serious2}
			\[
			\mathcal{O}_1\left( \max\left\lbrace \frac{\hat \varepsilon}{\lam \hat \rho^2}, \frac{d_0^2}{\lam \hat \varepsilon} \right\rbrace \right);
			\]
			%	\end{equation}
			\item[b)]  the total number of iterations performed
			%	performed by the RPB method with $\delta=\hat \varepsilon/3$ until 
			until it obtains a $ (\hat \rho,\hat \varepsilon) $-solution triple $ (\hat z_k,\hat v_k,\hat \varepsilon_k) $
			is bounded by
			%	\begin{equation}\label{cmplx:total2}
			\begin{equation}\label{cmplx:total-triple}
				\mathcal{O}_1\left( \max\left\lbrace \frac{MM_f}{\hat \rho^2}, \frac{MM_f d_0^2}{\hat \varepsilon^2}\right\rbrace  
				+
				\max\left\lbrace \frac{\hat \varepsilon}{\lam \hat \rho^2}, \frac{d_0^2}{\lam \hat \varepsilon} \right\rbrace 
				+
				\frac{\lam M M_f}{\hat \varepsilon}
				\right),
			\end{equation}
		\end{itemize}
		where $\lam$ and $\delta$ are two of the inputs
		to RPB (see its step 0),
		$d_0$ is as in \eqref{def:d0},
		$ M=M_f+M_h $, and
		$M_f$ and $M_h$ are as
		in (A3) and (A4), respectively.
	\end{theorem}
	\begin{proof}
		a) It follows from Proposition \ref{prop:approximate}(a) and 
		%		the definition of a $ \phi $-compatible triple in 
		Definition \ref{def:approximate} that
		$ (\hat z_k,\hat v_k,\hat \varepsilon_k) $ is a $ \phi $-compatible triple for every  $k \ge 1$.
		%?????? Using the first inequality in \eqref{ineq:vk} with index $ k\ge \max\{4d_0/(\lam \hat \rho), 8\hat \varepsilon/(3\lam \hat \rho^2) \} $ and the fact that $ \delta=\hat \varepsilon/3 $, we have
		%\[
		%\|\hat v_k\|\le \frac{2d_0}{\lam k} + \frac{\sqrt{2\delta}}{\sqrt{\lam k}} \le \frac{\hat \rho}{2} + \frac{\hat \rho}{2} = \hat \rho.
		%\]
		Moreover, the first inequality in \eqref{ineq:vk}
		and the fact that $ \delta=\hat \varepsilon/3 $ imply that for every $ k\ge \max\{4d_0/(\lam \hat \rho), 8\hat \varepsilon/(3\lam \hat \rho^2) \} $,
		\[
		\|\hat v_k\|\le \frac{2d_0}{\lam k} + \frac{\sqrt{2\delta}}{\sqrt{\lam k}} \le \frac{\hat \rho}{2} + \frac{\hat \rho}{2} = \hat \rho
		\]
		%????? Using the second inequality in \eqref{ineq:vk} with index $ k\ge \max\{ 405d_0^2/(2\lam \hat \varepsilon), 36 \} $ and the fact that $ \delta=\hat \varepsilon/3 $, we have
		%\[
		%\hat \varepsilon_k \le \delta + \frac{\|\hat z_k-z_0\|^2}{2k\lam}
		%\le \frac{5\delta}{2}\left( 1+\frac{1}{\sqrt{k}}+\frac{2}{5k}\right) + \frac{15d_0^2}{4\lam k}
		%\le \frac{5\hat \varepsilon}{6}\left( 1+\frac{1}{6}+\frac{1}{90}\right) + \frac{\hat \varepsilon}{54} = \hat \varepsilon.
		%\]
		and the second inequality for $\hat \varepsilon_k$ in \eqref{ineq:vk} and the fact that $ \delta=\hat \varepsilon/3 $ imply that for every  $ k\ge \max\{ 405d_0^2/(2\lam \hat \varepsilon), 36 \} $,
		\[
		\hat \varepsilon_k %\le \delta + \frac{\|\hat z_k-z_0\|^2}{2k\lam} 
		\le \frac{5\delta}{2}\left( 1+\frac{1}{\sqrt{k}}+\frac{2}{5k}\right) + \frac{15d_0^2}{4\lam k}
		\le \frac{5\hat \varepsilon}{6}\left( 1+\frac{1}{6}+\frac{1}{90}\right) + \frac{\hat \varepsilon}{54} = \hat \varepsilon.
		\]
		The above two observations then imply  that
		%Moreover, using the assumption that $ \delta=\hat \varepsilon/3$ and \eqref{ineq:vk}, we easily see that
		$ (\hat z_k,\hat v_k,\hat \varepsilon_k) $ must satisfy the two inequalities in Definition \ref{def:approximate} with $ (v,\eta)=(\hat v_k, \hat \varepsilon_k) $,
		and hence that $ (\hat z_k,\hat v_k,\hat \varepsilon_k) $ is a $ (\hat \rho,\hat \varepsilon) $-solution triple
		(see Definition \ref{def:approximate}),
		%\[
		%k\ge \max\left\lbrace \frac{4d_0}{\lam \hat \rho}, \frac{8\hat \varepsilon}{3\lam \hat \rho^2}, \frac{243d_0^2}{\lam \hat \varepsilon}, \left( \frac{162d_0^2}{\lam \hat \varepsilon}\right) ^{2/3}, 36\right\rbrace
		%\]
		for every index $k$ satisfying
		\[
		k\ge \max\left\lbrace \frac{4d_0}{\lam \hat \rho}, \frac{8\hat \varepsilon}{3\lam \hat \rho^2}, \frac{405d_0^2}{2\lam \hat \varepsilon}, 36\right\rbrace.
		\]
		The complexity bound in a) now follows from the last conclusion
		and the inequality $2\sqrt{ab}\le a+b $ with
		$a=\hat \varepsilon/(\lam \hat \rho^2)$ and $b=d_0^2/(\lam \hat \varepsilon)$.
		% 		for every $ a,b \in \R_+ $, and the fact that $ d_0/(\lam \hat \rho) $ is equal to the square root of the product of $ \hat \varepsilon/(\lam \hat \rho^2) $ and $ d_0^2/(\lam \hat \varepsilon) $.
		
		b) This statement immediately follows from a), Proposition \ref{prop:null-strong} with $\delta=\hat \varepsilon/3$ and the assumption that $M_h$ is finite in the beginning of this section.
		%		and the assumptions on $ \lam $.
	\end{proof}

	The following result describes the iteration-complexity
	for RPB to find 
	a $\bar \varepsilon$-solution pair
	$(x,\eta)=(\hat z_k, \hat \eta_k)$ 
	for the case in which $\dom h$
	is bounded.
	(Recall the definition of a $\bar \varepsilon$-solution pair is given in the first paragraph of Subsection \ref{subsec:notion}).
	Observe that the major difference between the result below and
	Theorem \ref{thm:suboptimal} is that the one below provides
	a certificate $\eta=\hat \eta_k$ of the $\bar  \varepsilon$-optimality of $x=\hat z_k$
	%(see the discussion following \eqref{def:pair})
	while
	Theorem \ref{thm:suboptimal} does not.
	Although it is possible to derive an
	iteration-complexity bound for any
	value of $\lam$ with little extra effort,
	the result below assumes for simplicity that
	$\lam$ lies in a certain range
	and obtains a simpler
	iteration-complexity bound under this assumption.
	%on this range.
	
	\begin{corollary}\label{cor:pair}
		Assume that
		$S \subset \R^n$ is a
		compact convex set containing
		$\dom h$ and let $\bar \varepsilon>0$ be a 
		given tolerance. Consider RPB with
		inner tolerance
		$\delta=\bar \varepsilon/6$ and
		prox stepsize $\lam$ satisfying 
		% 		\begin{equation}\label{cmplx:assumption3}
		% 			\Omega\left( \frac{\bar \varepsilon}{M M_f}\right) =\lam = {\cal O}\left( \frac{D_S^2}{\bar \varepsilon}\right)
		% 		\end{equation}
		$\bar \varepsilon/(C M M_f) \le \lam \le C D_S^2/\bar \varepsilon$
		where 
		$C>0$ is a universal constant,
		$ M=M_f+M_h $,
		$M_f$ and $M_h$ are as
		in (A3) and (A4), respectively,
		and $ D_S:=\sup\{\|u-u'\|:u,u'\in S \}$, and
		let $\{(\hat z_k,\hat v_k,\hat \varepsilon_k)\}$ and
		$\{\hat \eta_k\}$ denote the sequences obtained according
		to \eqref{def:hat zk}, \eqref{def:ek} and \eqref{def:etak}.
		% 		Moreover, consider
		% 		the sequence of triples
		% 		$\{(\hat z_k,\hat v_k,\hat \varepsilon_k)\}$ obtained according
		% 		to \eqref{def:z-hat} and \eqref{def:ek},
		% 		and for every $k \ge 1$, define
		% 		$\hat \eta_k$ as in \eqref{def:etak}.
		Then, the overall iteration-complexity of RPB  until
		it finds a $\bar \varepsilon$-solution pair
		$(\hat z_k,\hat \eta_k)$
		is 
		$\mathcal{O}_1(MM_fD^2_S/\bar \varepsilon^2)$.
		% 		\begin{equation}\label{cmplx:bounded}
		% 			\mathcal{O}_1\left( \frac{MM_fD^2_S}{\bar \varepsilon^2} \right).
		% 		\end{equation}
	\end{corollary}
	\begin{proof}
		The assumption on $S$ and the fact that $x_0\in \dom h$ clearly imply that $ D_S\ge d_0 $.
		Using this fact, the assumption on $\lam$, and Theorem \ref{thm:approximate}(b) with the tolerance pair $ (\hat \rho,\hat \varepsilon)=(\bar \varepsilon/(2D_S),\bar \varepsilon/2) $, we conclude that the overall iteration-complexity for RPB with $ \delta=\bar \varepsilon/6 $ to find a $ (\hat \rho, \hat \varepsilon) $-solution triple $ (\hat z_k, \hat v_k, \hat \varepsilon_k) $ is bounded by $\mathcal{O}_1(MM_fD^2_S/\bar \varepsilon^2)$. %\eqref{cmplx:bounded}. 
		% 		Using this observation together with \eqref{cmplx:assumption3}, we easily see that \eqref{cmplx:assumption2} holds with $ (\hat \rho,\hat \varepsilon)=(\bar \varepsilon/(2D_S),\bar \varepsilon/2) $. Hence, the first remark following Theorem \ref{thm:approximate}, together with the fact that $ D_S\ge d_0 $, implies that the overall iteration-complexity of RPB with $ \delta=\bar \varepsilon/6 $ until it finds a $ (\hat \rho, \hat \varepsilon) $-solution triple $ (\hat z_k, \hat v_k, \hat \varepsilon_k) $ is bounded by \eqref{cmplx:triple} with $ (\hat \rho,\hat \varepsilon)=(\bar \varepsilon/(2D_S),\bar \varepsilon/2) $, namely \eqref{cmplx:bounded}. 
		In view of the definition of a $ (\hat \rho, \hat \varepsilon) $-solution triple in Definition \ref{def:approximate}, we have 
		\begin{equation}\label{triple}
			\hat v_k \in \partial_{\hat \varepsilon_k}\phi(\hat z_k), \quad \|\hat v_k\|\le \hat \rho=\bar \varepsilon/(2D_S), \quad \hat \varepsilon_k\le \hat \varepsilon = \bar \varepsilon/2.
		\end{equation}
		The above inclusion and the remarks following \eqref{def:etak} then imply that the pair
		$ (\hat z_k, \hat \eta_k) $ satisfies the inclusion $ 0 \in \partial_{\hat \eta_k}\phi(\hat z_k) $.
		Moreover, the definition of $ \hat \eta_k $ in \eqref{def:etak}, the Cauchy-Schwarz inequality,
		% 		and the facts that $ \|\hat v_k\|\le \hat \rho=\bar \varepsilon/(2D_S) $ and $ \hat \varepsilon_k\le \hat \varepsilon = \bar \varepsilon/2 $ (see the definition of a $ (\hat \rho, \hat \varepsilon) $-solution triple in Definition \ref{def:approximate}), 
		and the two inequalities in \eqref{triple}, imply that
		\[
		\hat \eta_k \le \hat \varepsilon_k + \|\hat v_k\| D_S \le \frac{\bar \varepsilon}{2} + \frac{\bar \varepsilon}{2} =\bar \varepsilon,
		\]
		and hence that $ (\hat z_k, \hat \eta_k) $ is a $ \bar \varepsilon $-solution pair.
		We have thus shown the corollary.
	\end{proof}
	
	\section{Optimal complexity results for RPB} \label{sec:optimal}
	
	This section contains two subsections. The first one presents a lower complexity result.
% 	reviews basic concepts from complexity theory, introduces some important instance classes for \eqref{eq:ProbIntro}
    The second one shows the optimality of CS-CS and RPB  with respect to some important instance classes introduced in the first subsection.

	\subsection{A lower complexity bound}
	
	Before stating a lower complexity result, we first introduce some
	complexity concepts and define some important instance classes.
	
% 	In this subsection,
% 	we review some basic
% 	complexity theory concepts.
	Given a tolerance $\bar \varepsilon$ and
	an arbitrary class ${\cal I}$ of instances
	$(x_0, (f,f';h))$ such that $x_0 \in \dom h$ and
	$(f,f';h)$ satisfies (A1)-(A2),
	% 	Indeed, we start by describing
	% 	the class of algorithms
	% 	considered in the previous statement.
	% 	For any $\bar \varepsilon >0 $, $x_0 \in \dom h$, and
	% 	class ${\cal I}(x_0)$ of instances $(f,f';h)$
	% 	satisfying (A1)-(A3), 
	let ${\cal A}({\cal I},\bar \varepsilon)$
	denote the class of algorithms ${\cal A}$
	which, for some
	given $(x_0,(f,f';h)) \in {\cal I}$,
	start from $x_0$ and generate a finite
	sequence $\{x_{j-1}\}_{j=1}^J$, $J\ge 1$,
	% 	with $J=J^{\bar \varepsilon}_{x_0}((f,f';h);{\cal A})\ge 1$
	% 	and
	% 	a test point $y_K \in {\rm conv} \{x_0,\ldots,x_K\}$
	satisfying the following two properties:
	a) within
	$\{x_0,\ldots, x_{J-1}\}$, the iterate $x_{J-1}$ is the only one which is
	a $\bar \varepsilon$-solution of \eqref{eq:ProbIntro};
	and
	b) if $h$ is a quadratic function and $\nabla^2 h$ is a multiple of the identity matrix $I$, then for every $j \in \{1,\ldots,J-1\}$, there holds 
	%	\[
	%	x_k-x_0 \in \bigcup \left\{ \sum_{i=0}^{k-1} \alpha_i \partial f(x_i)
	%	: \alpha_i \in \R, \ \forall i=0,\ldots,k-1 \right\}.
	%	\]
	\begin{equation}\label{incl:xk}
		x_j \in x_0 + \text{Lin} \left\lbrace f'(x_0), \ldots, f'(x_{j-1}), \nabla h(x_0), \ldots, \nabla h(x_{j-1}) \right\rbrace
	\end{equation}
	where \text{Lin}$ \{\cdot \} $ is defined in Subsection~\ref{subsec:DefNot}.
	Clearly, the index $J=J^{\bar \varepsilon}_{x_0}((f,f';h);{\cal A})$ above is uniquely determined by the tolerance $\bar \varepsilon$,
	instance $(x_0,(f,f';h))$ and algorithm ${\cal A}$.
	The function
	$J^{\bar \varepsilon}_{x_0}(\cdot;{\cal A})$,
	defined on ${\cal I}$,
	is referred to as
	the  $\bar \varepsilon$-iteration complexity bound of ${\cal A}$ (with respect to ${\cal I}$). 
	% 	defined
	% 	on ${\cal I}$ can thought as the complexity
	% 	of ${\cal A}$  for instance 
	% 	$(x_0,(f,f';h))$, or simply, the instance complexity
	
	For any given
	$\bar \varepsilon >0$ and
	${\cal A} \in {\cal A}({\cal I},\bar \varepsilon)$,
	a
	$\bar \varepsilon$-upper complexity bound for ${\cal A}$ with respect to ${\cal I}$
	is defined to be
	an upper bound on the supremum of
	$J^{\bar \varepsilon}_{x_0}((f,f';h);{\cal A})$ as
	$(x_0,(f,f';h))$ varies in ${\cal I}$.
	Moreover, a
	$\bar \varepsilon$-upper complexity bound for some algorithm
	${\cal A} \in {\cal A}({\cal I},\bar \varepsilon)$
	with respect to ${\cal I}$
	is said to be a $\bar \varepsilon$-upper complexity bound
	for the class ${\cal I}$.
	% 	or simply an upper iteration-complexity
	% 	for ${\cal I}$ if $\bar \varepsilon$ is clear from the
	% 	context.
	For a given instance $(x_0,(f,f';h)) \in {\cal I}$,
	a lower bound on the infimum of $J^{\bar \varepsilon}_{x_0}((f,f';h),{\cal A})$ as ${\cal A}$
	varies in ${\cal A}({\cal I},\bar \varepsilon)$ is called a lower complexity bound of $(x_0,(f,f';h))$ with respect to ${\cal A}({\cal I},\bar \varepsilon)$.
	Moreover, a lower complexity bound
	for some instance in ${\cal I}$
	with respect to ${\cal A}({\cal I},\bar \varepsilon)$
	is called a $\bar \varepsilon$-lower complexity bound for
	the class ${\cal I}$.
	Clearly, if $M_1$ and $M_2$ are $\bar \varepsilon$-lower and $\bar \varepsilon$-upper
	complexity bounds for the class ${\cal I}$, respectively,
	then $M_1 \le M_2$. Moreover, if
	$M_2 = {\cal O}(M_1)$, then either $M_1$ or $M_2$ is said to
	be a $\bar \varepsilon$-optimal complexity bound for
	${\cal I}$ and any algorithm ${\cal A} \in {\cal A}({\cal I}, \bar \varepsilon)$
	which has a $\bar \varepsilon$-upper complexity bound equal to ${\cal O}(M_2)$ is said to be $\bar \varepsilon$-optimal for
	${\cal I}$.

% 	In order to describe the optimal complexity results for the CS-CS method,
% 	as well as RPB in the next section,
	We now define some important instance classes
	for \eqref{eq:ProbIntro}.
	% 	Let $ x_0^* $ denote the closest point  in $X^*$ to $x_0$ and define $ d_0:=\inf \{\|x_0-x^*\|: x^*\in X^*\}$.
	% %	In view of the definition of $ d(\cdot,\cdot) $ in Subsection~\ref{subsec:DefNot}, we have 
	% 	Clearly, we have
	% 	\begin{equation}\label{def:d0}
	% 		d_0=\|x_0-x_0^*\|. %= d(x_0,X^*).
	% 	\end{equation}
	
	\begin{definition}\label{def:ins-class}
		Given
		$(M_f,\mu,R_0) \in \R_+ \times \R_+ \times \R_{++}$,
		let ${\cal I}_\mu(M_f,R_0)$ denote the
		class
		consisting of all instances $(x_0,(f,f';h))$
		satisfying conditions (A1)-(A3) and
		the condition that
		$ d_0\le R_0 $ where $ d_0 $ is as in \eqref{def:d0}.
		Moreover, let $ {\cal I}_\mu^u(M_f,R_0) $ denote the unconstrained class consisting of all instances $(x_0,(f,f';h))\in {\cal I}_\mu(M_f,R_0)$ such that $ h\equiv \mu\|\cdot\|^2/2 $.
	\end{definition}
	
	The following result describes a
		$\bar \varepsilon $-lower complexity bound for any instance class $ {\cal I}  \supset {\cal I}_\mu^u(M_f,R_0) $. 
% 		introduced in Definition \ref{def:ins-class}.
		Its proof is postponed to Appendix \ref{sec:lb}.
	
	\begin{theorem}\label{thm:lb-cvx}
			For any given quadruple $ (M_f,\mu,R_0, \bar \varepsilon) \in \R_{+} \times \R_{+} \times \R_{++} \times  \R_{++}$,
	% satisfying $ \mu R_0^2 \le 8 \bar \varepsilon $,
	there exists an instance $ (x_0,(f,f';h)) $ such that:
			\begin{itemize}
				\item[a)] %$h(\cdot)=\mu\|\cdot\|^2/2$ and
			$ (x_0,(f,f';h)) \in {\cal I}_\mu^u(M_f,R_0) $;
				\item[b)] it has lower complexity bound with respect to $ {\cal A}({\cal I}_\mu^u(M_f,R_0), \bar \varepsilon) $
				given by
	%			\begin{equation}\label{cmplx:min}
	%				\Omega\left( \min\left\lbrace \frac{M_f^2 R_0^2}{\bar \varepsilon^2} , \frac{M_f^2}{\mu \bar \varepsilon}\right\rbrace +1 \right).
	%			\end{equation}
			\begin{equation}\label{cmplx:min}
			\left \lfloor \min\left\lbrace \frac{M_f^2 R_0^2}{128 \bar \varepsilon^2} , \frac{M_f^2}{8\mu \bar \varepsilon}\right\rbrace \right \rfloor +1.
			\end{equation}
	%			\item[c)] if $ M_f R_0/\bar \varepsilon\ge 8 $, then $ (f+h)(x_0)-(f+h)(x^*)>\bar \varepsilon $ where $x^*=\argmin\{(f+h)(x):x\in \R^n\}$.
			\end{itemize}
			As a consequence, \eqref{cmplx:min} is also a $ \bar \varepsilon $-lower complexity bound for
			any instance class
			${\cal I} \supset  {\cal I}_\mu^u(M_f,R_0)$.
% 			is nonempty
% 			and has $ \bar \varepsilon $-lower complexity bound given by .
		\end{theorem}

	It is worth mentioning that
	the second minimand in \eqref{cmplx:min}
	is smaller than the first
	one if and only if $\mu \ge 16\bar \varepsilon/R_0^2$, and
	converges to $\infty$ as $\mu$ approaches zero.

	We now make a few remarks regarding the relationship of
	Theorem \ref{thm:lb-cvx} with
	% traditional ones developed in the literature.
	% We  first compare Theorem \ref{thm:lb-cvx}
	the ones derived in Theorems 3.2.1 and 3.2.5 of
	\cite{nesterov2018lectures}. 
	First, the class of algorithms
	considered in these three results are the same and
	hence are based on the linear hull
	condition \eqref{incl:xk}. 
	Second, the above three results show the existence of
	bad instances $(x_0,(f',f;h))$ such that
	$h=\mu \|\cdot\|^2/2$ (and hence $h \equiv 0$ when $\mu=0$)
	but the functions $f$ of the ones
	of Theorems 3.2.1 and 3.2.5 of
	\cite{nesterov2018lectures} are
	$M_f$-Lipschitz on the ball ${\bar B}(x^*;R_0)$
	for some
	$x^* \in X^*$ while the $f$ for the one of
	Theorem \ref{thm:lb-cvx} 
	is $M_f$-Lipschitz on
	the whole $\R^n$. In contrast to the bad instances of \cite{nesterov2018lectures}, this additional property of the
	bad instance of Theorem \ref{thm:lb-cvx} allows us to show that \eqref{cmplx:min}
	is a $\bar \varepsilon$-lower complexity bound for a smaller instance class, namely ${\cal I}_\mu^u(M_f,R_0)$,
	than the one considered in \cite{nesterov2018lectures}.
	Third, 
	Theorem 3.2.5 (resp., Theorem 3.2.1) in
	\cite{nesterov2018lectures} obtains
	the $\bar \varepsilon$-lower complexity bound $M_f^2/(2\mu \bar \varepsilon)$ (resp., $M_f^2 R_0^2/(4\bar \varepsilon^2)$ ) only for $\mu  \ge 2\bar \varepsilon/R_0^2$ (resp, $\mu=0$),
	and hence \eqref{cmplx:min} is a valid $\bar \varepsilon$-lower complexity bound for any
	$\mu \in \{0\} \cup [2\bar \varepsilon/R_0^2, \infty) $.
	This contrasts with
	Theorem \ref{thm:lb-cvx} which establishes
	the $\bar \varepsilon$-lower complexity bound \eqref{cmplx:min} for any $\mu \ge 0$.

	\subsection{Optimal complexity results for the CS-CS and RPB methods} \label{subsec:optimal}
	
	This subsection establishes the $\bar \varepsilon$-optimality of the
	CS-CS and RPB with
	respect to
	some of the instance classes introduced in Definition \ref{def:ins-class} 
	as well as in this section.
	
	We first tackle the $\bar \varepsilon$-optimalilty of the CS-CS method.
	Let $(x_0,\lam) \in \dom h \times \R_{++}$ and $(M_f,\mu,R_0) \in \R_+ \times \R_+ \times \R_{++} $ be given.
	It is easy to see that
	CS-CS$(x_0,\lam)$
	satisfies property b) in the paragraph containing \eqref{incl:xk}.
	Hence, for any given universal constant $C >1$, it follows from
	Proposition \ref{prop:sub-new} and the definition of ${\cal I}_\mu(M_f,R_0)$
	in Definition \ref{def:ins-class} that
	CS-CS$(x_0,\lam)$ with
	$\bar \varepsilon/(C M_f^2) \le 4 \lam \le  \bar \varepsilon/M_f^2$
	is in ${\cal A}({\cal I}_\mu(M_f,R_0),\bar \varepsilon)$
	and has $\bar \varepsilon$-upper complexity bound for ${\cal I}_\mu(M_f,R_0)$ given by
	\begin{equation}\label{cmplx:bound}
		{\cal O}_1\left(\min \left\lbrace \frac{M_f^2 R_0^2}{\bar \varepsilon^2}, \left(\frac{M_f^2}{\mu \bar \varepsilon} + 1\right) \log\left( \frac{\mu R_0^2}{ \bar \varepsilon} + 1 \right) \right\rbrace  \right).
	\end{equation}
	%	\begin{equation}\label{eq:bound}
	%		{\cal O}_1\left(\min \left\lbrace \frac{M_f^2 R_0^2}{\bar \varepsilon^2}, \red{\left( 1 + \frac{M_f^2}{\mu \bar \varepsilon}\right)} \log_1\left( \frac{\mu R_0^2}{ \bar \varepsilon} \right) \right\rbrace  \right)
	%	\end{equation}
	This observation together with the $\bar \varepsilon$-lower complexity bound in Theorem \ref{thm:lb-cvx}  implies that 
	\eqref{cmplx:bound} is a $\bar \varepsilon $-optimal complexity bound (up to a logarithmic term) for
	any instance class ${\cal I}$ satisfying $ {\cal I}_\mu^u(M_f,R_0) \subseteq {\cal I} \subseteq {\cal I}_\mu(M_f,R_0) $ and
	that CS-CS$(x_0,\lam)$ with
	$\bar \varepsilon/(C M_f^2) \le 4\lam \le  \bar \varepsilon/M_f^2$ for a universal constant $C >1$
	is $ \bar \varepsilon $-optimal (up to a logarithmic term) for ${\cal I}$.

	We next tackle the $\bar \varepsilon$-optimalilty of the RPB method.
	The following result
	describes %ranges on $\lam>0$ and/or 
	conditions on $\bar \varepsilon$ and $(M_f,\mu,R_0)$ that guarantee the $\bar \varepsilon$-optimality of RPB with respect to
	some suitable instance
	classes.
	Its statement
	makes use of the two instance classes ${\cal I}_\mu^u(M_f,R_0)$ and ${\cal I}_\mu(M_f,R_0)$ introduced in Definition \ref{def:ins-class},
	as well as the
	instance class
	$ {\cal I}_0(M_f,R_0;C) $  defined as
	\begin{equation}\label{def:tI}
		{\cal I}_0(M_f,R_0;C) := \{ (x_0,(f,f';h)) \in {\cal I}_0(M_f,R_0):  \mbox{$\exists$  $M_h \le CM_f$ such that $h$ satisfies  (A4)} \}
	\end{equation}
	where $C$ is a  universal constant.

	\begin{theorem}\label{thm:opt2}
	    Let a universal constant $C >0$, tolerance $\bar \varepsilon >0$ and pair $(M_f,R_0) \in \R_+   \times \R_{++}$ be given such that 
	    $C M_f R_0/\bar \varepsilon\ge 1$.
	    Then, the following statements hold:
		\begin{itemize}
		    \item [a)] 
% 			under the extra assumption that $ M_f \ge \mu R_0 $,
% 			up to a logarithmic term,
			for any universal constant $C' >0$, RPB$(x_0, \lam, \bar \varepsilon/2)$ with
			$ \lam $ satisfying \eqref{ineq:lam2} with $d_0$ replaced by $R_0$
% 			\begin{equation}\label{eq:lam2}
% 				\frac{R_0}{M_f} \le \lam \le \min \left\lbrace \frac{R_0^2}{\bar \varepsilon}, \frac{1}{\mu}\right\rbrace 
% 			\end{equation}
			is (up to a logarithmic term) $\bar \varepsilon$-optimal for
			any instance class ${\cal I}$ and scalar $\mu \in [0,C' M_f/R_0]$ such that
			\begin{equation}\label{incl:I2}
				{\cal I}_\mu^u(M_f,R_0) \subseteq {\cal I}
				\subseteq {\cal I}_\mu(M_f,R_0);
			\end{equation}
			\item[b)]
			RPB$(x_0, \lam, \bar \varepsilon/2)$ with $ \lam $ satisfying \eqref{ineq:lam1} with $d_0$ replaced by $R_0$
% 			\begin{equation}\label{eq:lam1}
% 				\frac{\bar \varepsilon}{M_f^2} \le \lam \le \frac{R_0^2}{\bar \varepsilon}
% 			\end{equation}
			is $\bar \varepsilon$-optimal for
			any instance class ${\cal I}$ such that
			\begin{equation}\label{incl:I1}
				{\cal I}_0^u(M_f,R_0) \subseteq {\cal I}
				\subseteq {\cal I}_0(M_f,R_0;C).
			\end{equation}
		\end{itemize}
	\end{theorem}

	We now make some remarks about
	Theorem \ref{thm:opt2}.
	
	The inclusion ${\cal I}_0^u(M_f,R_0) \subseteq  {\cal I}_0(M_f,R_0;C)$
	always holds in view of \eqref{def:tI} and the fact that the composite component
	$h$ of any instance in ${\cal I}_0^u(M_f,R_0)$ is identically zero.
	Hence, in view of the last conclusion of Theorem \ref{thm:lb-cvx},
	\eqref{cmplx:min}  is also a $\bar \varepsilon$-lower complexity bound
	for	$ {\cal I}_0(M_f,R_0;C)$.

	Theorem \ref{thm:opt2}(a) shows that
	RPB with
	$R_0/M_f \le \lam \le C R_0^2/\bar \varepsilon$,
	similar to
	the CS-CS method 
	with
	$\bar \varepsilon/(C M_f^2) \le 4\lam \le  \bar \varepsilon/M_f^2$
	(see Subsection \ref{subsec:cs}),
	is $ \bar \varepsilon $-optimal (up to a logarithmic term) for
	the instance class ${\cal I}_\mu(M_f,R_0)$
	for any $\mu\ge 0$.
	Note that the two ranges of  $\lam$ above do not overlap when $C\le 1$ due to the assumption that $C M_f R_0/\bar \varepsilon\ge 1$ in Theorem \ref{thm:opt2}.
	
	On the other hand,
	Theorem \ref{thm:opt2}(b) asserts that RPB with
	$\lam$ within the much wider
	range $\bar \varepsilon/(CM_f^2) \le \lam \le C R_0^2/\bar \varepsilon$
	is $ \bar \varepsilon $-optimal for
	the smaller instance class $ {\cal I}_0(M_f,R_0;C)$,
	which includes the instance subclass where $h$ is the indicator function
	of a closed convex set.

	\section{Concluding remarks}\label{sec:conclusion}
	
	This paper presents a proximal bundle variant, namely,
	the RPB method, for solving CNCO problems.
	Like many other proximal bundle variants,
	i)
	RPB solves a sequence of prox bundle subproblems
	whose objective functions are obtained by
	a usual regularized composite cutting-plane strategy;
	and ii) 
	RPB performs either serious iterations
	during which the prox-centers are changed or
	null iterations where the prox-centers are left unchanged.
	%	As opposed to other proximal bundle variants:
	However, RPB uses the novel condition \eqref{ineq:hpe1}
	% 	instead of the descent
	% 	condition \eqref{ineq:descent}
	involving $\tx_j$
	to decide whether to perform a serious or null
	iteration. 
	Our analysis shows that the consideration of the sequence $\{\tx_j\}$
	plays an important role in the derivation of
	optimal complexity bounds for RPB over a large
	range of prox stepsizes $\lam$
	in the context of
	CNCO problems.
	
	As far as the authors are aware of,
	this is the first time that such results
	are obtained in the context of
	a proximal bundle variant.
	A nice feature of our analysis is that it is
	carried out in the context of
	CNCO problems and takes into account a flexible
	bundle management policy which allows cut removal but no cut aggregation.
	% 	(which does not allow
	% 	cut aggregation).
	Moreover, it places the CS-CS
	method under the umbrella of RPB
	in that the former can be viewed
	as an instance of the latter with
	a relatively small prox stepsize.
	% 	namely, 
	% 	$\lam=\Theta(\bar \varepsilon/MM_f)$.
	%	(which lies in the aforementioned range where
	%	RPB has optimal complexity).
	This paper also establishes iteration-complexity
	results for RPB to obtain
	iterates satisfying practical termination
	criteria.
	% 	such as a $ (\hat \rho, \hat \varepsilon) $-solution triple (see Theorem~\ref{thm:approximate})
	% 	or $ \bar \varepsilon $-solution pair  (see Corollary~\ref{cor:pair})
	% 	rather than the theoretical termination criterion \eqref{def:suboptimal}.
	
	%	An alternative and more general condition which can be used
	%	in place of (A2) is:  $f \in \bConv{n}$ and 
	%	there exists a scalar $M_f \ge 0$ such that
	%	%	for every $u \in \dom h$, we have 
	%	\begin{equation}\label{cite}
	%	\partial f(u) \cap \{ g\in \R^n : \|g \| \le M_f \} \ne \emptyset \qquad \forall u \in \dom h.
	%	\end{equation}
	%	If the subgradient of $f$ at $x_j$ in step 3 of the RPB
	%	method is chosen from
	%	the above set with $u=x_j$, then it can be easily
	%	verified that the whole analysis of this paper still holds.

	We now discuss some possible extensions of our analysis in this paper.
	First, recall that we have assumed throughout this paper that the prox stepsize $\lam$
	is constant. We believe that a slightly modified version of
	our analysis can be used to study the case in which
	$\lam$ is allowed to change (possibly within a positive closed bounded interval)
	at every iteration $j$ for which $j$ is a serious iteration index.
	% 	{\color{red}???}Second, we have assumed in Section \ref{sec:strong} that $\dom h$ is bounded
	% 	but we also believe that a slight, although more complicated, modification of
	% 	our analysis can handle the case in which the latter assumption is removed.
	Second, 
	% the analysis of Section  \ref{sec:strong} assumes
	% that $h$ is strongly convex. A natural question is whether the results
	% there are also valid under the more general assumption that
	% $\phi$ is strongly convex.
	% Fourth, 
	if $f$ is $\mu_f$-strongly convex and $h$ is $\mu_h$-strongly convex,
	then the CNCO problem \eqref{eq:ProbIntro}
	is clearly equivalent to another CNCO problem \eqref{eq:ProbIntro}
	in which
	$f$ is convex, $h$ is $\mu$-strongly convex, and $\mu=\mu_f+\mu_h$.
	Hence, if $\mu_f$ is known, then there is no loss of generality in assuming that only $h$
	is strongly convex.
	% 	although the aforementioned
	% 	transformation requires knowledge of $\mu_f$.
	Third, a natural question is whether,
	under the
	weaker assumption that $\phi$ is $\mu$-strongly convex,
	the results are still valid for
	RPB  directly applied to the CNCO problem \eqref{eq:ProbIntro}
	without using the above transformation. The advantage of the latter
	approach, if doable, is that it does not
	require the knowledge of $\mu_f$ (nor $\mu_h$).
	Fourth, it would be
	interesting to investigate a variant of RPB under the assumption that $f$ shares properties
	of both a smooth and a nonsmooth
	function, i.e., for some
	nonnegative scalars $M_f$ and $L_f$,
	there holds 
	$\|f'(x)-f'(x')\| \le 2 M_f + L_f\|x-x'\|$ for every $x, x' \in \dom h$.
	Fifth, it would be interesting
	to consider an RPB variant 
	which, instead of using the cutting-plane model $ f_j $ in \eqref{def:fj},
	uses the cut aggregation model considered for example in Chapter 7.4.4 of \cite{ruszczynski2011nonlinear} (see also \cite{du2017rate,de2014convex}). A clear advantage of
	the latter model is that
	the cardinality of the bundle
	is no more than two and, as a
	consequence, subproblem
	\eqref{def:xj} becomes easier to solve.
	Sixth, it would be interesting to extend the conclusion of Corollary \ref{thm:bound-strg} to the one where \eqref{ineq:lam2} is
	replaced by the wider range \eqref{ineq:lam1}.
	Note that such
	version of Corollary \ref{thm:bound-strg},
	if correct,
	would imply Corollary \ref{thm:bound-cvx} as a special case.

	\scriptsize{
		\bibliographystyle{plain}
		\bibliography{Proxacc_ref}
	}
	
%	\small
	\appendix

	\section{Proof of the iteration-complexity of the CS-CS method}
	\label{sec:CS-CS pf}
	
	The goal of this section is to establish
	a complexity bound for CS-CS$(x_0,\lam)$ with $\lam$
	satisfying
	$\bar \varepsilon/(CM_f^2)=4\lam \le \bar \varepsilon/M_f^2$ for a universal constant $C>1$
	without assuming any condition on the
	initial point $x_0$ other than just being in $\dom h$.
	Before presenting the complexity bound result, we first state a useful technical lemma.
	
	% 	Towards this end we will actually prove a more general result which implies the above remark.
	%	This section analyzes the iteration-complexity of
	%	the CS method with constant stepsize
	%	$\lam$ for any $\lam \le \bar \varepsilon/M_f^2$
	%	and under slightly more general assumptions
	%	than (C1) and (C2).
	%	The analysis of the CS method with constant stepsize
	%	is well-known under stronger assumptions than the
	%	ones considered here (see for example Lemma 9.25 of \cite{beck2017first})
	%	but such result in itself is not enough to imply
	%	the optimality of CS within the class ${\cal A}(M_f,x_0,\bar \varepsilon)$
	%	discussed in ????
	%	
	%	
	%	This section presents a technical lemma about the CS method establishing the iteration-complexity of CS with sufficiently small constant prox stepsize to find a $ \bar \varepsilon $-solution for any given triple $(M_f,x_0,\bar \varepsilon) \in \R_+ \times \R^n \times \R_{++} $. It also specializes to the case in which $x_0 \in X^*_h$, as a consequence, the conditions assumed in the lemma become (C1) and (C2) in Subsection~\ref{subsec:CS}.

	\begin{lemma} \label{lm:easyrecur}
		Assume that scalars $\theta \ge 1$ and $\delta>0$, and
		sequences of nonnegative scalars
		$\{\eta_j\}$ and  $\{\alpha_j\}$ satisfy
		\beq \label{eq:easyrecur}
		\eta_j \le \alpha_{j-1} - \theta \alpha_j + \delta \quad \forall j \ge 1.
		\eeq
		% 		for some positive scalars $\theta$ and $\delta$.
		Then, the following statements hold:
		\begin{itemize}
			% 			\item[a)] for every $k \ge 1$, 
			% 			\begin{equation}\label{ineq:conse}
			% 				%t_k \le \frac{\alpha_0}{\theta^{k-1}} + \frac{\Theta_{k}}{\theta^{k-1}} \delta, \quad 
			% 				\left[ \min_{1\le j \le k}  \eta_j \right] + \frac{\theta^k}{\sum_{j=1}^{k} \theta^{j-1}} \alpha_k \le  \frac1{\sum_{j=1}^{k} \theta^{j-1}}  \alpha_0 + \delta;
			% 			\end{equation}
			\item[a)]  $ \min_{1\le j \le k} \eta_j \le 2\delta $ for every $ k\ge 1$ such that
			% 			\[
			% 			k \ge \min \left\lbrace \frac{\alpha_0}{\delta}, \frac{2}{\theta-1}\ln\left( \frac{\alpha_0(\theta-1)}{\delta} + 1 \right)  \right\rbrace
			% 			\]
			\[
			k \ge \min \left\lbrace \frac{\alpha_0}{\delta},  \frac{\theta}{\theta-1} \log\left( \frac{\alpha_0(\theta-1)}{\delta} + 1 \right)  \right\rbrace
			\]
			% 			with the convention that
			% 			$0 \cdot \infty=\infty$;
			with the convention that the second term is equal to the first term
			% 			$2\alpha_0/\delta$ 
			when $\theta=1$ (Note  that the second term converges to the first term
			% 			$2\alpha_0/\delta$ 
			as $\theta \downarrow 1$.);
			
			\item[b)] $ \alpha_k \le \alpha_0 + k \delta$ for every $ k\ge 1 $.
		\end{itemize}
	\end{lemma}
	
	\begin{proof}
		a) Multiplying \eqref{eq:easyrecur} by $ \theta^{j-1} $ and summing the resulting inequality from $ j=1 $ to $ k $, we have
		\begin{equation}\label{ineq:conse}
			\sum_{j=1}^k \theta^{j-1} \left[ \min_{1\le j \le k}  \eta_j \right] \le \sum_{j=1}^k \theta^{j-1} \eta_j \le  \sum_{j=1}^k \theta^{j-1} \left(  \alpha_{j-1} - \theta \alpha_j + \delta \right) =  \alpha_0 - \theta^k \alpha_k +
			\sum_{j=1}^k \theta^{j-1} \delta.
		\end{equation}
		% 		and hence
		% 		\begin{equation}\label{ineq:conse}
		% 				%t_k \le \frac{\alpha_0}{\theta^{k-1}} + \frac{\Theta_{k}}{\theta^{k-1}} \delta, \quad
		% 				\left[ \min_{1\le j \le k}  \eta_j \right] + \frac{\theta^k}{\sum_{j=1}^{k} \theta^{j-1}} \alpha_k \le  \frac1{\sum_{j=1}^{k} \theta^{j-1}}  \alpha_0 + \delta.
		% 			\end{equation}
		% It follows from the fact that $\alpha_k\ge 0$ that for every $ k\ge 1 $,
		% \begin{equation}\label{ineq:eta-j}
		%     \min_{1\le j \le k}  \eta_j \le \frac{\alpha_0}{\sum_{j=1}^{k} \theta^{j-1}} + \delta,
		% \end{equation}
		% and hence that if $\theta \ge 2$, then $\min_{1\le j \le k}  \eta_j \le 2\delta$ for every $k \ge \log_2(\alpha_0/\delta+1)$.
		Using the fact that $ \theta \ge e^{(\theta-1)/\theta} $
		%	\red{($ \log(1+\alpha)\ge \alpha/(1+\alpha) $ for every $ \alpha \ge 0 $)} 
		for every $\theta \ge 1$, 
		we have
		\[
		\sum_{j=1}^{k} \theta^{j-1} = \max \left\lbrace k, \frac{\theta^k-1}{\theta-1} \right\rbrace \ge \max\left\lbrace k, \frac{e^{(\theta-1) k /\theta} - 1}{\theta-1}\right\rbrace.
		\]
		This inequality, \eqref{ineq:conse} and the fact that $\alpha_k\ge 0$ imply that for every $ k\ge 1 $,
		\[
		\min_{1\le j \le k}  \eta_j \le \alpha_0\min \left\lbrace \frac 1k, \frac{\theta-1}{e^{(\theta-1) k /\theta} - 1} \right\rbrace + \delta,
		\]
		which can be easily seen to imply a).
		% 		that $ \min_{1\le j \le k} \eta_j \le 2\delta $ for every $ k\ge 1 $ such that
		% 		\[
		% 		k \ge \min \left\lbrace \frac{\alpha_0}{\delta}, \frac{\theta}{\theta-1}\ln\left( \frac{\alpha_0(\theta-1)}{\delta} + 1 \right) \right\rbrace.
		% 		\]
		% 		Statement a) now follows from this conclusion and the fact that $\theta/(\theta-1) \le 2\max\{1/(\theta-1),1\}$.

		b) This statement follows from \eqref{ineq:conse}, the fact that $\eta_j\ge 0$, and the assumption that $ \theta\ge 1$.
	\end{proof}
	
% 	---------
	
% 	\[
% 	T(\theta;\beta)
% 	:= \frac{\log [ \beta (\theta-1) +1 ]}{\log \theta} \le 2
% 	\]
	
% 	\[ 
% 	\beta (\theta-1) +1 \le \beta L + 1 
% 	\]
% 	which holds if $\theta \ge \beta$.
	
% 	if $\beta>1$, then $T(\theta;\beta)$
% 	is strictly decreasing
% 	and converges to $1$ as
% 	$\theta \to \infty$.
% 	Assume that
% 	$\beta=\alpha_0/\delta>1$. We have two cases:
	
% 	1) $\theta \le 2$. Then,
% 	\[
% 			k \ge \min \left\lbrace \frac{\alpha_0}{\delta},  \frac{2}{\theta-1} \log\left( \frac{\alpha_0(\theta-1)}{\delta} + 1 \right)  \right\rbrace
% 			\]
% 			implies that
% 			$k \ge T(\theta;\beta)$;
			
% 			2) if $\theta \ge 2$, then
% 			\[
% 			k \ge \min \left\lbrace \frac{\alpha_0}{\delta}, \log\left( \frac{\alpha_0}{\delta} + 1 \right) / \log 2 \right\rbrace 
% 			\]

% 	---------
	
	Now we are ready to present the main result of the subsection.
	\begin{proposition}\label{prop:sub-new}
		Let $(M_f,\mu) \in \R_+ \times \R_+ $ 
		% 		satisfying $M_f^2/(\mu \bar \varepsilon)\ge 1/4$ 
		and instance $(x_0,(f,f';h))$ satisfying
		conditions (A1)-(A3)  be given.
		% 		Let $\{x_k\}$ denote the sequence of
		% 		iterates generated by 
		% 		the CS-CS method in \eqref{eq:sub} with  $ \lam\le \bar \varepsilon/(4M_f^2) $
		% 		and initial point $x_0$.
		Then, the number of iterations performed by CS-CS$(x_0,\lam)$
		with $\lam\le \bar \varepsilon/(4M_f^2) $
		until it finds
		a $ \bar \varepsilon $-solution is bounded by 
		\[
		\left \lfloor \min \left\lbrace \frac{d_0^2}{\lam \bar \varepsilon}, \frac{1+\lam \mu}{\lam \mu} \log \left( \frac{\mu d_0^2}{\bar \varepsilon} + 1 \right) \right\rbrace \right \rfloor +1.
		\]
		% 		\[
		% 		\left \lfloor \min \left\lbrace \frac{d_0^2}{\lam \bar \varepsilon}, 2\max \left\lbrace\frac{1}{\lam \mu},1\right\rbrace \ln\left( \frac{\mu d_0^2}{\bar \varepsilon} + 1 \right) \right\rbrace \right \rfloor +1.
		% 		\]
	\end{proposition}
	\begin{proof}
		% 		We first prove a general result holds for both convex and $ \mu $-strongly convex cases.
		Recall that an iteration of CS-CS$(x_0,\lam)$ is as in \eqref{eq:sub}.
		Using the fact that the objective function in \eqref{eq:sub} is $ (\mu+1/\lam) $-strongly convex
		and Theorem 5.25(b) of \cite{beck2017first},
		we conclude that for every $ j\ge 1 $ and $ u \in \dom h $,
		\begin{equation}\label{ineq:basic1}
			\ell_f(x_j;x_{j-1})  + h(x_j)+ \frac{1}{2\lam} \|x_j-x_{j-1}\|^2 + \frac{1}{2}\left( \mu + \frac{1}{\lam}\right)  \|u-x_j\|^2 
			\le \ell_f(u;x_{j-1})  + h(u) + \frac{1}{2\lam} \|u-x_{j-1}\|^2
		\end{equation}
		where $ \ell_f(u;v) := f(v)+\inner{f'(v)}{u-v} $ for every $ u,v \in \dom h $.
		Noting that
		(A4), \eqref{ineq:func}, the definition of $ \ell_f $, the triangle inequality, and the Cauchy-Schwarz inequality, imply that
		\[
		f(x_j) - \ell_f(x_j;x_{j-1}) \le |f(x_j)-f(x_{j-1})| + \|f'(x_{j-1})\| \|x_j-x_{j-1}\| \le 2M_f\|x_j-x_{j-1}\|,
		\]
		and using the definition of
		$\phi$ in \eqref{eq:ProbIntro}, and
		the fact that 
		$ \ell_f(\cdot;v)\le f(\cdot) $ for every $ v \in \dom h$, we then conclude from
		\eqref{ineq:basic1} with $u=x_0^*$ that
		% 		Rearranging the terms in \eqref{ineq:basic1}, using the definition of $ \ell_f $ and the aforementioned inequality of $ \ell_f $ with $ (u,v) = (u,x_{j-1})$, we obtain
		\begin{align*}
			\frac{1}{2}\left( \mu + \frac{1}{\lam}\right) & \|x_0^*-x_j\|^2 + \phi(x_j) - \phi^* 
			\le \frac{1}{2\lam} \|x_0^*-x_{j-1}\|^2  +f(x_j) - \ell_f(x_j;x_{j-1}) - \frac{1}{2\lam} \|x_j-x_{j-1}\|^2 \\
			&\le \frac{1}{2\lam} \|x_0^*-x_{j-1}\|^2  + 2M_f\|x_j-x_{j-1}\| - \frac{1}{2\lam} \|x_j-x_{j-1}\|^2 \le \frac{1}{2\lam} \|x_0^*-x_{j-1}\|^2 + 2\lam M_f^2 
			% 			\label{ineq:strong}
		\end{align*}
		% 		Using (A3), the definition of $ \ell_f $, the triangle inequality and the Cauchy-Schwarz inequality, we have for every $ j\ge 1 $,
		% 		\[
		% 		f(x_j) - \ell_f(x_j;x_{j-1}) \le |f(x_j)-f(x_{j-1})| + \|f'(x_{j-1})\| \|x_j-x_{j-1}\| \le 2M_f\|x_j-x_{j-1}\|.
		% 		\]
		% 		It follows from the above two inequalities that for every $ j\ge 1 $, 
		% 		\begin{align}
		% 		\frac{1}{2}\left( \mu + \frac{1}{\lam}\right)  \|u-x_j\|^2 + \phi(x_j) - \phi(u) &\le \frac{1}{2\lam} \|u-x_{j-1}\|^2  + 2M_f\|x_j-x_{j-1}\| - \frac{1}{2\lam} \|x_j-x_{j-1}\|^2 \nn \\
		% 		&\le \frac{1}{2\lam} \|u-x_{j-1}\|^2 + 2\lam M_f^2 \label{ineq:strong}
		% 		\end{align}
		where the last inequality is due to the fact that $ a^2+b^2\ge 2ab $ for every $ a,b \in R $.
		Since the above inequality satisfies \eqref{eq:easyrecur} with $ \eta_j=\phi(x_j) - \phi^*$, $\alpha_j=\|x_j-x_0^*\|^2/(2\lam)$, $\theta=1+\lam \mu$, and $ \delta= \bar \varepsilon/2 $ in view of the assumption that $ \lam \le \bar \varepsilon/(4M_f^2) $, it follows from Lemma \ref{lm:easyrecur}(a) and the fact that $ \alpha_0=d_0^2/(2\lam) $ that
		%		for every $ k\ge 1 $,
		%		\begin{equation}\label{ineq:unify}
		%			\min_{1\le j \le k} \phi(x_j)-\phi^* \le \frac{d_0^2}{2\lam \sum_{j=1}^{k} (1+\lam \mu)^{j-1}} + 2\lam M_f^2 \le \frac{d_0^2}{2\lam \sum_{j=1}^{k} (1+\lam \mu)^{j-1}} + \frac{\bar \varepsilon}{2},
		%		\end{equation}
		%		where the last inequality is due to
		%		the assumption that $ \lam \le \bar \varepsilon/(4M_f^2) $.
		%		Following an argument similar to the one
		%		used in the proof of Theorem \ref{thm:suboptimal}(a), it is now easy to
		%		see that
		$ \min_{1\le j \le k} \phi(x_j)-\phi^*\le \bar \varepsilon $ for 
		every index $k\ge 1$ such that %\eqref{ineq:k},
		\[
		k\ge \min \left\lbrace \frac{d_0^2}{\lam \bar \varepsilon}, \frac{1+\lam \mu}{\lam \mu}\log\left( \frac{\mu d_0^2}{\bar \varepsilon} + 1\right)\right\rbrace,
		\]
		% 		\[
		% 		\left \lfloor \min \left\lbrace \frac{d_0^2}{\lam \bar \varepsilon}, 2\max \left\lbrace\frac{1}{\lam \mu},1\right\rbrace \ln\left( \frac{\mu d_0^2}{\bar \varepsilon} + 1 \right) \right\rbrace \right \rfloor +1.
		% 		\]
		and hence that the conclusion of the lemma holds.
		% 		in view of the assumptions that  $ \lam \le \bar \varepsilon/(4M_f^2) $ and $M_f^2/(\mu \bar \varepsilon)\ge 1/4$.
	\end{proof}

	\section{Proof of Theorem \ref{thm:lb-cvx}}\label{sec:lb}
	We start by presenting two technical lemmas, which are the starting points of the lower complexity bound analysis.

	\begin{lemma}\label{lem:p}
		For every $ R>0 $, the function $ p_R:\R^n \to \R $ defined as
		\begin{equation}\label{def:p}
			p_R(x) := \left\{ \begin{array}{cc}
				\frac{1}{2}\|x\|^2 &  \mbox{if $ \|x\| \le R$;} \\
				R(\|x\|  - \frac{R}{2})    
				& \mbox{otherwise,}
			\end{array}
			\right.
		\end{equation}
		is a convex differentiable function whose gradient is bounded by $ R $
		everywhere on $\R^n$.
		% 			the following
		% 			statements hold:
		% 			\begin{itemize}
		% 				\item[a)]
		% 				$ p $ is a convex
		% 				differentiable function in $\R^n$, and its gradient is bounded by $ R $;
		% 				\item[b)]
		% 				$ f $ is a convex function;
		% 				\item[c)]
		% 				consider $ f': \dom h \to \R^n $ defined as $ f'(x)=\gamma e_{i^*} + \tau \nabla p(x) $ where $ e_{i^*} $ is the $ i^* $-th coordinate vector and $ i^* $ corresponds to the first maximal component of $ x^{(1)}, \ldots, x^{(k_0+1)} $, then
		% 				\begin{equation}\label{eq:g}
		% 				f'(x)\in \partial f(x), \quad \|f'(x)\| \le \gamma+\tau R \quad \forall x \in \R^n.
		% 				\end{equation}
		% %				
		% %				is convex and $\partial f(x) \subset \cball{0}{\gamma + \tau R}$ for every $x \in  \R^n$, where $ \cball{x}{R} $ is a closed ball $ \subset \R^n $ centered at $ x $ with radius $ R>0 $.				
		% %				 \item[c)]
		% %				 the minimization problem $ \min\{(f+h)(x):x\in \R^n\} $ has a global minimum $x^*$ satisfying
		% %				\[
		% %				\|x^*\|=\frac{\gamma}{(\tau+\mu) \sqrt{k}}, \quad (f+h)(x^*)=-\frac{\gamma^2}{2(\tau+\mu) k};
		% %				\]
		% %				moreover, $x^*$ is the unique global minimum of $f+h$ when $\mu>0$.
		% 			\end{itemize}
	\end{lemma}
	\begin{proof}
		Using the fact that the  function $ q:\R_+ \to \R $ defined as
		\[
		q(t) := \left\{ \begin{array}{cc}
			\frac{1}{2}t^2 &  \mbox{if $ t \le R$;} \\
			R(t  - \frac{R}{2})    
			& \mbox{otherwise}
		\end{array}
		\right.
		\]
		is increasing and convex,
		% 			whose derivative is 
		% 			\[
		% 			q'(t) = \left\{ \begin{array}{cc}
		% 			t &  \mbox{if $ t \le R$;} \\
		% 			R 
		% 			& \mbox{otherwise.}
		% 			\end{array}
		% 			\right.
		% 			\]
		% 			It is easy to see that $ q $ is an increasing function and
		% 			\[
		% 			q(t)\ge q(t') + q'(t')(t-t') \quad \forall t,t'\in R_+,
		% 			\]
		% 			and hence that $ q(t) $ is convex.\red{remove}
		and $ p_R(x)=q(\|x\|) $ for every $x \in \R^n$, it follows from
		% 			the fact $ q $ is an increasing and convex function,
		Proposition 2.1.8 in \cite{urruty1996convex1} that $ p_R $ is a convex function.
		Moreover, it is easy to see $ p_R $ is differentiable everywhere and its gradient is
		\begin{equation}\label{eq:grad p}
			\nabla p_R(x)= \left\{ \begin{array}{cc}
				x &  \mbox{if $\|x\| \le R$;} \\
				R  \frac{x}{\|x\| }
				& \mbox{otherwise,}
			\end{array}
			\right.
		\end{equation}
		and hence that $ \|\nabla p_R(x)\| \le R $.
	\end{proof}
	
	The following lemma plays an important role in our lower complexity bound analysis, since it constructs a worst-case instance $(x_0,(f,f';h))$ in the class $ {\cal I}_\mu^u(M_f;R_0)$ and provides several properties of the instance.
	
	\begin{lemma}\label{lem:opt}
		For any $ R>0 $, $\gamma\ge0$, $\tau \ge 0$, $k_0 \in \{1,\ldots,n\}$ and $ \mu \ge 0 $, consider $ x_0=0 $, and the functions $ f, h: \R^n \to \R$ and $f': \R^n \to \R^n $
		defined as
		\begin{align}
			f(x) &= f_{R,\gamma,\tau,k_0}(x):= \gamma \max_{1\le i \le k_0} x^{(i)} + \tau p_R(x), \label{def:f}\\
			h(x) &=h_\mu(x):=\frac{\mu}2 \|x\|^2, \quad
			f'(x):=\gamma e_{i^*} + \tau \nabla p_R(x) \label{def:h}
		\end{align}
		where $p_R(\cdot)$ is as in \eqref{def:p}, $ e_{i} $ denotes the $i$-th coordinate vector, and $ i^* $ is the
		smallest index $i \in I_{k_0}(x):= \Argmax\{ x^{(i)} : i =1,\ldots,k_0\}$.
		Then, the following statements hold:
		\begin{itemize}
			\item[a)]
			$ f $ is a convex function and for every $x \in \R^n$
			\begin{equation}\label{eq:g}
				f'(x)\in \partial f(x), \quad \|f'(x)\| \le \gamma+\tau R;
			\end{equation}
			\item[b)] the minimization problem $ \min\{(f+h)(x):x\in \R^n\} $ has a global minimum $x^*$ satisfying
			\[
			\|x^*\|=\frac{\gamma}{(\tau+\mu) \sqrt{k_0}}, \quad (f+h)(x^*)=-\frac{\gamma^2}{2(\tau+\mu) k_0};
			\]
			moreover, if $\mu>0$ then
			$x^*$ is the only global minimum of the above problem;
			\item[c)] 
			the instance $(x_0,(f,f';h))$ is in $ {\cal I}_\mu^u(M_f;R_0)$
			for any $ (M_f,R_0) \in \R_+ \times \R_{++} $ such that $M_f \ge \gamma+\tau R$ and $ R_0\ge d_0 $;
			\item[d)] $ (f+h)(x)\ge 0 $ for every $ x\in \R^{k_0,n} :=\left\lbrace x\in \R^n: x^{(i)}=0, \ i= k_0,\ldots, n \right\rbrace $;
			\item[e)] if $k < k_0$ and $x \in \R^{k,n}$,
			then $f'(x) \in \R^{k+1,n}$.
		\end{itemize}
	\end{lemma}
	\begin{proof}
		a) %We first show that $f$ is convex. 
		Noting that $ x^{(i)}=e_i^T x $, and using the definition of $ f $ in \eqref{def:f}, Lemma \ref{lem:p},  Proposition 2.1.2 in \cite{urruty1996convex1}, and the facts that $\gamma\ge 0$ and $\tau\ge 0$, we have $ f $ is convex.
		Moreover, it follows from \eqref{def:f} and Lemma \ref{lem:p} that 
		%		\begin{equation}\label{eq:f}
		%			f(x)=\gamma \max_{1\le i \le k_0} e_i^\top x + \tau p_R(x), \quad
		$ \partial f(x)=\gamma {\rm conv} \{e_i: i \in I_{k_0}(x)\} + \tau \nabla p_R(x) $,
		%		\end{equation}
		which together with the definition of $ f' $ in \eqref{def:h} implies that the inclusion in \eqref{eq:g} holds.
		Finally, using the definition of $ f' $, the triangle inequality and Lemma \ref{lem:p}, we conclude that the inequality in \eqref{eq:g} holds.
		%		and hence that $ f $ is convex in view of Proposition 2.1.2 in \cite{urruty1996convex1} and the facts that $\gamma\ge 0$ and $\tau\ge 0$. 
		%		Now we prove that the conclusion about $f'$ in \eqref{eq:g} holds.
		%		It immediately follows from the definition of $ f' $ in \eqref{def:h} and the second identity in \eqref{eq:f} that the inclusion in \eqref{eq:g} holds. Using the definition of $ f' $, the triangle inequality and Lemma \ref{lem:p}, we conclude that the inequality in \eqref{eq:g} holds.
		
		b) The first statement can be analogously proved by following a similar argument as in P196 of \cite{nesterov2018lectures}.
		The second statement immediately follows from the fact $ f+h $ is $ \mu $-strongly convex when $ \mu>0 $.
		
		c) It follows from the assumptions in the statement, the definition of $ h $ in \eqref{def:h}, and statements a) and b) that $ (x_0,(f,f';h)) $ satisfies $ d_0\le R_0 $ and (A1)-(A3) with $ M_f\ge \gamma+\tau R $. Hence, the statement holds.
		
		d) Using the definitions of $ p_R(x)$ and $h(x) $ in \eqref{def:p} and \eqref{def:h}, respectively, we have $ p_R(x)\ge 0$ and $ h(x)\ge 0 $ for every $ x \in \R^n $.
		This conclusion and the definition of $ f $ in \eqref{def:f} imply that for $ x\in \R^{k_0,n} $,
		\[
		(f+h)(x)\ge \gamma \max_{1\le i \le k_0} x^{(i)} \ge \gamma x^{(k_0)} = 0.
		\]
		
		e) Using the assumption that $x\in \R^{k,n}$, \eqref{eq:grad p} and the definition of $ i^* $ in the line below \eqref{def:h}, we have $ \nabla p_R(x) \in \R^{k,n} $ and $ i^* \le k $. It now follows from the definition of $ f' $ in \eqref{def:h} that $ f'(x) \in \R^{k+1,n} $.
	\end{proof}

	\vgap
	
	Now we are ready to prove Theorem \ref{thm:lb-cvx}.
	
	\vgap
	
	\noindent
	\textbf{Proof of Theorem \ref{thm:lb-cvx}}
		First note that the
		last claim of the theorem follows immediately from
		the claim above it
		and the definition of
		$\bar \varepsilon$-lower complexity bound (see the paragraph following \eqref{incl:xk}).
		We now show that,
		for an arbitrary 
		quadruple $ (M_f,\mu,R_0, \bar \varepsilon) \in \R_{+} \times \R_{+} \times \R_{++} \times  \R_{++}$,
		there exists
		an instance $(x_0,(f,f';h))$ satisfying a) and b).
		The proof considers the following two cases separately:
		\begin{itemize}
			\item[a)] $\mu R_0^2\le 8 \bar \varepsilon$;
			\item[b)] $\mu R_0^2\ge 8 \bar \varepsilon$.
		\end{itemize}
		{\bf Proof of case a):} Assume that
		condition a) is satisfied.
		The proof under this condition in turn considers the following
		two subcases separately:
		a1) $ M_f R_0/\bar \varepsilon < 8 $, and a2) $ M_f R_0/\bar \varepsilon \ge 8 $.
		
		For case a1), choose the dimension $n \ge 1$ arbitrarily, and consider the instance $ (x_0,(f,f';h)) $ as in Lemma \ref{lem:opt} with $(R, k_0, \gamma, \tau)=(R_0, n, 0,M_f/R_0 )$. %Clearly, the first part of a) holds. 
		Lemma \ref{lem:opt}(b) and the facts that $x_0=0$ and $\gamma=0$
		imply that $ x^*=0$ and $ d_0=\|x^*-x_0\|=0$, and
		hence that $d_0 \le R_0$
		and
		$M_f = \gamma + \tau R $
		due to the above definitions of $R$, $\gamma$ and $\tau$.
		Clearly, a) now follows from
		Lemma \ref{lem:opt}(c). 
		Note that $ M_f R_0/\bar \varepsilon<8 $ implies that \eqref{cmplx:min} reduces to 1. %$  \Omega(1) $.
		Since any algorithm has to perform at least one iteration, it follows that the instance $ (x_0,(f,f';h)) $ satisfies b).
		%	Clearly, case i) violates the assumption in c), and hence c) holds.
		
		For case a2), consider the instance $ (x_0,(f,f';h)) $ as in Lemma \ref{lem:opt} with dimension $n$ such that
		$n \ge k_0$ and
		$(R, k_0, \gamma,\tau)$
		defined as
		\begin{equation}\label{def:k0}
			R=R_0, \quad k_0= \left\lfloor \frac{M_f^2 R_0^2}{64 \bar \varepsilon^2} \right\rfloor, \quad \gamma=\frac{\sqrt{k_0}}{1+\sqrt{k_0}}(M_f+\mu R_0), \quad \tau=\frac{1}{1+\sqrt{k_0}}\left(\frac{M_f}{R_0} - \mu \sqrt{k_0} \right).
		\end{equation}
		% 	and with
		% 	\begin{equation}\label{eq:gamma-mu}
		% 		\gamma:=\frac{\sqrt{k_0}}{1+\sqrt{k_0}}(M_f+\mu R_0), \qquad \tau:=\frac{1}{1+\sqrt{k_0}}\left(\frac{M_f}{R_0} - \mu \sqrt{k_0} \right).
		% 	\end{equation}
		Using Lemma \ref{lem:opt}(b), \eqref{def:k0}, and the fact that $ x_0=0 $,
		it is easy to see that
		\begin{equation}\label{eq:relation}
			d_0 \le \|x^*-x_0\| = \|x^*\|
			%=\frac{\gamma}{(\tau+\mu) \sqrt{k_0+1}}
			=R_0, \qquad (f+h)(x^*)=%-\frac{\gamma^2}{2(\tau+\mu) (k_0+1)}=
			-\frac{(M_f+\mu R_0) R_0}{2(1+\sqrt{k_0})} \le -\frac{M_f R_0}{2(1+\sqrt{k_0})}
		\end{equation}
		where $ x^* $ is as in Lemma \ref{lem:opt}(b).
		Moreover, using the definitions of $ k_0 $ and $ \tau $ in \eqref{def:k0}, the assumption that $ \mu R_0^2 \le 8 \bar \varepsilon $, and the fact that $ x\ge \lfloor x \rfloor $ for every $ x\in \R $, we have
		\begin{equation}\label{eq:tau}
			\tau 
			= \frac{1}{1+\sqrt{k_0}}\left(\frac{M_f}{R_0} - \mu \left \lfloor \frac{M_f^2 R_0^2}{64 \bar \varepsilon^2} \right \rfloor ^{1/2} \right)
			\ge \frac{1}{1+\sqrt{k_0}} \frac{M_f}{R_0} \left( 1 - \frac{\mu R_0^2}{8\bar \varepsilon} \right)\ge 0.
		\end{equation}
		We next show that
		the above instance satisfies a) and b).
		Indeed, %the first part of a) is obvious and
		a) follows from
		Lemma \ref{lem:opt}(c) by noting
		that all the assumptions required
		by it follow from
		\eqref{def:k0}, \eqref{eq:relation}
		and \eqref{eq:tau}.
		We next show b). In view of the definition of $k_0$ in \eqref{def:k0}, it suffices to show that
		% 		a)  Since $ h=\mu\|\cdot\|^2/2 $, in order to show a), we need to show that $ (x_0,(f,f';h)) $ is in the class $ {\cal I}_\mu(M_f;R_0) $. Indeed, in view of
		% 		Lemma \ref{lem:opt}(c) and the obvious facts that $ R=R_0>0 $, $k_0\le n$, $\gamma\ge 0$ and $ d_0\le R_0 $, it suffices to
		% 		shows that $ M_f\ge \gamma+\tau R_0 $ and $\tau \ge 0$.
		% 		The first inequality follows immediately from
		% 		the definitions of $ \gamma $ and $\tau$ in \eqref{def:k0}.
		% 		The definitions of $ k_0 $ and $ \tau $ in \eqref{def:k0}, the assumption that $ \mu R_0^2 \le 8 \bar \varepsilon $, and the fact that $ x\ge \lfloor x \rfloor $ for every $ x\in \R $, imply that
		% 		\begin{equation}\label{eq:tau}
		% 		    \tau 
		% 		= \frac{1}{1+\sqrt{k_0}}\left(\frac{M_f}{R_0} - \mu \left \lfloor \frac{M_f^2 R_0^2}{64 \bar \varepsilon^2} \right \rfloor ^{1/2} \right)
		% 		\ge \frac{1}{1+\sqrt{k_0}} \frac{M_f}{R_0} \left( 1 - \frac{\mu R_0^2}{8\bar \varepsilon} \right)\ge 0.
		% 		\end{equation}
		% 		We then conclude that $ (x_0,(f,f';h)) \in {\cal I}_\mu(M_f;R_0)$.
		% 		
		% 		b) We will now show that $ (x_0,(f,f';h))$ has $ \bar \varepsilon $-lower complexity bound given by \eqref{cmplx:min}.
		%		Note that $ M_f R_0/\bar \varepsilon<8 $ in case i), and hence \eqref{cmplx:min} reduces to $ \Omega(1) $.
		%		Since any algorithm has to perform at least one iteration, it follows that the instance $ (x_0,(f,f';h)) $ satisfies b) in case i).
		%		We will consider case ii) from now on.
		% 		Note that,
		% 		in view of \eqref{cmplx:min} and \eqref{def:k0}, it suffices to show that
		the number of iterations performed by any algorithm ${\cal A}$ in ${\cal A}({\cal I}_\mu^u(M_f;R_0),\bar \varepsilon)$  %to find a $ \bar \varepsilon $-solution
		is at least $ k_0$.
		Indeed, first note that
		the assumption of case ii) imply that $ k_0\ge1 $ which, together
		with
		the definition of $ k_0 $ in \eqref{def:k0}, implies that
		\begin{equation}\label{ineq:k0}
			1+\sqrt{k_0} \le 2\sqrt{k_0} 
			= 2\left \lfloor \frac{M_f^2R_0^2}{64 \bar \varepsilon^2} \right \rfloor^{1/2}
			\le \frac{M_f R_0}{4 \bar \varepsilon}.
		\end{equation}
		Moreover, if $\{x_k\}$ is a sequence generated by ${\cal A}$, then it follows from
		the fact that $ x_0=0 $,
		condition \eqref{incl:xk}, Lemma \ref{lem:opt}(e), and a
		straightforward induction
		argument, that
		% 		arguments similar to the ones following equation (3.2.6) of \cite{nesterov2018lectures}, we conclude that
		$ x_k \in \R^{k+1,n} \subset \R^{k_0,n}$
		for every $ k \le k_0-1$. Hence, it
		follows from Lemma \ref{lem:opt}(d) that $ (f+h)(x_k)\ge 0 $ for every $ k\le k_0-1 $. This conclusion, the second
		relation in \eqref{eq:relation}, and \eqref{ineq:k0}, then imply that
		% 		\begin{equation}\label{ineq:f}
		\[
		(f+h)(x_k)-(f+h)(x^*) \ge -(f+h)(x^*) \ge \frac{M_f R_0}{2(1+\sqrt{k_0})} \ge 2 \bar \varepsilon \qquad \forall k \le k_0-1,
		\]
		% 		\end{equation}
		and hence that
		the number of iterations of ${\cal A}$ is at least $ k_0 $.
		%		Finally, c) follows from \eqref{ineq:f} with $ k=0 $.
		
		\noindent
		{\bf Proof of case b):} Assume that
		condition b) is satisfied.
		The proof under this condition in turn considers the following
		two subcases separately:
		b1) $ M_f^2/(\mu \bar \varepsilon) < 8 $, and b2) $M_f^2/(\mu \bar \varepsilon) \ge 8 $.
		
		For case b1), consider 
		the instance $ (x_0,(f,f';h)) $ such that $ x_0=0 $, $ f=0 $, $ f'=0 $ and $ h=\mu\|\cdot\|^2/2 $. 
		%		Clearly, the first part of a) holds.
		It is easy to see that $ x^*=0 $ and $ d_0=\|x^*-x_0\|=0$. 
		Note that $ (x_0,(f,f';h)) $ clearly satisfies (A1)-(A3), $ d_0\le R_0 $ and $ h=\mu\|\cdot\|^2/2 $, and hence that a) holds.
		Note that $ M_f^2/(\mu \bar \varepsilon)<8 $ implies that \eqref{cmplx:min} reduces to 1. % $ \Omega(1) $.
		Since any algorithm has to perform at least one iteration, it follows that the instance $ (x_0,(f,f';h)) $ satisfies b).
		
		For case b2),
		consider the instance $ (x_0,(f,f';h)) $ as in Lemma \ref{lem:opt} with dimension $n$ such that
		$n \ge k_0$ and
		$(R, k_0, \gamma,\tau)$
		defined as
		\begin{equation}\label{eq:k0}
			R=R_0, \qquad k_0= \left \lfloor \frac{M_f^2}{4\mu \bar \varepsilon} \right \rfloor, \qquad \gamma=M_f, \qquad \tau=0.
		\end{equation}
		Using Lemma \ref{lem:opt}(b), \eqref{eq:k0}, and the fact that $ x_0=0 $,
		it is easy to see that
		\begin{equation}\label{rel1}
			d_0\le \|x_0-x^*\|=\|x^*\|=\frac{M_f}{\mu \sqrt{k_0}}, \qquad (f+h)(x^*)=-\frac{M_f^2}{2\mu k_0}
		\end{equation}
		where $ x^* $ is as in Lemma \ref{lem:opt}(b).
		Moreover, it follows from the facts that $ M_f^2/(\mu \bar \varepsilon) \ge 8$ and $ \lfloor x \rfloor\ge x-1 $ for every $ x\in \R $ that $ \lfloor M_f^2/(4\mu \bar \varepsilon) \rfloor \ge M_f^2/(8\mu \bar \varepsilon) $.
		This inequality, the first relation in \eqref{rel1}, the definition of $ k_0 $ in \eqref{eq:k0}, 
		and the assumption that $ \mu R_0^2\ge 8 \bar \varepsilon $, imply that
		%		\[
		%		\frac{M_f^2}{\mu^2 d_0^2} = k_0 =\left \lfloor \frac{M_f^2}{4\mu \bar \varepsilon} \right \rfloor
		%		\ge \frac{M_f^2}{4\mu \bar \varepsilon} - 1 \ge \frac{M_f^2}{8\mu \bar \varepsilon} \ge \frac{M_f^2}{\mu^2 R_0^2},
		%		\]
		\begin{equation}\label{ineq:d0-R0}
			d_0%=\frac{M_f}{\mu \sqrt{k_0}}
			\le\frac{M_f}{\mu} \left \lfloor \frac{M_f^2}{4\mu \bar \varepsilon} \right \rfloor^{-\frac12} 
			%		\le \frac{M_f}{\mu}\left( \frac{M_f^2}{4\mu \bar \varepsilon} - 1\right) ^{-1/2}
			\le \frac{M_f}{\mu}\left( \frac{M_f^2}{8\mu \bar \varepsilon} \right) ^{-\frac12}
			=\left( \frac{8\bar \varepsilon}{\mu}\right)^{\frac12}
			\le R_0.
		\end{equation}
		We next show that
		the above instance satisfies a) and b).
		Indeed, %the first part of a) is obvious and
		a) follows from
		Lemma \ref{lem:opt}(c) by noting
		that all the assumptions required
		by it follow from
		\eqref{eq:k0} and \eqref{ineq:d0-R0}.
		We next show b).
		In view of the definition of $ k_0 $ in \eqref{eq:k0}, it suffices to show that the number of iterations performed by any algorithm $ {\cal A} $ in ${\cal A}({\cal I}_\mu^u(M_f;R_0),\bar \varepsilon)$  is at least $ k_0$.
		Assume then that $\{x_k\}$ is a
		sequence generated by ${\cal A}$.
		As in the proof of Theorem \ref{thm:lb-cvx}, we have
		$ (f+h)(x_k)\ge 0 $ for every $ k\le k_0-1 $. Hence, using
		the definition of $ k_0 $ in \eqref{eq:k0}
		% 		\[
		% 		k_0 =\left \lfloor \frac{M_f^2}{4\mu \bar \varepsilon} \right \rfloor < \left \lfloor \frac{M_f^2}{2\mu \bar \varepsilon} \right \rfloor \le \frac{M_f^2}{2\mu \bar \varepsilon}.
		% 		\]
		% 		Using the above inequality, the fact that $ (f+h)(x_k)\ge 0 $ for every $ k\le k_0-1 $ (see the sentence following \eqref{ineq:k0}), 
		and the second relation in \eqref{rel1}, we conclude that
		% 		\begin{equation}\label{ineq:x0-nonopt}
		\[
		(f+h)(x_k) - (f+h)(x^*) \ge -(f+h)(x^*)=\frac{M_f^2}{2\mu k_0}\ge 2\bar \varepsilon  \qquad \forall k \le k_0-1,
		\]
		% 		\end{equation}
		and hence that the number of iterations of ${\cal A}$ is at least $ k_0 $.
	\QEDA

	\section{Proof of Theorem \ref{thm:opt2}} \label{sec:pf-opt}
	
	We start by stating a technical but simple lemma about RPB$(x_0,\lam,\bar \varepsilon/2)$.
	
	\begin{lemma}\label{lem:RPB}
		For any $\bar \varepsilon, \lam >0$ and $(M_f,\mu,R_0) \in \R_+^3$,
		%		such that the instance class ${\cal I}_\mu(M_f,M_h;R_0,\bar \varepsilon)$ is nonempty.
		RPB$(x_0,\lam,\bar \varepsilon/2)$ is in ${\cal A}({\cal I}_\mu^u(M_f,R_0), \bar \varepsilon)$.
	\end{lemma}
	\begin{proof}
		To simplify notation within this proof, denote
		${\cal I}_\mu^u(M_f,R_0)$ simply by ${\cal I}_\mu^u$.
		Our goal is to show that RPB satisfies properties a) and
		b) in the definition (see the paragraph containing \eqref{incl:xk})
		of ${\cal A}({\cal I}_\mu^u,\bar \varepsilon)$. Indeed, a) follows from Theorem \ref{thm:suboptimal}(c).
		In order to show property b), assume that
		there exists $\alpha \ge \mu$ such that 
		$\nabla^2 h (x) = \alpha I$ for every $x \in \R^n$.
		Note first that the optimality condition of \eqref{def:xj}, the above assumption on $h$,
		and the facts that $ x_{j-1}^c=x_{\ell_0} $ and $ \partial f_j(x_j)={\rm conv} \{f'(x): x\in A_j \} $ (see Corollary 4.3.2 of \cite{urruty1996convex1}),
		imply that  for any two consecutive serious iteration indices
		$ \ell_0 $ and $ \ell_1 $ and any index $j$ such that
		$ \ell_0 < j \le \ell_1 $,
		\begin{align*}
			0 &\in \partial f_j(x_j) + \nabla h(x_j) + \frac{1}{\lam} (x_j-x_{\ell_0}) = {\rm conv} \{f'(x): x\in A_j \} + \nabla h(x_j) + \frac{1}{\lam} (x_j-x_{\ell_0}) \\
			&= {\rm conv} \{f'(x): x\in A_j \}+ \nabla h(x_{\ell_0}) + \left(\frac{1}{\lam} +\alpha \right)(x_j-x_{\ell_0}),
		\end{align*}
		and hence that \eqref{incl:xk} holds with $x_0$ replaced by $x_{\ell_0}$.
		Using this inclusion and a simple induction argument, it is easy to see that \eqref{incl:xk} holds for every $j\ge 1$, and hence that property b) holds.
	\end{proof}
	
	\vgap
	
	We are now ready to present the proof of Theorem \ref{thm:opt2}.
	
	\vgap
	
	\noindent
	{\bf Proof of Theorem \ref{thm:opt2}}
		For shortness, 
		RPB$ (x_0, \lam, \bar \varepsilon/2) $ is referred below to as RPB.
		
		a)
		In view of Theorem \ref{thm:lb-cvx}, a) will follow
		from the claim that
		\eqref{cmplx:bound} is a
		$\bar \varepsilon$-upper complexity bound
		for RPB
		with respect to the instance class $ {\cal I}_\mu(M_f,R_0)$.
		To show the latter claim, first note that RPB is in $ {\cal A}({\cal I}_\mu(M_f,R_0),\bar \varepsilon) $ in view of Lemma \ref{lem:RPB}.
		It follows from Corollary \ref{thm:bound-strg} and the fact $ d_0\le R_0 $, we conclude that
		\eqref{cmplx:bound}
		is a $ \bar \varepsilon $-upper complexity bound for RPB with respect to $ {\cal I}_\mu(M_f,R_0)$,
		and hence that the aforementioned claim holds.
		
		b) 
		In view of Theorem \ref{thm:lb-cvx}, b) will follow
		from the claim that
		\eqref{cmplx:bound} with $ \mu=0 $ is a
		$\bar \varepsilon$-upper complexity bound
		for RPB
		with respect to the instance class $ {\cal I}_0(M_f,R_0;C)$ (and hence to any instance class ${\cal I}$ satisfying \eqref{incl:I1}).
		To show the latter claim, first note that RPB is in $ {\cal A}({\cal I}_0(M_f,R_0;C),\bar \varepsilon) $ in view of Lemma \ref{lem:RPB} with $\mu=0$.
		Moreover, since
		the second inclusion of \eqref{incl:I1} and the definition of $ {\cal I}_0(M_f,R_0;C) $ in \eqref{def:tI} imply that
		$ M_h\le C M_f$ and $ d_0\le R_0 $, it follows from
		Corollary \ref{thm:bound-cvx} that
		$ {\cal O}_1(M_f^2 R_0^2/\bar \varepsilon^2) $ is a
		$ \bar \varepsilon $-upper complexity bound for RPB with respect to $ {\cal I}_0(M_f,R_0;C)$.
		Clearly,
		the previous bound is equal to \eqref{cmplx:bound} with $\mu=0$.
	\QEDA

\end{document}